\documentclass[]{amsart}
\usepackage{graphicx}
\usepackage{amsaddr}
\usepackage{amssymb}
\usepackage{amsmath}
\usepackage{amsthm}
\usepackage{bm}
\usepackage{hyperref}
\usepackage{dsfont}
\usepackage[ruled,vlined]{algorithm2e}
\usepackage{multirow}
\usepackage{booktabs}
\usepackage{orcidlink}
\usepackage{thmtools}
\usepackage{enumerate}
\usepackage[backend=biber,bibencoding=utf-8,citestyle=numeric-comp]{biblatex}
\usepackage{cleveref}

\hypersetup{hidelinks}

\newcommand{\pluseq}{\mathrel{+}=}
\newcommand{\asteq}{\mathrel{*}=}

\newcommand{\zeps}[3]{Z_{#1}(#2,#3)}

\renewcommand{\d}{\,\mathrm{d}}
\newcommand{\sums}{\sideset{}{'} \sum}

\AtBeginDocument{%
  }

\newcounter{mythm}
\crefalias{mythm}{section}
\numberwithin{mythm}{section}

\declaretheorem[style=plain, sibling=mythm]{theorem}
\declaretheorem[style=plain, sibling=mythm]{lemma}
\declaretheorem[style=plain, sibling=mythm]{corollary}
\declaretheorem[style=definition, sibling=mythm]{definition}

\declaretheorem[style=remark, sibling=mythm]{remark}

\declaretheorem[style=remark, numbered=no, name=Remark]{remark*}

\renewcommand{\Re}[1]{\mathrm{Re}(#1)}
\renewcommand{\Im}[1]{\mathrm{Im}(#1)}

\usepackage{amsfonts} 
\DeclareMathSymbol{\shortminus}{\mathbin}{AMSa}{"39}

\NewBibliographyString{refname}
\NewBibliographyString{refsname}
\DefineBibliographyStrings{english}{%
  refname = {ref\adddot},
  refsname = {refs\adddot}
}

\DeclareCiteCommand{\ccite}
  {%
  \ifnum\thecitetotal=1
    \bibstring{refname}%
  \else%
    \bibstring{refsname}%
  \fi%
  \addspace\bibopenbracket%
  \usebibmacro{cite:init}%
   \usebibmacro{prenote}}
  {\usebibmacro{citeindex}%
   \usebibmacro{cite:comp}}
  {}
  {\usebibmacro{cite:dump}%
   \usebibmacro{postnote}%
   \bibclosebracket}

\newrobustcmd*{\Ccite}{\bibsentence\cite}

\newcommand{\diffop}[1][ y]{\nabla_{\;\!\!\!\bm #1}}

\def\XXint#1#2#3{{\setbox0=\hbox{$#1{#2#3}{\int}$}
     \vcenter{\hbox{$#2#3$}}\kern-.5\wd0}}

\DeclareUnicodeCharacter{0306}{\u{}} 

\newcommand{\anisoreq}{%
Let $\Lambda$
be a $d$-dimensional
lattice, let $\bm\alpha\in\mathds N^d$, let $\bm x,\bm y\in\mathds R^d$ and let $\nu\in\mathds C$ such that $\nu\neq d+|\bm\alpha|$ if both $\bm y\in\Lambda^*$ and $\chi(\bm\alpha)=1$.
}
\newcommand{\anisoregreq}{%
Let $\Lambda$
be a $d$-dimensional
lattice, let $\bm\alpha\in\mathds N^d$, let $\bm x\in\mathds{R}^d$,
let
$\bm y \in (\mathds R^d\setminus\Lambda^*)\cup\{\bm 0\}$ and let $\nu\in\mathds C$.
}

\begin{document}

\title[High-order derivatives of generalized Epstein zeta functions]{Computation of anisotropic singular sums from high-order derivatives of Epstein zeta functions}

\author{Andreas A. Buchheit}
\address{Department of Mathematics, Saarland University, 66123 Saarbrücken, Germany\\Department of Mathematics, ETH Zürich, 8092 Zürich, Switzerland}

\author{Jonathan K. Busse}
\address{German Aerospace Center (DLR), Institute of Software Technology, High-Performance Computing Department, Linder Höhe, Cologne, 51147, Germany \\
Saarland University, Department of Mathematics, 66123 Saarbrücken, Germany}

\begin{abstract}
The precise and efficient evaluation of large-scale lattice sums involving power-law kernels is a fundamental computational problem in the simulation of classical and quantum systems with long-range interactions. While methods for spatially isotropic kernels, some based on Epstein zeta functions, have advanced considerably in recent years, the anisotropic case has lagged behind, despite its broad relevance to both fundamental and effective interactions such as the dipole interaction in magnetic materials.
In this work, we solve this issue by defining and analyzing anisotropic Epstein zeta functions for which we derive stably computable representations obtained from wave vector derivatives of lattice sums over isotropic interaction kernels. These functions find direct application in the analytical and numerical study of anisotropically interacting lattice systems. Going further, they provide the correction term in an exact equivalence between discrete lattices and their continuous analogs in the singular Euler-Maclaurin (SEM) expansion, a recent generalization of the classical Euler-Maclaurin summation formula to lattices and summands involving power-law kernels. Their connection to high-order derivatives of zeta functions can be used to improve convergence rates of numerical algorithms, for instance in micromagnetics, or to provide rapidly convergent expansions suitable for precomputations of generalized zeta functions. We derive a stably computable representation of anisotropic Epstein zeta functions, including the possibility for analytically removing Rayleigh--Wood singularities, and we develop a numerical algorithm for their stable evaluation for any lattice, power-law decay exponent and anisotropy order. We benchmark the algorithm against closed-form identities, direct summation and multi-precision results, obtaining machine precision across various lattices, power-law exponents, and anisotropy orders corresponding to derivatives up to order $60$.
An implementation is provided along with the article.
\end{abstract}

\maketitle
\section{Introduction}

Large-scale singular lattice sums whose summands decay as a power law of the distance play a key role in the simulation of classical and quantum long-range interacting matter \cite{campaPhysicsLongRangeInteracting2014,defenuLongrangeInteractingQuantum2023}.
Based on various applications in pure and applied mathematics, as well as in physics and chemistry, sums of this kind have been the object of intense mathematical interests, see i.e.~\cite{borwein2013lattice}.
In general, interaction kernels depend on both the norm and the direction of the lattice summation vector.
A prototypical anisotropic and homogeneous $d$-dimensional interaction kernel here reads
$$
V_{\nu,\bm\alpha}(\bm z)
= \frac{\bm z^{\bm\alpha}}{\vert \bm z \vert^\nu}
\,,\qquad 
\bm z\in\mathds R^d\setminus\{\bm 0\}\,,
$$
with
power-law decay exponent $\nu\in\mathds C$, an anisotropy multi-index $\bm\alpha$ with $\bm z^{\bm\alpha}=(z^{(1)})^{\alpha^{(1)}}\ldots (z^{(d)})^{\alpha^{(d)}}$,
and where we write $V_\nu=V_{\nu,\bm0}$.
Lattice sums over $V_{3}$ and $V_{5,\bm e_i+\bm e_j}$ classically arise due to dipolar interactions
\cite{cohenDipolarSumsPrimitive1955}.
Such dipolar couplings have recently been utilized to achieve entanglement in quantum
spin-exchange experiments
\cite{baoDipolarSpinexchangeEntanglement2023,hollandOndemandEntanglementMolecules2023}.
Higher-order multipole expansions likewise include anisotropic interaction kernels \cite{cumminsApplicationsEwaldMethod1976}, here as iterated gradients of $V_{1}$, linking anisotropy and differentiation.
Anisotropic interaction kernels appear effectively in the leading three-body terms in perturbative expansions of potential energies in theoretical chemistry \cite{buchheitEpsteinZetaMethod2025,robles-navarroExactLatticeSummations2025};
the leading term, the Axilrod-Teller-Muto (ATM) potential, contributes around $8\%$ of the total cohesive energy in solid argon \cite{schwerdtfeger2016towards} and up to $50\%$ of the binding energy in bilayer graphene \cite{anatole2010two}.
In topological superconductivity, chiral pairing induces anisotropic interaction kernels in winding vectors \cite{viyuelaChiralTopologicalSuperconductors2018}.
Lattice sums over $V_{\nu,2\bm\alpha}$ appear in rapidly converging
high-order quadrature rules for
boundary integral equations
\cite{wuCorrectedTrapezoidalRules2021,chenSinCoTrapHighOrderLocally2026}
with, among others, applications in elliptic partial differential equations.
Beyond lattice sums, anisotropic interaction kernels appear as convolution kernels in Calderón–Zygmund singular integral operators
\cite[Sec. III]{steinSingularIntegralsDifferentiability2016}.

\begin{remark*}[Notation]
Throughout, the positive integer $d\in\mathds N_+$ denotes the dimension, we define $0^0=1$, we denote by $\bm e_j$ the $j$'th canonical unit vector, 
we denote by $\|\bm u\|_1=|u^{(1)}|+\ldots+|u^{(d)}|$, with absolute value $|\cdot|$, the one-norm
of a vector $\bm u=(u^{(1)},\ldots,u^{(d)})^T\in\mathds C^d$, 
we denote by $\mathds N=\{0,1,\dots\}$ the non-negative integers and call $|\bm\alpha|=\|\bm\alpha\|_1$ the absolute value of a multi-index $\bm\alpha\in\mathds N^d$, and otherwise denote by $|\bm\cdot|$ the Euclidean norm.
\end{remark*}

We begin our exposition by introducing the concept of a $d$-dimensional lattice.

\begin{definition}[Lattices]
Let \(A\in\mathds{R}^{d\times d}\) be regular.
We call the periodic set of points \(\Lambda = A\mathds{Z}^{d}\) a Bravais lattice with elementary cell volume  \(V_{\Lambda}=|\:\!\!\det A|\). We further denote by \(\Lambda^{\ast}=A^{-T}\mathds{Z}^{d}\) the reciprocal lattice.
\end{definition}

For vanishing anisotropy multi-index $\bm\alpha=\bm 0$ we recover the isotropic interaction kernel $V_{\nu}$, whose lattice sum augmented by an oscillatory factor is the Epstein zeta function \cite{epstein1903theorieI,epstein1903theorieII} that generalizes the classical Riemann zeta function to higher dimensions.

\begin{definition}[Epstein zeta function]
\label{epsteindef}
Let $\Lambda$
be a $d$-dimensional
lattice, let $\bm x,\bm y\in\mathds R^d$ and let $\nu\in\mathds C$.
The Epstein zeta function is defined as the meromorphic continuation of the lattice sum
$$
\zeps{\Lambda,\nu}{\bm x}{\bm y} = \sums_{\bm z \in \Lambda} 
e^{-2\pi i \bm y\cdot \bm z}V_{\nu}(\bm z-\bm x)
\,,\qquad 
\Re{\nu}>d\,,
$$
 to \(\nu\in\mathds C\), where the primed sum signifies the exclusion of the summand $\bm z=\bm x$.
  \end{definition}

The Epstein zeta function
has been key in
recent studies
in long-range interacting
classical and quantum many-body systems
such as
magnetism
\cite{kimFieldInducedMagnonDecays2025,yadavObservationUnprecedentedFractional2025,buchheitZetaExpansionLongrange2026},
in long-range interacting Ising models \cite{koziol2025melting,koziolQuantumAnnealingLattice2025},
and theoretical chemistry
\cite{robles-navarroExactLatticeSummations2025,robles-navarroLatticeInstabilitiesTransformation2026}.
This work introduces its natural anisotropic generalization as follows.

\begin{definition}[Anisotropic Epstein zeta function]
\label{def-zeta-aniso}
Let $\Lambda$
be a $d$-dimensional
lattice, let $\bm\alpha\in\mathds N^d$, let $\bm x,\bm y\in\mathds R^d$ and let $\nu\in\mathds C$.
The anisotropic Epstein zeta function is defined as the meromorphic continuation of the lattice sum
$$
Z_{\Lambda,\nu,\bm\alpha}(\bm x, \bm y)=
\sums_{\bm z \in \Lambda} e^{-2\pi i \bm y \cdot \bm z}V_{\nu,\bm\alpha}(\bm z-\bm x)\,
,\qquad \Re\nu>d+|\bm\alpha|\,,
$$
to $\nu\in \mathds  C$.
\end{definition}

The anisotropic Epstein zeta function is closely connected to both shift and wave vector derivatives of Epstein zeta functions, enabling its precomputation through series expansions.
This connection is due to the appearance of the anisotropic monomial $\bm z^{\bm\alpha}$ in partial derivatives, prominently in plane waves
$
\diffop^{\bm\alpha} e^{\bm z\cdot\bm y}
=\bm z^{\bm\alpha}e^{\bm z\cdot\bm y}
$,
and consequently
$$
\diffop^{\bm\alpha} \zeps{\Lambda,\nu}{\bm 0}{\bm y}
=(-2\pi i)^{|\bm\alpha|} Z_{\Lambda,\nu,\bm\alpha}(\bm 0,\bm y)\,,
$$
where $\diffop^{\bm\alpha}=\partial_{y^{(1)}}^{\alpha^{(1)}}\ldots\partial_{y^{(d)}}^{\alpha^{(d)}}$.
More broadly, vector
derivatives of Epstein zeta functions form the foundation of a recent, exponentially fast converging algorithm for computing magnetic interactions between arrays of solid bodies \cite{buchheitZetaExpansionLongrange2026}
and stabilize the computation of wave vector integrals over products of Epstein zeta functions \cite{buchheitEpsteinZetaMethod2025}.
There, the removal of singularities of the derivatives increases the radius of analyticity, improving the convergence rate of the quadrature routine.
The anisotropic Epstein zeta function can be decomposed into a power-law singularity and a regularized analytic part as a function of $\bm y$ near  $\bm y=\bm 0$
via
$$
Z_{\Lambda, \nu,\bm\alpha}^{\mathrm{reg}}(\bm 0,\bm y) = Z_{\Lambda,\nu,\bm\alpha}(\bm 0,\bm y) - \frac{\hat s^{(\bm\alpha)}_{\nu}(\bm y)}{(-2\pi i)^{|\bm\alpha|}V_{\Lambda}},\qquad \bm y\neq \bm 0\,,
$$
and continuously extended to $\bm y = \bm 0$, where $\hat s_\nu^{(\bm\alpha)}$ denotes the $\bm\alpha$-derivatives of the distributional Fourier transform of $s_\nu(\cdot)=\vert \bm \cdot\vert^{-\nu}$
and $V_{\Lambda}=\vert\:\!\!\det A\vert$ is the volume of the elementary lattice cell of $\Lambda=A\mathds Z^d$.
The regularized anisotropic Epstein zeta function plays a key role in the recently derived singular Euler-Maclaurin expansion (SEM), the generalization of the classical Euler-Maclaurin formula, known from analytic number theory, to physically relevant lattice sums with power-law singularities
\cite{buchheit2022efficient,buchheit2022singular,buchheit2023exact}.
In particular, the difference between integral and sum of an integrable function
$f_{\bm x}(\bm y)=s_\nu(\bm y-\bm x) g(\bm y)$
with band-limited $g$ and power-law singularity $s_{\nu}$
can be represented as 
a series over the derivatives $g^{(\bm\alpha)}$ of $g(\bm \cdot)$
$$
\sums_{\bm y\in\Lambda}f_{\bm x}(\bm y)
-\frac{1}{V_{\Lambda}}\int_{\mathds R^d}f_{\bm x}(\bm y)\d \bm y
=\sum_{\bm\alpha \in \mathds N^d}
c_{\bm\alpha}g^{(\bm\alpha)}(\bm x)
\,,
$$
where the coefficients
$c_{\bm\alpha}$
are given in terms of the regularized anisotropic Epstein zeta functions, which is a direct consequence of
\cite{buchheit2022singular}, with a numerical example in \Cref{sec:num-sem}.
This representation forms the basis for an efficient computation of singular sums at a numerical cost independent of the number of non-zero summands. 

Special cases of anisotropic Epstein zeta functions have long appeared in the literature, yet a unified treatment has not been established.
Computational treatments, such as the classical Ewald summation or the truncation of the lattice sum, 
have so far only been established for particular choices of the exponent \cite{cumminsApplicationsEwaldMethod1976}, multi-indices of small absolute value \cite{cohenDipolarSumsPrimitive1955,viyuelaChiralTopologicalSuperconductors2018} and vanishing shift- and wave vectors
\cite{wuCorrectedTrapezoidalRules2021,chenSinCoTrapHighOrderLocally2026}
and do not extend to arbitrary real parameters.
Differentiating Epstein zeta functions formally yields their anisotropic generalization; differentiation through point values, however, is numerically unstable,
rendering the computation of lattice sums over strongly anisotropic interaction kernels corresponding to
high-order derivatives intractable.
While the stable and efficient computation of Epstein zeta functions has been established in \cite{buchheitComputationPropertiesEpstein2026},
their series expansions necessary for precomputation have thus been inaccessible.
Convergence rates and stability of numerical schemes such as quadrature rules crucially depend on the distance to singularities, but a complete picture of the analytic properties of lattice sums with anisotropic interaction kernels has so far been missing.

We solve these challenges by deriving exact analytic representations of anisotropic lattice sums.
Our main result is a stably computable representation of the anisotropic Epstein zeta function 
that reaches the previously intractable strongly anisotropic regime, based on a harmonic decomposition of the differential operator.
This rests on the backward stable evaluation of the harmonic polynomials, for which we derive an algorithm.
We further derive a representation which straightforwardly corresponds to derivatives in both vector arguments and provide a complete discussion of its analytic properties.
Notably, we discuss the singularity structure and holomorphic extension in both the exponent $\nu\in\mathds C$ and, in the vector arguments $\bm x,\bm y$, to complex subsets of $\mathds C^d$.
We present an algorithm complete with error bounds for the anisotropic Epstein zeta function for arbitrary lattices, real exponents, real shift and wave vectors and arbitrary anisotropy strength $\bm\alpha\in\mathds N^d$, with benchmarks up to $|\bm\alpha|=60$.
The algorithm, including the harmonic polynomial routine, is implemented in EpsteinLib, a high-performance C library complete with Python, Julia and Mathematica bindings.
In numerical experiments, wherever independent reference values exist, we stably obtain
machine precision for the anisotropic Epstein zeta function
and backward stability for the harmonic polynomials.
In the most relevant case of one to three dimensions we observe a linear increase in runtime in $0\le |\bm\alpha|\le 10$
with evaluation times of less than a millisecond.
We showcase the library by the singular Euler-Maclaurin expansion for a non-trivial geometry in three dimensions.
Our stably computable representations
together with their implementation in EpsteinLib and the complete account of analytic properties 
establish the anisotropic Epstein zeta function as a powerful and versatile tool for computing long-range interacting classical and quantum lattices.

This work is structured as follows.
In \Cref{sec:base} we introduce the set zeta function, a generalization of the Epstein zeta function with a modified oscillating prefactor convenient for differentiation.
We derive a representation for the set zeta function derivatives for shifted lattices, suitable for the study of the analytic structure
of the anisotropic Epstein zeta function
and the computation of low order vector derivatives of the Epstein zeta function.
We further discuss the joint holomorphic extension of the anisotropic Epstein zeta function and the handling of its singularities, in particular
the regularization in the wave vector around the origin, where the singularity in $\bm y =\bm 0$ is removed.
In \Cref{sec:main} 
we present our main result, a stable method for the computation of the anisotropic Epstein zeta function and the regularization in the wave vector around the origin based on a harmonic decomposition of the derivative operator.
The algorithm for the computation of the anisotropic Epstein zeta function, based on our stably computable representation of the harmonic polynomials, is introduced in \Cref{sec:alg}.
We rigorously benchmark precision, stability and runtime of the implementation of the algorithm provided by EpsteinLib 
and showcase an application of the singular Euler-Maclaurin expansion
in \Cref{sec:num}.
We present our conclusions in \Cref{sec:outlook}.

\section{Analytic properties}

\label{sec:base}

This section focuses on the analytic properties of anisotropic lattice sums based on derivatives of isotropic lattice sums.
In \Cref{sec:base-defs}
we begin by introducing the wave vector derivatives of generalized set zeta functions,
which, in the case of shifted lattices, agree with anisotropic Epstein zeta functions up to an oscillatory prefactor.
In \Cref{sec:base-crandall}, we derive a Crandall type representation suitable for both the computation of low-order derivatives and the discussion of the holomorphic extension in \Cref{sec:base-hol}.
We conclude with a discussion of the handling of the singularities in \Cref{sec:base-sing}.

\subsection{Elementary properties}
\label{sec:base-defs}

For methodological reasons, we first define the wave vector derivatives of the set zeta function. 
The latter is a generalization of the Epstein zeta function to uniformly discrete sets which we also define here.
It includes an additional oscillatory Bloch prefactor that simplifies subsequent representation that make use of wave vector derivatives.

\begin{definition}[Set zeta derivatives]
\label{defsetzeta}
Let $L\subseteq \mathds R^d$ be uniformly discrete, that is,
$\inf\{|\bm x-\bm y|:\bm x,\bm y\in L,\ \bm y\neq \bm x\}>0$.
Let $\bm y\in \mathds R^d$, let $\nu\in \mathds C$
and let $\bm\alpha\in\mathds N^d$.
Then the set zeta derivatives are defined as
\[
Z^{(\bm\alpha)}_{L,\nu}(\bm y)=
\diffop^{\bm\alpha}\sums_{\bm z \in L} \frac{e^{-2\pi i \bm z\cdot \bm y}}{\vert \bm z\vert^\nu}\,,\qquad \mathrm{Re}(\nu)>d+|\bm\alpha|\,,
\]
and meromorphically continued to $\nu\in \mathds  C$, whenever such a continuation exists. 
In the case of no derivatives $\bm\alpha=\bm 0$ we call $Z^{(\bm 0)}_{L,\nu}(\bm y)$ the set zeta function for which we write
 $Z_{L,\nu}(\bm y)$.
\end{definition}

In \Cref{sec:zetadirectsum} we show that the lattice sum defined above converges absolutely.
The set zeta function was first introduced as a generalization of the Epstein zeta function to geometries with boundaries
 \cite{buchheitComputationLatticeSums2024}.
Set zeta derivatives with respect to $\bm x$ for $n$-dimensional lattices in $d$-dimensional space appear in \cite{buchheitZetaExpansionLongrange2026}.
In this work we focus on the special case of shifted lattices where the set zeta derivatives agree with anisotropic Epstein zeta functions up to an oscillatory prefactor, as the following lemma shows.
Here and throughout we directly state all identities involving anisotropic Epstein zeta functions and set zeta derivatives in full generality regarding the holomorphic extension in $\nu$. The region of holomorphy is discussed in \Cref{sec:base-hol}.

\begin{lemma}
\label{zeta-ani-set-zeta-der}
Let $\Lambda$ be a $d$-dimensional lattice and let $\bm x,\bm y\in\mathds R^d$ and let $\nu\in\mathds C$.
In the case of shifted lattices $L=\Lambda-\bm x$, 
the Epstein zeta function agrees, up to an oscillatory prefactor, with the set zeta function for shifted lattices 
$$
Z_{\Lambda,\nu}(\bm x,\bm y)
=
e^{-2\pi i\bm x\cdot\bm y}
Z_{\Lambda-\bm x,\nu}(\bm y)\,,
$$
if not both $\nu = d$ and $\bm y\in\Lambda^*$.
For $\bm\alpha\in\mathds N^d$,
the anisotropic Epstein zeta function agrees, up to an oscillatory prefactor, with the set zeta derivatives for shifted lattices
$$
Z_{\Lambda,\nu,\bm\alpha}(\bm x,\bm y)
=
(-2\pi i)^{-|\bm\alpha|}
e^{-2\pi i\bm x\cdot\bm y}
Z^{(\bm\alpha)}_{\Lambda-\bm x,\nu}(\bm y)\,,
$$
if not altogether $\nu=d+|\bm\alpha|$, $\bm y\in\Lambda^*$ and $\chi(\bm\alpha)=1$.
Here $\chi$ is the indicator function of the set of multi-indices with all-even components, that is $\chi(\bm\alpha)=1$ if all components of $\bm\alpha$ are even and $\chi(\bm\alpha)=0$ otherwise.
\end{lemma}

\begin{proof}
For the first and second part of the statement we may restrict our discussion to $\mathrm{Re}(\nu)>d$ and $\mathrm{Re}(\nu)>d+|\bm\alpha|$ respectively. 
The statement then follows by the uniqueness of the holomorphic continuation, whose existence we prove in \Cref{hol} for the regions as stated.

For the first part observe that direct sums of the Epstein zeta function, see \Cref{epsteindef}, and the set zeta function, see \Cref{defsetzeta}, agree up to the oscillatory factor $e^{\pm 2\pi i \bm x\cdot \bm y}$ that does not change the convergence of the sum.
By the uniqueness of the meromorphic continuation both sides agree in $\nu\in\mathds C$.

By \Cref{zetadirectsum}, the direct sum of the anisotropic Epstein zeta function, see \Cref{def-zeta-aniso}, agrees with the differentiated series of the set zeta derivative, see \Cref{defsetzeta}, up to the factor $(-2\pi i)^{\pm|\bm\alpha|}e^{\pm 2\pi i\bm x\cdot \bm y}$ which concludes the second part of the proof.
\end{proof}

\begin{remark}
\label{diff-then-cont}
Note that in the second equation in \Cref{zeta-ani-set-zeta-der},
the right-hand side is understood as the meromorphic continuation of 
the set zeta derivatives $Z^{(\bm\alpha)}_{\Lambda-\bm x,\nu}(\bm y)$ from $\Re{\nu}>d+|\bm\alpha|$ to $\nu\in\mathds C$.
In contrast,
the derivative $\diffop^{\bm\alpha}$ of the meromorphically continued set zeta function $Z_{\Lambda-\bm x,\nu}(\bm y)$ with $\Re{\nu}\le d+|\bm\alpha|$ is in general ill defined at $\bm y\in\Lambda^*$ due to power-law singularities at reciprocal lattice points as discussed in \Cref{sec:base-sing}.
\end{remark}

The anisotropic Epstein zeta function in turn, can be used to compute the derivatives of Epstein zeta functions in both wave and shift vector.

\begin{lemma}
Let $\Lambda$ be a $d$-dimensional lattice and let $\bm\alpha\in\mathds N^d$.
Let $\nu\in\mathds C\setminus\{d\}$ and $\bm y\in\mathds R^d$, then for any $\bm x\in\mathds R^d\setminus \Lambda$, the $\bm\alpha$-derivatives of the Epstein zeta function
with respect to $\bm x$ are given by
$$
\diffop[x]^{\bm\alpha} 
\zeps{\Lambda,\nu}{\bm x}{\bm y}=
a_{\Lambda,\nu,\bm\alpha}\,
e^{-2\pi i \bm x\cdot \bm y}
Z_{\Lambda^*\!,d-\nu,\bm\alpha}(\bm y,-\bm x)\,,
$$
with the prefactor $a_{\Lambda,\nu,\bm\alpha}$ defined as
$$
a_{\Lambda,\nu,\bm\alpha}=(2\pi i)^{|\bm\alpha|}
\frac{\pi^{\nu-d/2}}{V_{\Lambda}}\frac{\Gamma((d-\nu)/2)}{\Gamma(\nu/2)}\,,
$$
in the sense of the limit for $\nu\in d+2\mathds N_+$.
Let $\nu\in\mathds C$ and $\bm x\in\mathds R^d$, then for any $\bm y\in\mathds R^d\setminus \Lambda^*$, the $\bm\alpha$-derivatives of the Epstein zeta function
with respect to $\bm y$ are given by
$$
\diffop^{\bm\alpha} 
\zeps{\Lambda,\nu}{\bm x}{\bm y}
=
\sum_{\bm\beta \le \bm\alpha}
b_{\bm\alpha,\bm\beta}(\bm x)
Z_{\Lambda,\nu,\bm\beta}(\bm x,\bm y)
\,,
$$
where $b_{\bm\alpha,\bm\beta}$ denotes the monomial
$$
b_{\bm\alpha,\bm\beta}(\bm x)=
(-2\pi i)^{|\bm\alpha|}\binom{\bm\alpha}{\bm\beta}\bm x^{\bm\alpha-\bm\beta}\,,
$$
with
$\binom{\bm\alpha}{\bm\beta}=\binom{\alpha^{(1)}}{\beta^{(1)}}\ldots\binom{\alpha^{(d)}}{\beta^{(d)}}$
for $\bm\alpha,\bm\beta\in\mathds N^d$.
\end{lemma}

\begin{proof}
Given $\nu\in\mathds C\setminus\{d\}$ and $\bm y\in \mathds R^d$, the Epstein zeta function is smooth in $\bm x\in\mathds R^d\setminus\Lambda$ by \cite[Th. 8]{buchheitComputationPropertiesEpstein2026}, so differentiation is justified.
We employ the functional equation  \cite[Cor. 7]{buchheitComputationPropertiesEpstein2026}
$$
\zeps{\Lambda,\nu}{\bm x}{\bm y}
=
\frac{\pi^{\nu-d/2}}{V_{\Lambda}}
\frac{\Gamma((d-\nu)/2)}{\Gamma(\nu/2)}
e^{-2\pi i\bm x\cdot\bm y}
\zeps{\Lambda^*\!,d-\nu}{\bm y}{-\bm x}
$$
where the vector arguments of the last term are swapped, 
so that the last two factors are a set zeta function where the oscillatory factor is removed
$$
e^{2\pi i\bm y\cdot(-\bm x)}
\zeps{\Lambda^*\!,d-\nu}{\bm y}{-\bm x}
=
Z_{\Lambda^*-\bm y,d-\nu}(-\bm x)\,.
$$
Differentiation with respect to $\bm x$ thus yields 
$$
\diffop[x]^{\bm\alpha}
\zeps{\Lambda,\nu}{\bm x}{\bm y}
=
(-1)^{|\bm\alpha|}\frac{\pi^{\nu-d/2}}{V_{\Lambda}}
\frac{\Gamma((d-\nu)/2)}{\Gamma(\nu/2)}
Z^{(\bm\alpha)}_{\Lambda^*-\bm y,d-\nu}(-\bm x)
$$
and the statement follows from rewriting the set zeta derivative on the right-hand side as an anisotropic Epstein zeta function, see \Cref{zeta-ani-set-zeta-der}.

Given $\nu\in\mathds C$ and $\bm x\in\mathds R^d$, the Epstein zeta function is smooth in $\bm y\in\mathds R^d\setminus\Lambda^*$ by \cite[Th. 8]{buchheitComputationPropertiesEpstein2026}.
By \Cref{zeta-ani-set-zeta-der}, we have
$
Z_{\Lambda,\nu}(\bm x,\bm y)
=
e^{-2\pi i\bm x\cdot\bm y}
Z_{\Lambda-\bm x,\nu}(\bm y)
$.
Applying $\diffop^{\bm\alpha}$ to both sides yields
$$
\diffop^{\bm\alpha}Z_{\Lambda,\nu}(\bm x,\bm y)
=
\sum_{\bm\beta\le\bm\alpha}\binom{\bm\alpha}{\bm\beta}
(-2\pi i \bm x)^{\bm\alpha-\bm\beta}
e^{-2\pi i\bm x\cdot\bm y}
Z^{(\bm\beta)}_{\Lambda-\bm x,\nu}(\bm y)
$$
by the product rule for the derivative.
By \Cref{zeta-ani-set-zeta-der}
we can write the last three factors on the right-hand side in terms of the anisotropic Epstein zeta function
where $(-2\pi i)^{|\bm\beta|}e^{-2\pi i\bm x\cdot\bm y}$ is consumed, that is
$$
(-2\pi i \bm x)^{\bm\alpha-\bm\beta}
e^{-2\pi i \bm x\cdot\bm y} Z^{(\bm\beta)}_{\Lambda-\bm x,\nu}(\bm y)
=
(-2\pi i)^{|\bm\alpha|}\bm x^{\bm\alpha-\bm \beta} Z_{\Lambda,\nu,\bm\beta}(\bm x,\bm y)
$$
which concludes the proof.
\end{proof}

The anisotropic Epstein zeta function exhibits numerous symmetries in its arguments $\bm{x}$ and $\bm{y}$, as well as with respect to lattice rescaling, which we summarize in the following lemma.

\begin{lemma}[Symmetries] 
\label{lem:syms}
\anisoreq 
Then:
     \begin{enumerate}
\item (Inversion)  Inversion of $\bm x$ equals inversion of $\bm y$ up to a sign,
$$
Z_{\Lambda, \nu,\bm\alpha}(-\bm x,\bm y) 
= 
(-1)^{|\bm\alpha|}Z_{\Lambda, \nu,\bm\alpha}(\bm x,- \bm y)
\,.
$$
\item (Translation) The anisotropic Epstein zeta function is, up to a prefactor, $\Lambda$-periodic in $\bm x$ and $\Lambda^\ast$-periodic in $\bm y$, that is, for $\bm v \in \Lambda$ and $\bm p \in \Lambda^*$, it holds that
$$
Z_{\Lambda, \nu,\bm\alpha}(\bm x+\bm v,\bm y+\bm p) = e^{-2 \pi i  \bm y \cdot \bm v} Z_{\Lambda, \nu,\bm\alpha}(\bm x,\bm y).
$$
 \item (Scaling) For $\mu >0$,
$$ Z_{\Lambda, \nu,\bm\alpha}(\bm x,\bm y)
= \mu^{\nu-|\bm\alpha|}Z_{\mu \Lambda, \nu,\bm\alpha}(\mu \bm x,\bm y / \mu).
$$
     \end{enumerate}
\end{lemma}

\begin{proof}
We may restrict our discussion to $\mathrm{Re}(\nu)>d+|\bm\alpha|$ by the uniqueness of the analytic continuation, whose existence for the region as stated we prove in \Cref{hol}.
There, the anisotropic Epstein zeta function is defined via an absolutely convergent lattice sum, see \Cref{zetadirectsum}.
The first part is a direct consequence of $\Lambda=-\Lambda$ and
$$
V_{\nu,\bm\alpha}((-\bm z)-\bm x)
=
(-1)^{|\bm\alpha|}
V_{\nu,\bm\alpha}(\bm z-(-\bm x))\,.
$$
For the second property, rewrite $Z_{\Lambda, \nu,\bm\alpha}(\bm x+\bm v,\bm y+\bm p)$ as
$$
e^{-2\pi i \bm y\cdot \bm v}\sums_{\bm z \in \Lambda} \, 
e^{-2\pi i \bm p \cdot \bm z}e^{-2\pi i \bm y \cdot (\bm z-\bm v)}V_{\nu,\bm\alpha}(\bm z-(\bm x+\bm v))
$$
so the identity follows from
$\Lambda-\bm v = \Lambda$
and
$e^{-2\pi i \bm p \cdot \bm z} =1$ as $\bm p\cdot \bm z\in \mathds Z$ for $\bm z\in \Lambda$ since $\bm p\in \Lambda^\ast$.
The scaling symmetry for the anisotropic Epstein zeta function is obtained via rewriting $Z_{\mu \Lambda, \nu,\bm\alpha}(\mu \bm x,\bm y / \mu)$ as
\[
\sums_{\bm z \in \Lambda}\, 
\frac{(\mu\bm z-\mu\bm x)^{\bm\alpha}}{\left|\mu \bm z - \mu \bm x\right|^\nu} 
e^{-2 \pi i \bm y/\mu \cdot \mu \bm z}
= 
\mu^{|\bm\alpha|-\nu} 
Z_{\Lambda, \nu,\bm\alpha}(\bm x,\bm y)
\,.
\qedhere
\]
\end{proof}

The anisotropic Epstein zeta function vanishes for various particular parameter choices due to symmetry.

\begin{lemma}[Symmetry related zeros]
\label{syms-zero}
Let $\Lambda$ be a $d$-dimensional lattice, let $\bm\alpha\in\mathds N^d$,
let $\nu\in\mathds C$ and let $\bm x,\bm y\in\mathds R^d$.
Then $Z_{\Lambda,\nu,\bm\alpha}(\bm x,\bm y)=0$ if any of the following holds.
\begin{enumerate}
    \item (Inversion) We have that $|\bm\alpha|$ is odd, $\bm x\in\Lambda$ and $2\bm y\in\Lambda^*$.
    \item (Shift vector mirroring) There exists $j\in\{1,\ldots,d\}$ such that $\Lambda$ is mirror-symmetric in the $j$'th component,  $\alpha^{(j)}$ is odd, $2x^{(j)} \bm e_j\in\Lambda$ and  $y^{(j)}+p^{(j)} = 0$
for some $\bm p\in\Lambda^*$.
    \item (Wave vector mirroring) There exists $j\in\{1,\ldots,d\}$ such that $\Lambda$ is mirror-symmetric in the $j$'th component,  $\alpha^{(j)}$ is odd, $2y^{(j)} \bm e_j\in\Lambda^*$ and  $x^{(j)}+v^{(j)} = 0$
for some $\bm v\in\Lambda$.
\end{enumerate}
\end{lemma}

\begin{proof}
Let $\bm v=-\bm x$ and $\bm p=-2\bm y$, so that by the conditions we have $\bm v\in\Lambda$ and $\bm p\in\Lambda^*$, and by the translational symmetry in \Cref{lem:syms}
we have
$$
q\,Z_{\Lambda,\nu,\bm\alpha}(\bm x,\bm y)
=
Z_{\Lambda,\nu,\bm\alpha}(\bm 0,-\bm y)
\,,
$$
with the phase $q=e^{-2\pi i \bm y\cdot \bm v}$.
By \Cref{lem:syms}, from the inversion symmetry, and since $|\bm\alpha|$ odd, it follows that
$
Z_{\Lambda,\nu,\bm\alpha}(\bm 0,-\bm y)
=
-Z_{\Lambda,\nu,\bm\alpha}(\bm 0,\bm y)
$,
and by the translational symmetry for $\bm p = \bm 0$ we have
$
-Z_{\Lambda,\nu,\bm\alpha}(\bm 0,\bm y)
=
-q\,Z_{\Lambda,\nu,\bm\alpha}(\bm x,\bm y)
$.
Since the phase $q$ does not vanish, we have shown that
$Z_{\Lambda,\nu,\bm\alpha}(\bm x,\bm y)
=
-Z_{\Lambda,\nu,\bm\alpha}(\bm x,\bm y)$, so that the anisotropic Epstein zeta function vanishes.

For the mirror related symmetries, we restrict our discussion to $\mathrm{Re}(\nu)>d+|\bm\alpha|$ so that the statement extends to $\nu\in\mathds C$ by the identity theorem for holomorphic functions since zero is an entire function.
For $\bm u\in\mathds R^d$ use the notation 
$$\bm u^{[j]}=(u^{(1)},u^{(2)},\ldots,u^{(j-1)},u^{(j+1)},\ldots,u^{(d)})^T\in\mathds R^{d-1}$$
for vectors  where the $j$'th component is missing.
First assume
 $2x^{(j)} \bm e_j\in\Lambda$
and $y^{(j)}+p^{(j)} = 0$.
Let $\tilde{\bm y}=\bm y+\bm p$, then by the translational symmetry, rewrite
$ Z_{\Lambda,\nu,\bm\alpha}(\bm x,\bm y)$ as
$$
Z_{\Lambda, \nu,\bm\alpha}(\bm x,\tilde{\bm y}) = 
\sums_{\bm z \in \Lambda}\, 
(z^{(j)}-x^{(j)})^{\alpha^{(j)}}\frac{(\bm z^{[j]}-\bm x^{[j]})^{\bm\alpha^{[j]}}e^{-2\pi i\bm z^{[j]}\cdot \tilde{\bm y}^{[j]}}}{\big((\bm z^{[j]} - \bm x^{[j]})^2+(z^{(j)}-x^{(j)})^2\big)^{\nu/2}}
\,,
$$
where we used that $e^{-2\pi i z^{(j)}\tilde{y}^{(j)}}=1$ since $\tilde{y}^{(j)}=0$.
By \Cref{zetadirectsum} the lattice sum converges absolutely so that reordering of the summands does not change its value.
Note that $\Lambda$ is mirror symmetric in the $j$'th component and that $\Lambda =2x^{(j)}\bm e_j+\Lambda$ since $2x^{(j)}\bm e_j\in\Lambda$
so that $z^{(j)}\mapsto 2 x^{(j)}-z^{(j)}$ is a symmetry of the lattice.
This maps $z^{(j)}-x^{(j)}$ to $-(z^{(j)}-x^{(j)})$
and since $\alpha^{(j)}$ is odd we
have that the summand 
changes sign under this map and the fixed points $z^{(j)}=x^{(j)}$ contribute zero as $\alpha^{(j)}\geq 1$,
so pairing these values makes the sum vanish.

Now assume $2y^{(j)} \bm e_j\in\Lambda^*$
and $x^{(j)}+v^{(j)} = 0$
and, with $\tilde{\bm x}=\bm x+\bm v$, rewrite
$e^{-2 \pi i  \bm y \cdot \bm v} Z_{\Lambda,\nu,\bm\alpha}(\bm x,\bm y)$ as
$$
Z_{\Lambda, \nu,\bm\alpha}(\tilde{\bm x},\bm y) = 
\sums_{\bm z \in \Lambda}\, 
e^{-\pi i \bm z\cdot (2 y^{(j)}\bm e_j)}
(z^{(j)})^{\alpha^{(j)}}
\frac{(\bm z^{[j]}-\tilde{\bm x}^{[j]})^{\bm\alpha^{[j]}}e^{-2\pi i\bm z^{[j]}\cdot \bm y^{[j]}}}{\big((\bm z^{[j]} - \tilde{\bm x}^{[j]})^2+(z^{(j)})^2\big)^{\nu/2}}
\,.
$$
As $(2 y^{(j)}\bm e_j)\in\Lambda^*$ and $\bm z\in\Lambda$ we have
$
e^{-\pi i \bm z\cdot (2 y^{(j)}\bm e_j)}
=\pm 1
$
where the sign is fixed by the parity of the integer $2z^{(j)}y^{(j)}$, which is invariant under $z^{(j)}\mapsto -z^{(j)}$.
Since $\alpha^{(j)}$ is odd, and due to the mirror symmetry of the lattice, the summand is odd in $z^{(j)}$, with the terms $z^{(j)}=0$ vanishing, so the sum vanishes and the claim follows as $e^{-2\pi i\bm y\cdot\bm v}\neq 0$.
\end{proof}

In the following section, we derive a representation that allows to access the meromorphic continuation.

\subsection{Derivative Crandall representation}
\label{sec:base-crandall}

We access the holomorphic continuation of the anisotropic Epstein zeta function through a Crandall type representation
which one of the authors introduced for the set zeta function in 
\cite{buchheitComputationLatticeSums2024}
based on the work of Richard  Crandall 
\cite{crandall2012unified}.
The Crandall representation is obtained by the classical Riemann splitting formulated at the level of the kernel $|\bm\cdot|^{-\nu}$.
For an extensive discussion of these ideas and the Crandall representation of the Epstein zeta function, see \cite[Sec. 2]{buchheitComputationPropertiesEpstein2026}.
Here, we summarize the main definitions and properties, beginning with a definition of the upper Crandall function.

\begin{definition}[Upper Crandall function]
\label{def:crandall-upper}
Let $\nu\in\mathds C$ and $\bm z\in D$
with the complex sector
$$
D=\{\bm u\in\mathds C^d:|\Re{\bm u}|>|\Im{\bm u}|\},
$$
where real and imaginary parts are applied component-wise and where $\vert \bm \cdot \vert$ denotes the Euclidean norm.
Furthermore use the notation
$$
\bm z^2=\sum_{j=1}^d(z^{(j)})^2.
$$
We define the upper Crandall function $G_{\nu}(\bm z)$ as
$$
G_{\nu}(\bm{z})= 
	\frac{\Gamma(\nu/2,\pi \bm{z}^2)}{(\pi \bm{z}^2)^{\nu/2}}
,\qquad G_{\nu}(\bm 0)
=-\frac{2}{\nu},
$$
where $\Gamma(s,x)$,
$s,x\in\mathds C$, $\Re x>0$
denotes the upper incomplete Gamma function, defined by the holomorphic continuation of
$$
\Gamma(s,x)=
\int_x^{\infty}t^{s}e^{-t}\frac{\d t}{t}
,\qquad \Re{s}>0,
$$
to  
$s\in\mathds C$.
\end{definition}

The upper Crandall function is convenient for series representations, due to its smoothness and superexponential decay as stated in the following lemma, which is a part of \cite[Lem. 4]{buchheitComputationPropertiesEpstein2026}.

\begin{lemma}[Properties of the upper Crandall function]
\label{crandall-upper-prop}
Let $\nu \in\mathds C$ and $\bm z\in D$.
The upper Crandall function $G_{\nu}(\bm z)$ is a jointly holomorphic function in $(\nu,\bm z)\in \mathds C \times D$.
Define
$$
u=\pi \mathrm{Re}(\bm z^2),\quad v=\max\{0,\mathrm{Re}(\nu/2)-1\}\,.
$$ 
We then have
$$
|G_\nu(\bm z)| \leq \frac{e^{-u}}{u-v}\,, \quad u>v\,.
$$
Let $\Re{\nu}<0$, then $\lim_{\bm z\to\bm 0}G_{\nu}(\bm z)=G_{\nu}(\bm 0)$.
\end{lemma}

The following Crandall type representation conveniently accesses the meromorphic continuation of the anisotropic Epstein zeta function through two super-exponentially converging series involving the derivatives of the exponential and derivatives of the upper Crandall function.
These derivatives of upper Crandall functions are given by sums of shifted upper Crandall functions with polynomial weights as shown in \cite[Lem. 4.1]{buchheitZetaExpansionLongrange2026}, for which we provide a new proof for the complex case in \Cref{sec:crandall-derivatives}.

\begin{theorem}[Derivative Crandall representation of the anisotropic Epstein zeta function]
\label{derivative-crandall-representation}
\anisoreq
Then,
for any splitting parameter $\lambda>0$, it holds that
\begin{align*}
Z_{\Lambda,\nu,\bm\alpha}(\bm x,\bm y)
=
\frac{(\pi/\lambda^2)^{\nu/2}}{\Gamma(\nu/2)}
\Bigg[
&\sum_{\bm z\in \Lambda-\bm x}
\bm z^{\bm\alpha}
G_{\nu}(\bm{z}/\lambda)e^{-2\pi i\bm y\cdot (\bm z+\bm x)}
\\
&\quad+
(-2\pi i)^{-|\bm\alpha|}
\frac{\lambda^{d+|\bm\alpha|}}{V_{\Lambda}}\sum_{\bm p\in\Lambda^{\ast}+\bm y}
G^{(\bm\alpha)}_{d-\nu}
(\lambda \bm p)e^{-2\pi i\bm x\cdot\bm p}
\Bigg]\,.
\end{align*}
The derivatives of the upper Crandall function, which we call upper Crandall derivatives, are given by
$$
G_\nu^{(\bm\alpha)}(\bm u) = 
\sum_{2\bm\beta\le \bm\alpha}p_{\bm\alpha,\bm \beta}(\bm u)
G_{\nu+2|\bm\alpha|-2|\bm\beta|}( \bm u)
\,, 
\qquad \bm u\in D
$$
with $D$ as in \Cref{def:crandall-upper}.
Furthermore, for $\Re\nu+|\bm\alpha|<0$ and $\bm u\in\mathds R^d\setminus\{\bm 0\}$, the derivatives of the upper Crandall function continuously extend to $\bm u=\bm 0$ with
$$
G_{\nu}^{(\bm\alpha)}(\bm 0)= 
-2\chi(\bm\alpha)\frac{p_{\bm\alpha,\bm\alpha/2}(\bm 0)}{\nu+|\bm\alpha|}\,,
$$
where the right-hand side forms the
meromorphic continuation to $\nu\in\mathds C$ 
with a simple pole at $\nu=-|\bm\alpha|$ if and only if $\chi(\bm\alpha)=1$.
Here, $p_{\bm\alpha,\bm \beta}$ is the monomial
$$
p_{\bm\alpha,\bm \beta}(\bm y)
=
(-\pi)^{|\bm\alpha|-|\bm \beta|}
(\bm\alpha-\bm\beta)_{\bm\beta}
\binom{\bm\alpha}{\bm \beta}  
(2\bm y)^{\bm\alpha-2 \bm \beta}\,,
$$
and
$\chi(\bm\alpha)=1$ if all components of $\bm\alpha$ are even and $\chi(\bm\alpha)=0$ otherwise.
Furthermore,
 $(\nu)_{k}=\nu(\nu-1)\ldots(\nu-k+1)$, $\nu\in\mathds C$,  $k\in\mathds N_+$, $(\nu)_{0}=1$ denotes the falling Pochhammer symbol and we write
$(\bm\alpha)_{\bm\beta}=(\alpha^{(1)})_{\beta^{(1)}}\ldots (\alpha^{(d)})_{\beta^{(d)}}$.
\end{theorem}

\begin{proof}
In this proof we derive the derivative Crandall representation for $\Re\nu>d+|\bm\alpha|$.
In the proof of \Cref{hol} we show that the right-hand side of the derivative Crandall representation is entire in $\nu$ if $\bm y\notin\Lambda^*$ or $\chi(\bm\alpha)=0$
and that the right-hand side is holomorphic in $\nu\in\mathds C\setminus\{d+|\bm\alpha|\}$ for $\bm y\in\Lambda^*$ and $\chi(\bm\alpha)=1$.
Therefore the derivative Crandall representation is the unique holomorphic extension of the anisotropic Epstein zeta function and the left-hand side is understood as this continuation.
The explicit representation for the upper Crandall derivatives is derived in \Cref{sec:crandall-derivatives}.

In what follows we will show a derivative Crandall representation for the set zeta derivatives, that is we will show that
\begin{align*}
Z_{\Lambda-\bm x,\nu}^{(\bm{\alpha})}(\bm y)
=
\frac{(\pi/\lambda^2)^{\nu/2}}{\Gamma(\nu/2)}
\Bigg[
&\sum_{\bm z\in \Lambda-\bm x}
(-2\pi i \bm z)^{\bm\alpha}
G_{\nu}(\bm{z}/\lambda)e^{-2\pi i\bm y\cdot \bm z}
\\
&\quad+
\frac{\lambda^{d+|\bm\alpha|}}{V_{\Lambda}}\sum_{\bm p\in\Lambda^{\ast}}
G^{(\bm\alpha)}_{d-\nu}
(\lambda(\bm p+\bm y))e^{-2\pi i\bm x\cdot\bm p}
\Bigg]\,.
\end{align*}
Then, the derivative Crandall representation for the anisotropic Epstein zeta function is a direct consequence of \Cref{zeta-ani-set-zeta-der} as the nowhere-vanishing oscillatory prefactor can be absorbed in the two series whose convergence remains unchanged.

Let $\Re\nu>d+|\bm\alpha|$.
By the Crandall representation for the Epstein zeta function 
\cite[Th. 6]{buchheitComputationPropertiesEpstein2026}
and by \Cref{zeta-ani-set-zeta-der}, we have
\begin{align*}
Z_{\Lambda-\bm x,\nu}(\bm y)
=
\frac{(\pi/\lambda^2)^{\nu/2}}{\Gamma(\nu/2)}
\Bigg[
&\sum_{\bm z\in \Lambda-\bm x}
G_{\nu}(\bm{z}/\lambda)e^{-2\pi i\bm y\cdot \bm z}
\\
&+
\frac{\lambda^{d}}{V_{\Lambda}}\sum_{\bm p\in\Lambda^{\ast}}
G_{d-\nu}
(\lambda(\bm p+\bm y))e^{-2\pi i\bm x\cdot\bm p}
\Bigg]\,.
\end{align*}
We have
$\diffop^{\bm\alpha}e^{-2\pi i \bm y\cdot\bm z}=(-2\pi i \bm z)^{\bm\alpha}e^{-2\pi i \bm y\cdot\bm z}$
and by the chain rule
$
\diffop^{\bm\alpha}G_{d-\nu}(\lambda(\bm p+\bm y))
=
\lambda^{|\bm\alpha|}G^{(\bm\alpha)}_{d-\nu}(\lambda(\bm p+\bm y))
$
so the derivative Crandall representation follows from differentiating under the sum.
We show that $\diffop^{\bm\alpha}$ commutes with the sums
by showing that a majorant of the $\diffop^{\bm\gamma}$ differentiated series with some sufficiently large central cutout
$$
\sum_{\substack{\bm z\in \Lambda-\bm x \\|\bm z|>r}}
(-2\pi i \bm z)^{\bm\gamma}
G_{\nu}(\bm{z}/\lambda)e^{-2\pi i\bm y\cdot \bm z}
+
\frac{\lambda^{d+|\bm\gamma|}}{V_{\Lambda}}\sum_{\substack{\bm p\in\Lambda^{\ast}\\ |\bm p|>r}}
G^{(\bm\gamma)}_{d-\nu}
(\lambda(\bm p+\bm y))e^{-2\pi i\bm x\cdot\bm p}
$$
converges absolutely and locally uniformly in $\bm y\in\mathds R^d$ for $\bm\gamma\in\mathds N^d$, $\bm\gamma\le\bm\alpha$ similar to the proof of \Cref{zetadirectsum}.

The absolute value of the differentiated summand of the first sum is bounded by 
$
(2\pi)^{|\bm\gamma|}|\bm z|^{|\bm\gamma|}|G_{\nu}(\bm z/\lambda)|
$
uniformly in $\bm y$, 
so exchanging summation and differentiation is justified by the superexponential decay of the upper Crandall function, see \Cref{crandall-upper-prop}.

For the second sum let $K\subseteq\mathds R^d$ be compact
and choose $r>0$ large enough so that for any 
$\bm p\in\mathds R^d$ with $|\bm p|>r$ we have
$(\bm p+\bm y)^2\ge \bm p^2/2$
and
$\bm p^2\ge 2 (v+1)/(\pi \lambda^2)$
with $v\ge 0$
and $v\ge \max\{(d-\Re\nu+t)/2-1:t\in\{0,2,\ldots,2|\bm\alpha|\}\}$
for all $\bm y\in K$.
Then by \Cref{crandall-upper-prop} we have
$$
|G_{d-\nu+2|\bm\gamma|-2|\bm\beta|}(\lambda (\bm p+\bm y))|
\le e^{-\pi \lambda^2 \bm p^2/2}
$$
for $|\bm p|>r$
for all $2\bm \beta\le\bm\gamma$ uniformly in $\bm y\in K$.
In addition choose $r$ large enough so that
$2|\bm p+\bm y| \le 3|\bm p|$ 
and $3\lambda|\bm p|\ge 1$
for $|\bm p|>r$
so we have the loose upper bound
$$
|p_{\bm\gamma,\bm\beta}(\lambda(\bm p+\bm y))| 
\le 
(3\lambda\pi)^{|\bm\gamma|}|\bm\gamma|^{|\bm\gamma|}
|\bm\gamma|!
|\bm p|^{|\bm\gamma|}\,
$$
for all $2\bm \beta\le \bm\gamma$ uniformly in $\bm y\in K$.
Excluding the finite set of summation indices
$$
\sum_{\substack{\bm p\in\Lambda^* \\ |\bm p|>r}}
|G^{(\bm\gamma)}_{d-\nu}(\lambda(\bm p+\bm y))
e^{-2\pi i \bm x\cdot\bm p}|
\le 
C_{\bm\gamma,\lambda}
\sum_{\substack{\bm p\in\Lambda^* \\ |\bm p|>r}}
|\bm p|^{|\bm\gamma|}
e^{-\pi \lambda^2|\bm p|^2/2}
$$
with $C_{\bm\gamma,\lambda}=
(3\lambda\pi)^{|\bm\gamma|}|\bm\gamma|^{|\bm\gamma|}
|\bm\gamma|!(\sum_{2\bm \beta\le\bm\gamma}1)$
converges uniformly in $\bm y\in K$ due to the superexponential decay of the upper bound.
As $K$ was arbitrary both sums converge locally uniformly in $\bm y\in\mathds R^d$.
\end{proof}

\begin{remark}[Derivative Crandall representations for set zeta functions]
The beginning of the proof of \Cref{derivative-crandall-representation} features the derivative Crandall representation   for the set zeta derivatives.
A Crandall type representation of set zeta derivatives with respect to $\bm x$ for $n$-dimensional lattices in $d$-dimensional space can be found in \cite{buchheitZetaExpansionLongrange2026}.
For Crandall representation for the set zeta function without derivatives in geometries with boundaries, see
\cite{buchheitComputationLatticeSums2024}.
\end{remark}

The derivative Crandall representation in
\Cref{derivative-crandall-representation} forms the meromorphic continuation of the anisotropic Epstein zeta function in both the exponent and in the vector arguments
and allows for the removal of singularities without cancellation error as will be discussed in \Cref{sec:base-sing}.
It is exact and can be employed in computations, though precision is lost in floating point arithmetic for large $|\bm\alpha|$.
A stably computable representation for arbitrary anisotropy corresponding to both low- and high-order derivatives based on a harmonic decomposition of the derivative operator is derived in \Cref{sec:main}.
We now employ the derivative Crandall representation to discuss the holomorphic extension of anisotropic Epstein zeta functions.

\subsection{Holomorphic extension}
\label{sec:base-hol}

The anisotropic Epstein zeta function can be holomorphically extended in the complex plane in its various arguments, where singularities can occur at the shifted dimension $\nu\to d+|\bm\alpha|$ as well as in the lattice $\bm x\in \Lambda$ and the reciprocal lattice $\bm y\in\Lambda^*$.

\begin{theorem}[Holomorphic extension of the anisotropic Epstein zeta function]
\label{hol}
Define the set $D_L\subseteq \mathds C^d$ as the  following intersection of $d$-dimensional complex cones centered at $L\subseteq \mathds R^d$,
$$
D_L=\{\bm u\in\mathds C^d:|\Re{\bm u}-\bm z|>|\Im{\bm u}|\ \forall \bm z\in L\}\,,
$$
and let $\chi$ be as in \Cref{zeta-ani-set-zeta-der}.
Let $\Lambda$ be a $d$-dimensional lattice, let $\bm\alpha\in\mathds N^d$ and let $\nu\in\mathds C$.
For $\bm x \notin \Lambda$ and $\bm y \notin \Lambda^*$, the anisotropic Epstein zeta function can be holomorphically extended to  
$$
(\nu,\bm x,\bm y)\in 
\mathds C\times D_{\Lambda}\times D_{\Lambda^*}\,.
$$
For $\bm x\in\Lambda$ and $\bm y \notin \Lambda^*$, the anisotropic Epstein zeta function can be holomorphically extended to 
$$(\nu,\bm y)\in \mathds C\times D_{\Lambda^*}\,.
$$
For
$\bm y \in \Lambda^*$, $\bm x \notin \Lambda$, and $\chi(\bm\alpha)=0$, the anisotropic Epstein zeta  function can be holomorphically extended to 
$$
(\nu,\bm x)\in \mathds C\times D_{\Lambda}\,,
$$
and for $\chi(\bm\alpha)=1$
 it can be holomorphically extended to
$$
(\nu,\bm x)\in \Big( \mathds C\setminus\{d+|\bm\alpha|\}\Big)\times D_{\Lambda}
$$
with a simple pole at $\nu = d+|\bm\alpha|$.
Finally, for $\bm x\in\Lambda$, $\bm y\in\Lambda^*$, and $\chi(\bm\alpha)=0$, the anisotropic Epstein zeta function is holomorphic
in $\nu\in\mathds C$
and for $\chi(\bm\alpha)=1$ 
it is holomorphic
in $\nu\in\mathds C\setminus\{d+|\bm\alpha|\}$ with a simple pole at $\nu=d+|\bm\alpha|$.
\end{theorem}

\begin{proof}
Let $(\nu,\bm x,\bm y)\in W$ with $W=
\mathds C\times D_\Lambda \times D_{\Lambda^\ast}$.
By the properties of the upper Crandall functions, see \Cref{crandall-upper-prop}, and by the representation for the upper Crandall derivatives in \Cref{derivative-crandall-representation}
each summand in the derivative Crandall representation is individually holomorphic in $\nu\in\mathds C$, in each component of $\bm x\in D_{\Lambda}$ and in each component of $\bm y\in D_{\Lambda^*}$
and the prefactor $(\pi/\lambda^2)^{\nu/2}/\Gamma(\nu/2)$ is entire in $\nu\in\mathds C$.
By Hartogs' theorem \cite[Theorem 2.2.8]{hormanderIntroductionComplexAnalysis1990}
it follows that each summand is jointly holomorphic in 
$W$.
For the joint holomorphy of the anisotropic Epstein zeta function it is therefore enough to show that both sums in the derivative Crandall representation with a central cutout converge absolutely and locally uniformly in $W$, that is, there exists an $r>0$, such that
$f_r\colon W\to [0,\infty)$,
\begin{align*}
f_r(\nu,\bm x,\bm y)&=\sum_{\substack{\bm z\in \Lambda\\ |\bm z|> r}}
|(\bm z-\bm x)^{\bm\alpha}
G_{\nu}((\bm z-\bm x)/\lambda)e^{-2\pi i\bm y\cdot \bm z}|
\\
&\quad
+(2\pi)^{-|\bm\alpha|}\frac{\lambda^{d+|\bm\alpha|}}{V_{\Lambda}}\sum_{\substack{\bm p\in\Lambda^{\ast} \\ |\bm p|>r}}|G^{(\bm\alpha)}_{d-\nu}
(\lambda(\bm p+\bm y))e^{-2\pi i\bm x\cdot(\bm p+\bm y)}|
\end{align*}
converge uniformly in a compact subset $K$ of $W$.
Assume $(\nu,\bm x,\bm y)\in K$ and $\bm z,\bm p\in\mathds R^d$ with $\min\{|\bm z|,|\bm p|\}>r$.
Choose $r>0$ large enough such that $|\bm z-\bm x|\le 2|\bm z|$, $\Re{(\bm z-\bm x)^2}\ge \bm z^2/2$ and 
$\bm z^2\ge 2 \lambda^2 (v_1+1)/\pi$
with $v_1=\max \{0,\Re\nu/2-1\}$.
Then by \Cref{crandall-upper-prop}
$$
|(\bm z-\bm x)^{\bm\alpha}
G_{\nu}((\bm z-\bm x)/\lambda)
|
\le 2^{|\bm\alpha|} |\bm z|^{|\bm\alpha|}
e^{-\pi |\bm z|^2/(2\lambda^2)}\,.
$$
Furthermore, choose $r$ large enough such that
$\Re{(\bm p+\bm y)^2}\ge \bm p^2/2$
and
$\bm p^2\ge 2 (v_2+1)/(\pi \lambda^2)$
with $v_2\ge 0$
and $v_2\ge \max\{(d-\Re\nu+t)/2-1:t\in\{0,2,\ldots,2|\bm\alpha|\}\}$
as well as
$2|\bm p+\bm y| \le 3|\bm p|$ 
and $3\lambda|\bm p|\ge 1$.
Then, by the representation for the derivatives of the upper Crandall function in \Cref{derivative-crandall-representation} and with $C_{\bm\gamma,\lambda}=
(3\lambda\pi)^{|\bm\gamma|}|\bm\gamma|^{|\bm\gamma|}
|\bm\gamma|!(\sum_{2\bm \beta\le\bm\gamma}1)$ defined as in its proof we have
$$
|G^{(\bm\alpha)}_{d-\nu}(\lambda(\bm p+\bm y))
e^{-2\pi i \bm x\cdot\bm p}|
\le 
C_{\bm\alpha,\lambda}
|\bm p|^{|\bm\alpha|}
e^{-\pi \lambda^2|\bm p|^2/2}\,.
$$
Finally let $c_1,c_2,c_3$ be large enough such that
$$
\sup_{(\nu,\bm x,\bm y)\in K}|e^{-2\pi i \bm y\cdot \bm z}|
\le 
e^{c_1|\bm z|}
,\qquad 
\sup_{(\nu,\bm x,\bm y)\in K}|e^{-2\pi i \bm x\cdot(\bm p+\bm y)}|
\le 
c_2e^{c_3|\bm p|}\,.
$$
Then the majorant
\begin{align*}
f_r(\nu,\bm x,\bm y)\le
&
2^{|\bm\alpha|}\sum_{\substack{\bm z\in \Lambda\\ |\bm z|> r}}
|\bm z|^{|\bm\alpha|}e^{c_1|\bm z|}e^{-\pi |\bm z|^2/(2\lambda^2)}
\\
&\quad+
c_2
(2\pi)^{-|\bm\alpha|}
C_{\bm\alpha,\lambda}\frac{\lambda^{d+|\bm\alpha|}}{V_{\Lambda}}\sum_{\substack{\bm p\in\Lambda^{\ast} \\ |\bm p|>r}}
|\bm p|^{|\bm\alpha|}
e^{c_3|\bm p|}e^{-\pi \lambda^2|\bm p|^2/2}\,.
\end{align*}
converges uniformly in $K$ as the superexponential factors in each sum dominate the exponential and polynomial factors.
From compact uniform convergence then follows holomorphy of the anisotropic Epstein zeta function on $W$, see \cite[Theorem 10.28]{rudin1987real}.

The case of $\bm x\in\Lambda$ and $\bm y\in D_{\Lambda^*}$ proceeds in analogy, as the apparent singularity in the summand $\bm z=\bm x$
vanishes for $\bm\alpha\neq \bm 0$ and for $\bm\alpha = \bm 0$ we recover the Epstein zeta function with an oscillatory prefactor which can be holomorphically extended in $(\nu,\bm y)\in \mathds C\times D_{\Lambda^*}$ by \cite[Th. 8]{buchheitComputationPropertiesEpstein2026}.

In case that $\bm y\in\Lambda^*$, we find that the term $\bm p= -\bm y$ is of the form
$$
\chi(\bm\alpha)(-2\pi i)^{-|\bm\alpha|}
\frac{(\pi/\lambda^2 )^{\nu/2}}{\Gamma(\nu/2)}
\frac{\lambda^{d+|\bm\alpha|}}{V_{\Lambda}}
\frac{2 p_{\bm\alpha,\bm\alpha/2}(\bm 0)}{(\nu-d-|\bm\alpha|)}
\,,
$$
which exhibits a simple pole at $\nu = d+|\bm\alpha|$ if $\chi(\bm\alpha)=1$, that is, if all components of $\bm\alpha$ are even. Thus for $\bm y\in \Lambda^*$ and $\chi(\bm\alpha)=1$, the anisotropic Epstein zeta function 
can be holomorphically extended to
$(\nu, \bm x)\in (\mathds C\setminus \{d+|\bm\alpha|\})\times D_{\Lambda}$. Finally, the discussion of the $\bm z = \bm x$ and $\bm p = -\bm y$ terms above also shows that for $\bm x \in \Lambda$ and $\bm y\in \Lambda^\ast$, the anisotropic Epstein zeta function is meromorphic in $\nu \in \mathds C$ with at most a simple pole at $\nu=d+|\bm\alpha|$ that appears if and only if $\chi(\bm\alpha)=1$.
\end{proof}

\begin{remark}
\label{hol-set-zeta}
The set zeta derivatives for shifted lattices can be holomorphically extended to the holomorphy regions of the anisotropic Epstein zeta function
as they differ by
the prefactor
$(-2\pi i)^{\pm |\bm\alpha|}e^{\pm 2\pi i \bm x\cdot\bm y}$
that does not vanish and that is entire in the superset $(\nu,\bm x,\bm y)\in\mathds C\times \mathds C^d\times \mathds C^d$ of all holomorphy regions, 
that is to say, \Cref{hol} holds verbatim with $Z_{\Lambda,\nu,\bm\alpha}(\bm x,\bm y)$ replaced by
$Z^{(\bm\alpha)}_{\Lambda-\bm x,\nu}(\bm y)$.
\end{remark}

We now discuss the handling of the various singularities in the vector arguments, enabling evaluations near singularities without cancellation.

\subsection{Singularity handling}
\label{sec:base-sing}

\begin{figure}
    \centering
    \includegraphics[width=1\linewidth]{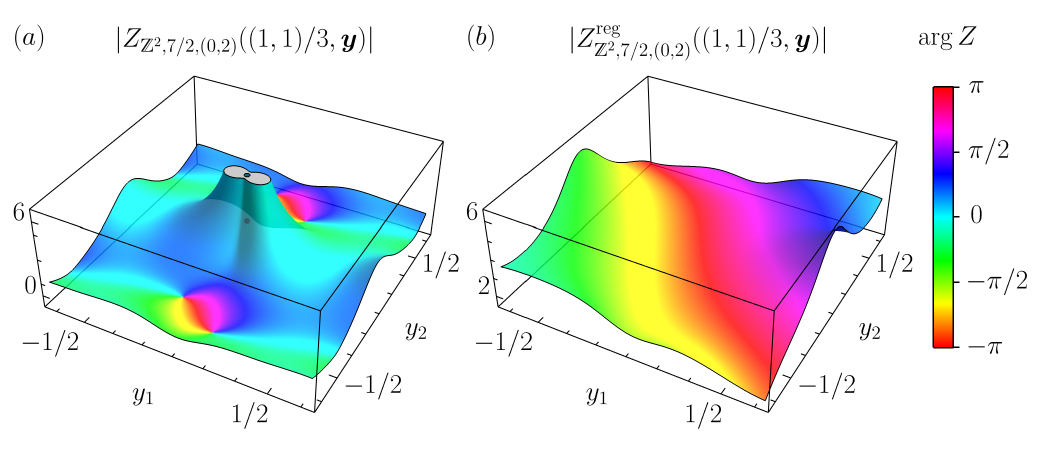}
    \caption{
    Two-dimensional anisotropic Epstein zeta function (a) and its regularized counterpart (b) for $\nu=7/2$, $\Lambda=\mathds Z^2$, $\bm x=(1,1)/3$ and $\bm\alpha=(0,2)$ as a function of $\bm y\in[-1/2,1/2]^2$,
    where the height of the surface encodes the absolute function values and the color encodes the arguments of the complex function values according to the cyclic colormap shown on the right.
    The anisotropic Epstein zeta function exhibits discontinuities at wave vector points in the reciprocal lattice, in particular at the origin $\bm y=\bm 0$, where the absolute value of the function at the origin, approximately equal to $3.11$, can be seen as a dark dot through the transparent surface (a). The regularized anisotropic Epstein zeta function as a function of \(\bm y\) is analytic in the elementary cell of the reciprocal lattice $\bm y \in [-1/2,1/2]^2$ and coincides with the anisotropic Epstein zeta function at $\bm y=\bm 0$ (b).}
    \label{fig:reg}
\end{figure}

The anisotropic Epstein zeta function exhibits power-law singularities near shift vectors in the lattice and near wave vectors in the reciprocal lattice.
These singularities appear due to the non-analyticity of the upper Crandall function at zero argument
$$
G_{\nu}(\bm u)\propto 
\frac{\Gamma(\nu/2)}{|\bm u|^{\nu}}
,\qquad 
\bm u\to\bm 0
$$
where $\bm u\in\mathds R^d$ and $\nu\in\mathds C$, $\Re{\nu}>0$, which is apparent from \Cref{def:crandall-upper} and the finiteness of $\Gamma(\nu/2)$.
The singularities at wave vectors in the reciprocal lattice correspond to Rayleigh--Wood anomalies whose handling is nontrivial in numerical algorithms for wave scattering, i.e.~in quasi-periodic structures \cite{denlinger2017fast,wood1902xlii,bruno2020evaluation}. 
The singularities in the wave vector around the origin as seen in \Cref{fig:reg} are of special interest, due to their appearance in integrals involving Epstein zeta function factors, 
such as in anomalous quantum spin wave dispersion relations
\cite{buchheitComputationPropertiesEpstein2026},
superconductors with long-range interactions \cite{buchheit2023exact}
or many-body zeta functions
\cite{buchheitEpsteinZetaMethod2025,robles-navarroExactLatticeSummations2025}.
Splitting the Epstein zeta function into a power-law singularity and a regular part enables integration without cancellation error.
A Taylor series of the resulting regularized Epstein zeta function, see \cite{buchheitEpsteinZetaMethod2025}, leads to what we will introduce here as the regularized anisotropic Epstein zeta function.

In general, any diverging summand in the derivative Crandall representation, \Cref{derivative-crandall-representation}, can be regularized as a function of the vector arguments, by replacing the upper Crandall function with the lower Crandall function, which is defined as follows.

\begin{definition}[Lower Crandall function]
\label{def:crandall-lower}
Let $\nu\in\mathds C\setminus(-2\mathds N)$ and $\bm z\in \mathds C^d$.
For $\bm z\neq \bm 0$ we define the lower Crandall function as
$$
g_{\nu}(\bm z)=
\frac{\gamma(\nu/2,\pi \bm{z}^2)}{(\pi \bm{z}^2)^{\nu/2}}
$$
and continuously extend to $g_{\nu}(\bm 0)=2/\nu$.
Here $\gamma(s,x)$, $x\in\mathds C$ denotes the lower incomplete Gamma function defined by the holomorphic extension of 
$$
\gamma(s,x)
=
\int_0^x t^s e^{-t}\frac{\d t}{t}
,\qquad 
\Re{s}>0
$$
to $s\in \mathds C\setminus (-\mathds N)$.
\end{definition}

For $\nu\notin (-2\mathds N)$,
the lower Crandall function is the regular part of the upper Crandall function as a function of the vector argument as the following lemma shows, which is a part of 
\cite[Lem. 4]{buchheitComputationPropertiesEpstein2026}.

\begin{lemma}[Properties of the lower Crandall function]
\label{crandall-lower-prop}
The lower Crandall function $g_{\nu}(\bm u)$ 
is jointly holomorphic in
$(\nu,\bm u)\in (\mathds C\setminus(-2\mathds N))\times \mathds C^d$.
Let $(\nu,\bm u) \in (\mathds C\setminus(-2\mathds N)) \times D$, then
we have the fundamental relation
$$
\frac{\Gamma(\nu/2)}{(\pi\bm u^2)^{\nu/2}}
=
g_{\nu}(\bm u)+G_{\nu}(\bm u)
$$  
where $D$ as in \Cref{def:crandall-upper}.
\end{lemma}

Special care needs to be taken in the case of $\nu\in(-2\mathds N)$, where a logarithmic factor appears due to the singularities of the Gamma function as will be seen in what follows.
We define the regularized anisotropic Epstein zeta function as the difference of the anisotropic Epstein zeta function with the derivative of the wave vector power-law singularity at the origin as follows.

\begin{definition}[Regularized anisotropic Epstein zeta function]
\label{def-zeta-aniso-reg}
Let $\Lambda$ be a $d$-dimensional lattice, let $\bm x,\bm y\in\mathds R^d$,
 let $\nu\in \mathds C$ and let $\bm\alpha\in\mathds N^d$.
We define the regularized anisotropic Epstein zeta function via 
\[Z_{\Lambda, \nu,\bm\alpha}^{\mathrm{reg}}(\bm x,\bm y) =e^{2\pi i\bm x\cdot\bm y} Z_{\Lambda,\nu,\bm\alpha}(\bm x,\bm y) -\frac{\hat s^{(\bm\alpha)}_{\nu}(\bm y)}{(-2\pi i)^{|\bm\alpha|}V_{\Lambda}}
,\qquad \bm y\neq \bm 0
\]
and continuously extend to $\bm y = \bm 0$. Here,
$\hat s^{(\bm\alpha)}_\nu$ denotes the $\bm\alpha$-derivative of the distributional Fourier transform of $s_\nu(\bm\cdot)=\vert \bm \cdot\vert^{-\nu}$,
which for $\nu \not\in (d+2\mathds N)$ is given by
\cite[p. 363]{gelfand1964generalizedI}
\[\hat{s}_\nu(\bm y) = \frac{\pi^{\nu/2}}{\Gamma(\nu/2)}\Gamma\big((d-\nu)/2\big)
(\pi \bm y^2)^{(\nu - d)/2}.
\]
For $\nu= d+ 2 \ell$ with $\ell\in\mathds N$, this Fourier transform is uniquely defined up to a polynomial of order $2\ell$, see \cite{hormander2003analysisI}. We adopt the choice
\[
\hat s_{d+2\ell}(\bm y)= \frac{\pi^{\ell+d/2}}{\Gamma(\ell+d/2)}\frac{(-1)^{\ell+1}}{\ell!} 
( \pi \bm y^2 )^{\ell} \log (\pi  \bm y^{2}).
\]
In the case of $\bm\alpha=\bm 0$
we write $Z^{\rm reg}_{\Lambda,\nu}(\bm x,\bm y)$ which corresponds to the regularized Epstein zeta function \cite[Def. 5]{buchheitComputationPropertiesEpstein2026}.
\end{definition}

Equivalently, the regularized anisotropic Epstein zeta function can be understood as the derivative of the regularized Epstein zeta function up to a prefactor that depends on the absolute value of the multi-index.

\begin{lemma}
\label{aniso-reg-derivative-id}
\anisoregreq
Then
$$
\diffop^{\bm\alpha} 
Z^{\rm reg}_{\Lambda,\nu}(\bm x,\bm y)
=
(-2\pi i)^{|\bm\alpha|}
Z^{\rm reg}_{\Lambda,\nu,\bm\alpha}(\bm x,\bm y)\,.
$$
\end{lemma}

\begin{proof}
By 
\cite[Th. 11]{buchheitComputationPropertiesEpstein2026}
we have that the regularized Epstein zeta function can be holomorphically extended to $\bm y\in D_{\Lambda^*\setminus \{\bm 0\}}$,
so the left-hand side is well defined.
Let $\bm y\neq \bm 0$.
The regularized Epstein zeta function has the same oscillatory prefactor convention as the set zeta derivatives so that by \Cref{def-zeta-aniso-reg} and the first equation of \Cref{zeta-ani-set-zeta-der} we have
$$
\diffop^{\bm\alpha}
Z^{\rm reg}_{\Lambda,\nu}(\bm x,\bm y)
=
Z^{(\bm\alpha)}_{\Lambda-\bm x,\nu}(\bm y)
-
\frac{1}{V_{\Lambda}}
\hat s^{(\bm\alpha)}_{\nu}(\bm y)
$$
By the second equation of \Cref{zeta-ani-set-zeta-der} the right-hand side is equal to
$$
(-2\pi i)^{|\bm\alpha|}
e^{2\pi i\bm x\cdot\bm y} Z_{\Lambda,\nu,\bm\alpha}(\bm x,\bm y) -\frac{\hat s^{(\bm\alpha)}_{\nu}(\bm y)}{V_{\Lambda}}
$$
and factoring out the prefactor and \Cref{def-zeta-aniso-reg} yields the representation for the derivatives.
In \Cref{hol-reg} we will show that the regularized anisotropic Epstein zeta function continuously extends to $\bm y=\bm 0$ so that the statement holds at the origin.
\end{proof}

Due to the removal of a designated singularity at a reciprocal lattice point, the regularized anisotropic Epstein zeta function does not satisfy the translation symmetry in the wave vector shown in \Cref{lem:syms}, but a scaling symmetry can be derived as follows.

\begin{corollary}[Scaling symmetry of the regularized anisotropic Epstein zeta functions]
\label{lem:syms-reg}
\anisoregreq
Let $\mu>0$.
Then, for $\nu\notin (d+2\mathds N)$, we have
$$
Z^{\rm reg}_{\Lambda,\nu,\bm\alpha}(\bm x,\bm y) 
=
\mu^{\nu-|\bm\alpha|} 
Z^{\rm reg}_{\mu\Lambda,\nu,\bm\alpha}(\mu\bm x,\bm y/\mu)\,.
$$
Furthermore define $Y_{\ell}(\bm y)=(\bm y^2)^{\ell}/\ell!$ with $\ell\in\mathds N$.
Then for $\nu= d+2\ell$ we have
$$
Z^{\rm reg}_{\Lambda,\nu,\bm\alpha}(\bm x,\bm y) 
=
\mu^{\nu-|\bm\alpha|} 
Z^{\rm reg}_{\mu\Lambda,\nu,\bm\alpha}(\mu\bm x,\bm y/\mu) 
+
\frac{(-1)^{\ell} \pi^{2\ell+d/2}\log(\mu^2)}{(-2\pi i)^{|\bm\alpha|}V_{\Lambda}\Gamma(\ell+d/2)}Y_{\ell}^{(\bm\alpha)}(\bm y)\,.
$$
\end{corollary}

\begin{proof}
Let $\bm y\neq \bm 0$.
For $\nu\notin (d+2\mathds N)$  the scaling symmetry of the regularized anisotropic Epstein zeta function follows from the scaling symmetry of the anisotropic Epstein zeta function, 
from \Cref{def-zeta-aniso-reg},
$$
\frac1{V_{\mu\Lambda}}\hat{s}_{\nu}(\bm y/\mu)
=
\frac 1{\mu^dV_{\Lambda}}\mu^{d-\nu}\hat{s}_\nu(\bm y)\,,
$$
where a $\mu^{-|\bm\alpha|}$ factor appears on the left-hand side by the chain rule when applying $\diffop^{\bm\alpha}$ to both sides of the equation above.
For $\nu=d+2\ell$ we obtain the correction term from
\[
\frac1{V_{\mu\Lambda}}\hat{s}_{d+2\ell}(\bm y/\mu)
=
\frac1{\mu^{d}V_{\Lambda}}\mu^{-2\ell}\Big(\hat{s}_{d+2\ell}(\bm y)
+
(-1)^{\ell} \frac{\pi^{2\ell+d/2}\log (\mu^2)}{\Gamma(\ell+d/2)}Y_{\ell}(\bm y)
\Big)\,.
\]
Differentiating both sides and applying the chain rule gives the statement, where we note that by the regularity of the regularized anisotropic Epstein zeta function and since the correction term is a polynomial in $\bm y$, all sides of the equations continuously extend to $\bm y = \bm 0$.
\end{proof}

The Crandall type representation for the regularized anisotropic Epstein zeta function, where we replace the diverging upper Crandall function with the lower Crandall function reads as follows.

\begin{theorem}[Derivative Crandall representation for the regularized anisotropic Epstein zeta function]
\label{low-ders-aniso-reg}
Let $\Lambda$ be a $d$-dimensional lattice, let $\bm x \in\mathds{R}^d$, let
$\bm y \in (\mathds R^d\setminus\Lambda^*)\cup\{\bm 0\}$, let $\nu \in \mathds{C}$ and let $\bm\alpha\in\mathds N^d$.
Then for $\lambda > 0$,
\begin{align*}
Z^{\rm reg}_{\Lambda,\nu,\bm\alpha}(\bm x,\bm y)
=
&\frac{(\pi/\lambda^2)^{\nu/2}}{\Gamma(\nu/2)}
\Bigg[
\sum_{\bm{z}\in\Lambda-\bm x}
\bm z^{\bm\alpha}G_{\nu}(\bm z/\lambda)e^{-2\pi i\bm y\cdot \bm z} 
\\
&\quad
+
(-2\pi i)^{-|\bm\alpha|}
\frac{\lambda^{d}}{V_{\Lambda}}
\Big[
G_{d-\nu,\lambda}^{\mathrm{reg},(\bm\alpha)}(\bm y)+
\lambda^{|\bm\alpha|}
\sum_{\substack{\bm p\in\Lambda^{\ast} \\ \bm p \neq \bm 0}}G_{d-\nu}^{(\bm\alpha)}
(\lambda (\bm p+\bm y))e^{-2\pi i \bm{x}\cdot \bm p}
\Big]\Bigg]
\end{align*}
where the regularized upper Crandall function is defined as
$$
G_{\nu,\lambda}^{\rm reg}(\bm y)
=
G_{\nu}(\lambda\bm y)-
\lambda^{-\nu}\pi^{-(d-\nu)/2}\Gamma((d-\nu)/2)\hat{s}_{d-\nu}(\bm y).
$$
If \(\nu\notin -2\mathds N\)
or $\nu=-2\ell$ with $\ell\in\mathds N$ and $\ell <|\lceil \bm\alpha/2\rceil|$ we have the $\bm\alpha$-derivatives
$$
G_{\nu,\lambda}^{\mathrm{reg},(\bm\alpha)}(\bm y)=
-\lambda^{|\bm\alpha|}g_{\nu}^{(\bm\alpha)}(\lambda\bm y)
\,,
$$
and for any $\ell\in\mathds N$ it holds that
$$
G^{\rm reg,(\bm\alpha)}_{-2\ell,\lambda}(\bm y)
=(-\pi \lambda^2)^{\ell}(H_\ell-\gamma -2\log\lambda)Y_\ell^{(\bm\alpha)}(\bm y)
-\sum_{\substack{j=0 \\ j\neq \ell}}^{\infty}\frac{(-\pi \lambda^2)^j}{(j-\ell)}Y_j^{(\bm\alpha)}(\bm y)\,,
$$
where $Y_\ell$ as in \Cref{lem:syms-reg}.
Here, \(\gamma\) is the Euler–Mascheroni constant,
 \(H_\ell\) is the \(\ell\)'th harmonic number
and the ceiling function $\lceil \bm\cdot \rceil$ is applied component wise.
\end{theorem}

\begin{proof}
Let $\bm y\neq \bm 0$.
By \Cref{zeta-ani-set-zeta-der}, we have
$$
Z^{\rm reg}_{\Lambda,\nu,\bm\alpha}(\bm x,\bm y)
=
(-2\pi i)^{-|\bm\alpha|}\Big(
Z^{(\bm\alpha)}_{\Lambda-\bm x,\nu}(\bm y)
-\frac{1}{V_{\Lambda}}\hat s_{\nu}^{(\bm\alpha)}(\bm y)\Big)
$$
and the representation is immediate from subtracting the singularity from the summand $\bm p=\bm 0$ in the derivative Crandall representation, \Cref{derivative-crandall-representation}, with the prefactor $e^{2\pi i\bm x\cdot\bm y}$ absorbed.
For $\nu\notin (-2\mathds N)$, we have
$$
G_{\nu,\lambda}^{\mathrm{reg},(\bm\alpha)}(\bm y)
=\diffop^{\bm\alpha}\Big[G_{\nu}(\lambda \bm y) - \frac{\Gamma(\nu/2)}{(\pi \lambda^2 \bm y^2)^{\nu/2}}\Big]
$$
where the right-hand side simplifies to the derivatives of lower Crandall function with a minus sign by the fundamental relation, see \Cref{crandall-lower-prop}.
For $\ell\in\mathds N$, write
$$
G_{-2\ell,\lambda}^{\operatorname{reg},(\bm\alpha)}(\bm y)=
\diffop^{\bm\alpha}\Big[
(\pi\lambda^2\bm y^2)^\ell\Big(\Gamma(-\ell,\pi\lambda^2\bm y^2)
+\frac{(-1)^\ell}{\ell!} (\log (\pi\lambda^2\bm y^2)-\log\lambda^2)\Big)\Big]
$$
and apply the series representation
$$z^{\ell}\Gamma(-\ell,z)=
\frac{(-z)^\ell}{\ell!}(H_\ell-\gamma-\log z)-\sum_{\substack{j=0 \\ j \neq \ell}}^{\infty}
\frac{(-z)^j}{(j-\ell)j!}
$$
see \cite[Eq. (5.1.12) \& (6.5.9)]{abramowitz}
at \(z=\pi\lambda^2\bm y^2\) to obtain
$$
G^{\rm reg,(\bm\alpha)}_{-2\ell,\lambda}(\bm y)
=(-\pi \lambda^2)^{\ell}(H_\ell-\gamma -2\log\lambda)Y_\ell^{(\bm\alpha)}(\bm y)
-\sum_{\substack{j=0 \\ j\neq \ell}}^{\infty}\frac{(-\pi \lambda^2)^j}{(j-\ell)}Y_j^{(\bm\alpha)}(\bm y)\,.
$$
For $\ell <|\lceil\bm\alpha/2\rceil|$, the order of the derivative is larger than the order of the polynomial $(\bm y^2)^{\ell}$ so that $Y_\ell^{(\bm\alpha)}(\bm y)$ vanishes and we have
$$
G^{\rm reg,(\bm\alpha)}_{-2\ell,\lambda}(\bm y)
=
-\diffop^{\bm\alpha}\sum_{\substack{j=0 \\ j\neq \ell}}^{\infty}\frac{(-\pi \lambda^2 \bm y^2)^j}{(j-\ell)j!}
\,.
$$
The series on the right-hand side is the series expansion of $g_{\nu}(\lambda \bm y)$ without the singular term $j=\ell$, which $\diffop^{\bm\alpha}$ annihilates, by \cite[Eq. (6.5.4) \& (6.5.29)]{abramowitz}.
By the chain rule we have $\diffop^{\bm\alpha}g_{\nu}(\lambda\bm y)=\lambda^{|\bm\alpha|}g^{(\bm\alpha)}_{\nu}(\lambda\bm y)$ which concludes the proof.
\end{proof}

\begin{remark}
\label{remark:hol-low-ders-aniso-reg}
\Cref{low-ders-aniso-reg} corresponds to the derivative Crandall representation for the regularized Epstein zeta function \cite[Th. 9]{buchheitComputationPropertiesEpstein2026} at $\bm\alpha=\bm 0$.
The regularized Crandall function was first introduced there.
In \cite[Lem. 10]{buchheitComputationPropertiesEpstein2026} it is shown that $G^{\rm reg}_{\nu,\lambda}(\bm y)$ can be holomorphically extended to $\bm y\in\mathds C^d$, where $\nu\in\mathds C$ and $\lambda>0$ are fixed.
\end{remark}

As a consequence, $Z^{\rm reg}_{\Lambda,\nu,\bm\alpha}(\bm x,\bm y)$ with $\bm\alpha=\bm 0$ is holomorphic in $\bm y$ in a complex region around the origin, see \cite[Th. 11]{buchheitComputationPropertiesEpstein2026}, which we will now generalize to $\bm\alpha\in\mathds N^d$.
To that end we first provide explicit representations for the derivatives of the regularized Crandall function, whose existence is guaranteed by \Cref{remark:hol-low-ders-aniso-reg}.
The derivatives of the lower Crandall function are given similarly to the derivatives of the upper Crandall function derived in \Cref{derivative-crandall-representation}.
The identity for real vector arguments is a part of \cite[Lem. 4.1]{buchheitZetaExpansionLongrange2026}. 
We give a new proof that extends to complex vector arguments in \Cref{sec:crandall-derivatives}.

\begin{lemma}[Lower Crandall derivatives]
\label{crandall-lower-der}
Let $\bm\alpha \in \mathds N^d$,
let $\nu\in\mathds C\setminus(-2|\lceil\bm\alpha/2\rceil|-2\mathds N)$ and let $\bm u\in \mathds C^d$.
Then the derivatives of the lower Crandall function, which we call lower Crandall derivatives, are given by
$$
g_\nu^{(\bm\alpha)}(\bm u) = 
\sum_{2\bm\beta\le \bm\alpha}p_{\bm\alpha,\bm \beta}(\bm u)
g_{\nu+2|\bm\alpha|-2|\bm\beta|}(\bm u)\,,
$$
where, for $\nu=-2\ell$, $\ell\in\mathds N$, $\ell<|\lceil\bm\alpha/2\rceil|$, we define $g_{\nu}^{(\bm\alpha)}(\bm u)$ in terms of the right-hand side.
At $\bm u=\bm 0$, the right-hand side of
$$
g_{\nu}^{(\bm\alpha)}(\bm 0)= 
2\chi(\bm\alpha)\frac{p_{\bm\alpha,\bm\alpha/2}(\bm 0)}{\nu+|\bm\alpha|} 
$$
forms the meromorphic continuation to 
$\nu\in\mathds C$ with a simple pole at $\nu=-|\bm\alpha|$ if and only if $\chi(\bm\alpha)=1$.
Here,
$p_{\bm\alpha,\bm\beta}$ and $\chi$ are defined as in \Cref{derivative-crandall-representation}.
\end{lemma}

The derivatives of the polynomial $Y_\ell$ in \Cref{low-ders-aniso-reg} can be represented as follows.

\begin{lemma}
\label{derivative-y2n}
Let $\ell\in\mathds N$, let $\bm\alpha\in\mathds N^d$ and let
$\bm y\in\mathds R^d$.
Then
$$
Y_\ell^{(\bm\alpha)}(\bm y)
=
\sum_{\substack{\bm\beta \ge\bm\alpha/2 \\ |\bm\beta| = \ell }}
\frac{(2\bm\beta)_{\bm\alpha}}{\bm\beta!}
\bm y^{2\bm\beta-\bm\alpha}\,,
$$
and in particular 
$Y^{(\bm\alpha)}_\ell(\bm y)=0$ if $\ell < |\lceil \bm\alpha/2\rceil|$.
Here, $Y_\ell(\bm \cdot)$ is defined as in \Cref{lem:syms-reg}.
\end{lemma}

\begin{proof}
By the multinomial theorem, we have
$$
Y_\ell^{(\bm\alpha)}(\bm y)
=\frac{1}{\ell!}\sum_{\substack{\bm\beta\ge \bm 0\\|\bm\beta| =\ell}}\frac{\ell!}{\bm\beta!}\diffop^{\bm\alpha}\bm y^{2\bm\beta}\,.
$$
The first equation of the statement follows from
$$
\diffop^{\bm\alpha}\bm y^{\bm\gamma}=
	\begin{cases}
(\bm\gamma)_{\bm\alpha}\bm y^{\bm\gamma-\bm\alpha}
&\bm\alpha\le\bm\gamma, \\
0&\text{else},
\end{cases}
$$
with $\bm\gamma\in\mathds N^d$,
for $\bm\gamma=2\bm\beta$. 

For the second part of the statement let $\ell <|\lceil \bm\alpha/2\rceil|$.
Then the sum is empty, as $\bm\beta\in\mathds N^d$ and $\bm\beta\ge\bm\alpha/2$ 
imply $\bm\beta\ge\lceil\bm\alpha/2\rceil$
which is a contradiction with $|\bm\beta|=\ell$,
so that the derivative vanishes.
\end{proof}

The regularized anisotropic Epstein zeta function can be holomorphically extended as a function of the wave vector around the origin.

\begin{theorem}[Holomorphic extension of the regularized anisotropic Epstein zeta function]
\label{hol-reg}
Let $\Lambda$ be a $d$-dimensional lattice, let $\bm\alpha\in\mathds N^d$ and let 
$
D_L=\{\bm u\in\mathds C^d:|\Re{\bm u}-\bm z|>|\Im{\bm u}|\ \forall \bm z\in L\}\,,
$
be as in \Cref{hol}.
Then the regularized anisotropic Epstein zeta function can be holomorphically extended to \[
(\nu,\bm x, \bm y)\in \big(\mathds C\setminus (d+2|\lceil \bm\alpha/2\rceil| +2\mathds N)\big)\times D_\Lambda\times D_{\Lambda^\ast\setminus\{\bm 0\}}.
\]
For $\nu\in (d+2|\lceil \bm\alpha/2\rceil| +2 \mathds N)$, the regularized anisotropic Epstein zeta function can be holomorphically extended to
\[(\bm x, \bm y) \in D_\Lambda\times D_{\Lambda^\ast\setminus\{\bm 0\}}.
\]
Finally, for $\bm x\in \Lambda$, the regularized anisotropic Epstein zeta function can be holomorphically extended to
\[
(\nu,\bm y) \in \big(\mathds C\setminus (d+2|\lceil \bm\alpha/2\rceil| +2\mathds N)\big)\times D_{\Lambda^\ast\setminus\{\bm 0\}}.
\]
\end{theorem}

\begin{proof}
Let $\lambda>0$ be fixed. We show compact uniform convergence and holomorphy of every summand in Crandall's representation.
Note that beyond the summand at $\bm p=\bm 0$, the derivative Crandall representation of the anisotropic Epstein zeta function, \Cref{derivative-crandall-representation}, is the same as the derivative Crandall representation for the regularized anisotropic Epstein zeta function, \Cref{low-ders-aniso-reg}, with an oscillatory exponential factor $e^{2\pi i\bm x\cdot\bm y}$ independent of the summation variables.
Therefore, uniform convergence is the same as in the proof of \Cref{hol}.

We separately investigate the holomorphy of the derivative Crandall representation without the regularized Crandall function and the holomorphy of $G^{\mathrm{reg},(\bm\alpha)}_{\nu,\lambda}(\bm y)$, so that the intersection of the domains of holomorphy yields the desired result. 
Removing the $\bm p=\bm 0$ summand 
$G^{(\bm\alpha)}_{d-\nu}(\lambda(\bm 0+\bm y))$ in the derivative Crandall representation, \Cref{derivative-crandall-representation},
extends the holomorphy to
$$
(\nu,\bm x, \bm y)\in \mathds C\times D_\Lambda\times D_{\Lambda^\ast\setminus\{\bm 0\}}
$$
and for $\bm x\in\Lambda$ to
$
(\nu,\bm y) \in  \mathds C\times D_{\Lambda^\ast\setminus\{\bm 0\}}
$. 

Let $\nu\in\mathds C\setminus (d+2|\lceil \bm\alpha/2\rceil| +2\mathds N)$ then, by \Cref{low-ders-aniso-reg} we have that
$G^{\mathrm{reg},(\bm\alpha)}_{d-\nu,\lambda}(\bm y)$ is a derivative of a Crandall function that by \Cref{crandall-lower-der} is a sum of terms of the form
$$
p_{\bm\alpha,\bm\beta}(\bm y) 
g_{d-\nu+2|\bm\alpha|-2|\bm\beta|}(\bm y)
$$
with $\bm\beta \in\mathds N^d$ with $2\bm\beta\le\bm\alpha$.
Here the first factor is a polynomial that is holomorphic in $\bm y\in\mathds C^d$ and by \Cref{crandall-lower-prop}
the second factor is jointly holomorphic in
$(\nu,\bm y)\in \big(\mathds C\setminus(d+2|\bm\alpha|-2|\bm\beta|+2\mathds N)\big)\times \mathds C^d$.
It follows that $G^{\mathrm{reg},(\bm\alpha)}_{d-\nu,\lambda}(\bm y)$
is a finite sum of functions jointly holomorphic in
$$
(\nu,\bm y)\in \big(\mathds C\setminus(d+2|\lceil \bm\alpha/2\rceil|+2\mathds N)\big)\times \mathds C^d
$$
where we used that 
$2\bm\beta\le\bm\alpha$ implies
$\bm\alpha-\bm\beta\ge \bm\alpha-\lfloor\bm\alpha/2\rfloor =\lceil\bm\alpha/2\rceil$.

Finally let $\nu\in(d+2|\lceil \bm\alpha/2\rceil|+2\mathds N)$, then $G^{\mathrm{reg},(\bm\alpha)}_{d-\nu,\lambda}(\bm y)$ is holomorphic in $\bm y\in\mathds C^d$ by \Cref{remark:hol-low-ders-aniso-reg}.
\end{proof}

We have shown that the derivative Crandall representations \Cref{derivative-crandall-representation} and \Cref{low-ders-aniso-reg} are convenient tools for the study of the analytic properties of the anisotropic Epstein zeta function and its regularized counterpart.
These representations are exact and can be used for computations, though caution is warranted in floating point arithmetic.
On the one hand, numerical experiments show that derivatives of the regularized Crandall function can be stably computed through
\Cref{crandall-lower-der}
and \Cref{derivative-y2n}
for arbitrary anisotropy.
On the other hand, unfortunately the representation for the derivatives of the non-regularized upper Crandall function in \Cref{derivative-crandall-representation} introduces catastrophic cancellation errors for high-order anisotropy, rendering \Cref{derivative-crandall-representation} unsuitable for stable computations.
We informally call the anisotropy high-order if $|\bm\alpha|$ is large, and low-order otherwise.
In the following section, we derive stably computable representations for arbitrary anisotropy through a harmonic decomposition of the partial derivative.

\section{Stably computable representations}
\label{sec:main}

Both the analytic structure and abundant numerical evidence, such as the exponential growth of the error in \Cref{fig:fig_stab} in the later \Cref{sec:num}, suggest that computing upper Crandall derivatives in floating point arithmetic introduces systematic loss in precision for high-order anisotropic Epstein zeta functions, and that stably computable representations cannot be obtained from applying the partial derivatives directly at the level of the derivative Crandall representation in \Cref{derivative-crandall-representation}.
More promisingly, even one-dimensional anisotropy $\alpha=2n$, $n\in\mathds N$ amounts to a shift in the power-law exponent
$$
Z_{\Lambda,\nu,2 n}(x,y)
=
Z_{\Lambda,\nu-2 n}(x,y)\,,
$$
which follows directly from the lattice sum, see \Cref{def-zeta-aniso}.
Here $x,y\in\mathds R$ and $\nu$ and $\Lambda$ are as in \Cref{zeta-ani-set-zeta-der}.
Investigating high-order differential operators that can be applied to the set zeta function without incurring loss of precision, we find the following multi-dimensional generalization of this convenient property.

\begin{lemma}
\label{poly-laplacian}
Let $\Lambda$
be a $d$-dimensional
lattice, let $n\in\mathds N$, let $\bm x,\bm y\in\mathds R^d$ and let $\nu\in\mathds C$ such that $\nu\neq d+2n$ if $\bm y\in\Lambda^*$. 
Then
$$
\Delta^n Z_{\Lambda-\bm x,\nu}(\bm y)=
(-4\pi^2)^nZ_{\Lambda-\bm x,\nu-2n}(\bm y)\,,
$$
where the poly-Laplacian $\Delta^n$ acts on $\bm y$,
and where the left-hand side is understood as the meromorphic continuation of the termwise differentiated series, see \Cref{diff-then-cont}.
\end{lemma}

\begin{proof}
Let $\Re\nu > d+2n$,
then by \Cref{zetadirectsum} we can differentiate under the sum
$$
\Delta^n Z_{\Lambda-\bm x, \nu}(\bm y)
	=
    (-4\pi^2)^n 
    \sums_{\bm{z}\in \Lambda-\bm x} 
    |\bm z|^{2n}
    \frac{e^{-2\pi i \bm y\cdot \bm{z}}}{|\bm{z}|^{\nu}}\,
$$
where the right-hand side is the shifted set zeta function and where we used that
$$
\Delta^n e^{-2\pi i\bm z\cdot\bm y}
=
((-2\pi i \bm z)^2)^n e^{-2\pi i\bm z\cdot\bm y}\,.
$$
By \Cref{hol-set-zeta} the right-hand side of the lemma is the meromorphic continuation to $\nu\in\mathds C$ as stated and the left-hand side is understood in the sense of this unique meromorphic continuation.
\end{proof}

Since anisotropic Epstein zeta functions agree with set zeta derivatives up to an oscillatory prefactor, see \Cref{zeta-ani-set-zeta-der},
and since
$\Delta^n$ can be written as a sum of partial derivatives,
some linear combinations of anisotropic Epstein zeta functions correspond to shifted Epstein zeta functions without anisotropy, even though the individual summands cannot be stably computed that way.
By a harmonic decomposition, that is a decomposition in terms of homogeneous harmonic polynomials, the partial derivative splits into
poly-Laplacians and factors that isolate the anisotropic contributions, where the poly-Laplacians can be applied without differentiating and the remaining differential operators have stably computable closed-form expressions when applied to the upper Crandall functions.

In \Cref{sec:harmonic-decomp} we introduce a harmonic decomposition
and show how the associated differential operators act on upper Crandall functions.
Based on this decomposition of arbitrary partial derivatives, 
we derive a Crandall type representation of the anisotropic Epstein zeta function in \Cref{sec:harmonic-aniso} 
and an analogous representation for the regularized anisotropic Epstein zeta function in \Cref{sec:harmonic-reg}.

\subsection{Harmonic decomposition}
\label{sec:harmonic-decomp}

By decomposing the partial derivative in terms of poly-Laplacians and homogeneous harmonic polynomials,  we aim to generalize \Cref{poly-laplacian} to general partial derivatives, enabling the stable computation of the anisotropic Epstein zeta function for arbitrary anisotropy.

\begin{definition}[Harmonic polynomials]
\label{harmonic-def}
A $d$-dimensional polynomial
$h\colon\mathds R^d\to\mathds C$ is called harmonic
if $\Delta h(\bm y)=0$ for every $\bm y\in\mathds R^d$.
\end{definition}

We will employ homogeneous harmonic polynomials implicitly defined by the harmonic decomposition of a monomial,
as shown in the first part of the following remark, a result that can be found in \cite[Th. 5.7]{axlerHarmonicFunctionTheory2001}.
The second part of the remark is an explicit representation
\cite[Lem. 4.10]{flaviDecompositionsPowersQuadrics2025}, which we will sometimes use in proofs to derive properties of harmonic polynomials.

\begin{remark} Let $\bm y\in\mathds R^d$.
\label{harmonic-decomp}
\begin{enumerate}
\item (Harmonic decomposition)
The monomial $\bm y^{\bm\alpha}$, $\bm\alpha\in\mathds N^d$, can be written as a sum of products of two factors: powers $|\bm y|^{2k}$, and
uniquely determined homogeneous harmonic polynomials
 $h_{k,\bm\alpha}$
 of degree $|\bm\alpha|-2k$, $k\in\mathds N$, $2k\le |\bm\alpha|$; that is to say
$$
\bm y^{\bm\alpha}
=
\sum_{k=0}^{\lfloor |\bm\alpha|/2\rfloor}
|\bm y|^{2k}h_{k,\bm\alpha}(\bm y)
\,,\qquad \bm y \in\mathds R^d\,.
$$
\item (Direct representation) Let $d\ge 2$. Each homogeneous harmonic polynomial
$h_m(\bm y)$ of degree $m\in\mathds N$
is a combination of $m$'th powers of linear forms
$(\bm a\cdot \bm y)^m$
associated with a finite number of isotropic vectors 
$V\subseteq \{\bm a\in\mathds C^d:\bm a^2=0\}$; that is to say
$$
h_m(\bm y) = \sum_{\bm a\in V}(\bm a\cdot \bm y)^mc_{\bm a}
\,,\qquad
\bm y\in\mathds R^d\,,
$$
with coefficients
$c_{\bm a}\in\mathds C$ independent of $\bm y$.
\end{enumerate}
\end{remark}

Replacing $\bm y$ by $\diffop$, the first part of \Cref{harmonic-decomp} yields the harmonic decomposition of the partial derivative
into
poly-Laplacians and factors that capture the anisotropic contributions to the derivative
$$
\diffop^{\bm\alpha}
=
\sum_{k=0}^{\lfloor |\bm\alpha|/2\rfloor}
\Delta^kh_{k,\bm\alpha}(\diffop)
\,,
$$
where the poly-Laplacian $\Delta^k$ acts on $\bm y$.
As a consequence of the second part of \Cref{harmonic-decomp}, applying the anisotropic differential operator to the Crandall functions leaves them unchanged beyond a shift of the power-law and a vector dependent pre-factor.

\begin{lemma}
\label{harmonic-crandall}
Let $\bm y\in \mathds R^d\setminus\{\bm 0\}$,
$\nu\in\mathds C$ and let $h_m$ be a $d$-dimensional homogeneous harmonic polynomial of degree $m\in\mathds N$.  Then
$$
h_{m}(\diffop)G_{\nu}(\bm y)
=
(-2\pi)^mh_m(\bm y)G_{\nu+2m}(\bm y)\,.
$$
\end{lemma}

\begin{proof}
By \Cref{harmonic-applied-to-f(y2)} and the chain rule, for $\eta\in\mathds R$,
$$
h_m(\diffop)e^{-\eta \bm y^2}
=
(-2\eta)^mh_m(\bm y)e^{-\eta\bm y^2}
$$
and by the integral representation
$$
h_m(\diffop)
G_{\nu}(\bm y)=h_m(\bm y)
\int_{-1}^{1}|t|^{-\nu}(-2\pi/t^2)^me^{-\pi\bm y^2/t^2}\frac{\mathrm{d} t}{|t|}\,,
$$
where the integral is a shifted upper Crandall function with a prefactor $(-2\pi)^m$ as stated.
\end{proof}

\subsection{Anisotropic Epstein zeta functions}
\label{sec:harmonic-aniso}

The harmonic decomposition of the differential operator yields a stable representation of the anisotropic Epstein zeta function.

\begin{theorem}[Harmonic-decomposition Crandall representation]
\label{harmonic-decomposition-crandall-representation}
\anisoreq
Then for any $\lambda>0$, it holds that
\begin{align*}
Z_{\Lambda,\nu,\bm\alpha}(\bm x,\bm y)
=
\sum_{k=0}^{\lfloor |\bm\alpha|/2\rfloor}
\frac{(\pi/\lambda^2)^{\nu/2-k}}{\Gamma(\nu/2-k)}
\Bigg[
\sum_{\bm z\in \Lambda-\bm x}
h_{k,\bm\alpha}(\bm z)G_{\nu-2k}(\bm z/\lambda)e^{-2\pi i\bm y\cdot (\bm z+\bm x)}\\
+(-1)^k
\frac{\lambda^{d+2|\bm\alpha|-4k}}{i^{|\bm\alpha|}V_{\Lambda}}
\sum_{\bm p\in \Lambda^{\ast}+\bm y}
h_{k,\bm\alpha}(\bm p)G_{d-\nu+2(|\bm\alpha|-k)}
(\lambda\bm p)e^{-2\pi i\bm x\cdot \bm p}
\Bigg]
\end{align*}
where $h_{k,\bm\alpha}$ are the homogeneous harmonic polynomials in \Cref{harmonic-decomp} and where each summand in which $h_{k,\bm\alpha}$ vanishes is understood to be zero, irrespective of the other factors.
\end{theorem}

\begin{proof}
Let $\bm y\notin \Lambda^*$.
By the identity for the poly-Laplacian derivatives for the set zeta function, see \Cref{poly-laplacian}, and the harmonic decomposition of the partial derivative we have
$$
Z^{(\bm\alpha)}_{\Lambda-\bm x,\nu}(\bm y)
=
\sum_{k=0}^{\lfloor |\bm\alpha|/2\rfloor}
(-2\pi i)^{2k}h_{k,\bm\alpha}(\diffop) Z_{\Lambda-\bm x,\nu-2k}(\bm y).
$$
and from the equivalence of the set zeta derivatives with the anisotropic Epstein zeta function up to a phase factor, see \Cref{zeta-ani-set-zeta-der}, it follows that
$$
Z_{\Lambda,\nu,\bm\alpha}(\bm x,\bm y)
=e^{-2\pi i \bm x\cdot\bm y}
\sum_{k=0}^{\lfloor |\bm\alpha|/2\rfloor}
\frac{h_{k,\bm\alpha}(\diffop)}{(-2\pi i)^{|\bm\alpha|-2k}} \Big[ 
e^{2\pi i \bm x\cdot\bm y}Z_{\Lambda,\nu-2k}(\bm x,\bm y)
\Big]\,.
$$
Applying the derivative Crandall representation for the anisotropic Epstein zeta function, \Cref{derivative-crandall-representation}, at $\bm\alpha=\bm 0$ we have that $Z_{\Lambda,\nu,\bm\alpha}(\bm x,\bm y)$ is equal to
\begin{align*}
e^{-2\pi i\bm x\cdot\bm y}\sum_{k=0}^{\lfloor |\bm\alpha|/2\rfloor}
\frac{h_{k,\bm\alpha}(\diffop)}{(-2\pi i)^{|\bm\alpha|-2k}} 
&\frac{(\pi/\lambda^2)^{\nu/2-k}}{\Gamma(\nu/2-k)}
\Bigg[
\sum_{\bm z\in \Lambda-\bm x}
G_{\nu-2k}(\bm{z}/\lambda)e^{-2\pi i\bm y\cdot \bm z}\\
&\qquad+
\frac{\lambda^d}{V_{\Lambda}}\sum_{\bm p\in\Lambda^{\ast}}
G_{d-(\nu-2k)}
(\lambda(\bm p+\bm y))e^{-2\pi i\bm x\cdot \bm p}
\Bigg]\,.
\end{align*}
Here 
we freely exchange the nowhere-vanishing oscillatory prefactor $e^{\pm 2\pi i \bm x\cdot\bm y}$ with the two series whose convergence remains unchanged.
Analogously to the reasoning in the proof of \Cref{derivative-crandall-representation}, we may exchange the differential operator with the sum due to the superexponential decay of the upper Crandall function, see \Cref{crandall-upper-prop}.
For the first sum note that for any $d$-dimensional homogeneous polynomial $h_m(\bm y)$ of degree $m\in\mathds N$ we have
$$
\frac{h_m(\diffop)}{(-2\pi i)^{m}}e^{-2\pi i\bm y\cdot \bm z}
=
h_m(\bm z)e^{-2\pi i\bm y\cdot \bm z}
$$
which follows from applying all partial derivatives sequentially and noting that the degree of $h_{k,\bm\alpha}$ is $|\bm\alpha|-2k$.
For the second sum let $\bm u=\lambda(\bm p+\bm y)$ and use that by the chain rule
$$
\frac{h_{k,\bm\alpha}(\diffop)}{(-2\pi i)^{|\bm\alpha|-2k}}G_{d-(\nu-2k)}(\bm u)
=(-1)^k\frac{\lambda^{2|\bm\alpha|-4k}}{i^{|\bm\alpha|}}
h_{k,\bm\alpha}(\bm p+\bm y) G_{d-\nu+2(|\bm\alpha|-k)}(\bm u)
$$
so that the equation in the theorem follows.

Let $\Re{\nu}>d+|\bm\alpha|$ and let $\bm y\in\Lambda^*$.
Since $\Lambda^*+\bm y=\Lambda^*$, it
is enough to consider the $\bm p$-dependence
$$
h_{k,\bm\alpha}(\bm p)G_{d-\nu+2(|\bm\alpha|-k)}(\lambda \bm p) e^{-2\pi i \bm x\cdot\bm p}
$$
of the summand in the second series
at $\bm p=\bm 0$.
Note $h_{k,\bm\alpha}(\bm p)$ is homogeneous of degree $m=|\bm\alpha|-2k$ and
for $t=d-\nu+2(|\bm\alpha|-k)$ we have $\Re t-m=d+|\bm\alpha|-\Re\nu<0$
so that by \Cref{continuity-crandall-upper} the display above continuously extends to $\bm p= \bm 0$.
By \Cref{h-vanish} we have that
$h_{k,\bm\alpha}(\bm 0)=0$ if
$\chi(\bm\alpha)=0$ or $2k\neq |\bm\alpha|$, so the limit of the product above vanishes unless $2k=|\bm\alpha|$, in which case, with all prefactors, it equals
$$
-2
(-1)^{|\bm\alpha|/2}\chi(\bm\alpha)
\frac{(\pi/\lambda^2)^{(\nu-|\bm\alpha|)/2}}{\Gamma((\nu-|\bm\alpha|)/2)}
\frac{\lambda^{d}}{i^{|\bm\alpha|}V_{\Lambda}}
\frac{h_{\lfloor|\bm\alpha|/2\rfloor,\bm\alpha}(\bm 0)}{(d+|\bm\alpha|-\nu)}\,,
$$
which is meromorphic in $\nu\in\mathds C$
with a simple pole at $\nu=d+|\bm\alpha|$ if $\chi(\bm\alpha)=1$.
Both the prefactor $(\pi/\lambda^2)^{\nu/2-k}/\Gamma(\nu/2-k)$ and the first sum are entire in $\nu$, in particular the singularity of the real-space summand in $\bm x\in\Lambda$ gets removed since
$$
\frac{G_{\nu-2k}(\bm 0)}{\Gamma(\nu/2-k)}
=-\frac{1}{\Gamma(\nu/2-k+1)}\,.
$$
For $\Re{\nu}>d+|\bm\alpha|$ both sides of the statement are continuous in $\bm y\in\mathds R^d$, 
the left-hand side by \Cref{zetadirectsum} and the right-hand side by the extension just shown together with the locally uniform convergence of both series as in the proof of \Cref{hol}.
For $\bm y\in\Lambda^*$, both sides are holomorphic in $\nu$ in the region as stated, the left-hand side
by the holomorphy of the anisotropic Epstein zeta function, see \Cref{hol}, and the right-hand side preceding discussion, so that the statement follows by the identity theorem.
\end{proof}

The main result allows for the stable and efficient evaluation of anisotropic Epstein zeta functions, given a suitable implementation of the harmonic polynomials.
We demonstrate the resulting stability in \Cref{sec:num}
and provide a backward stable algorithm for the harmonic polynomials  in \Cref{sec:alg}.

\subsection{Regularized anisotropic Epstein zeta functions}
\label{sec:harmonic-reg}

We derive a stable representation for the regularized anisotropic Epstein zeta function, based on the harmonic decomposition of the partial derivative.
We begin by showing how the poly-Laplacian and the differential operator obtained from the harmonic polynomials act on the power-law singularity, the proof of which can be found in \Cref{sec:appendix-reg}.

\begin{lemma}
\label{harmonic-decomp-sing}
Let $h_m$ denote a homogeneous harmonic polynomial of degree $m\in\mathds N$,
let $\bm y\in\mathds R^d\setminus\{\bm 0\}$, and let
 $k\in\mathds N$.
Let $\nu\in\mathds C\setminus(d+2(k+m+\mathds N))$, then
$$
h_m(\diffop)\Delta^k 
\hat{s}_\nu(\bm y)
=
\frac{2^k(-2\pi^2)^{m+k}}{(\nu/2-k-1)_{m}}
h_m(\bm y)
\hat{s}_{\nu-2(m+k)}(\bm y)\,,
$$
in the sense of the limit in $\nu$ if $(\nu/2-k-1)_{m}=0$.
Let $\nu=d+2\ell$ for
 $\ell\in\mathds N$, $\ell\ge k+m$,
then the right-hand side is multiplied by
$$
\Big(1+\frac{H_{\ell,k,m+2k,d}}{\log(\pi\bm y^2)}
\Big)
$$
where 
$H_{\ell,k,n,d}=H_{\ell}-H_{\ell+k-n}+\psi(\ell+d/2)-\psi(\ell+d/2-k)$.
Here, $H_{\ell}$ denotes the $\ell$'th harmonic number and $\psi(\cdot)=\Gamma'(\cdot)/\Gamma(\cdot)$ denotes the digamma function.
\end{lemma}

The Crandall type representation for the regularized anisotropic Epstein zeta function that utilizes the harmonic decomposition of the singularity reads as follows.

\begin{theorem}[Harmonic-decomposition Crandall representation of the regularized anisotropic Epstein zeta function]
Let $\Lambda$ be a $d$-dimensional lattice, let $\bm x \in\mathds{R}^d$, let
$\bm y \in (\mathds R^d\setminus\Lambda^*)\cup\{\bm 0\}$, let $\nu \in \mathds{C}$ and let $\bm\alpha\in\mathds N^d$.
Then for $\lambda > 0$ it holds that
\begin{align*}
Z^{\rm reg}_{\Lambda,\nu,\bm\alpha}(\bm x,\bm y)
&=
\sum_{k=0}^{\lfloor |\bm\alpha|/2\rfloor}
\frac{(\pi/\lambda^2)^{\nu/2-k}}{\Gamma(\nu/2-k)}
\Bigg[
\sum_{\bm z\in \Lambda-\bm x}
h_{k,\bm\alpha}(\bm z)G_{\nu-2k}(\bm{z}/\lambda)e^{-2\pi i\bm y\cdot \bm z}
\\
&\qquad\quad+(-1)^k
\frac{\lambda^{d}}{i^{|\bm\alpha|}V_{\Lambda}}
\bigg[
h_{k,\bm\alpha}(\bm y)G^{\mathrm{reg}}_{d-\nu,\lambda,k,|\bm\alpha|}(\bm y)
\\
&\qquad\qquad\quad
+\lambda^{2|\bm\alpha|-4k}\sum_{\substack{\bm p\in\Lambda^{\ast}+\bm y\\ \bm p\neq \bm y}}
h_{k,\bm\alpha}(\bm p)G_{d-\nu+2(|\bm\alpha|-k)}
(\lambda \bm p)e^{-2\pi i\bm x\cdot(\bm p-\bm y)}
\bigg]
\Bigg].
\end{align*}
If $t\in\mathds C\setminus(-2\mathds N)$ or $t=-2\ell$ with $\ell\in\mathds N$ and $\ell<n-k$, the regularized Crandall function for the harmonic method is defined as
$$
G^{\mathrm{reg}}_{t,\lambda,k,n}(\bm y)
=-\lambda^{2n-4k}g_{t+2(n-k)}(\lambda \bm y)\,.
$$
If $\ell\in\mathds N$, $\ell\ge n-k$, we define
$$
G^{\mathrm{reg}}_{-2\ell,\lambda,k,n}(\bm y)
=
\lambda^{2n-4k}
G_{2(n-\ell-k)}(\lambda \bm y)
+
\lambda^{2(\ell-k)}
\frac{ (- \pi \bm y^2 )^{\ell+k-n} }{(\ell+k-n)!}
\Big(
\log (\pi  \bm y^{2})
+
H_{\ell,k,n,d}\Big)
$$
where the following representation is numerically stable around $\bm y = \bm 0$
$$
G^{\mathrm{reg}}_{-2\ell,\lambda,k,n}(\bm y)
=
\lambda^{2(\ell-k)}
\frac{(-\pi\bm y^2)^{\ell+k-n}}{(\ell+k-n)!}H'_{\ell,\lambda,k,d}
-\lambda^{2n-4k}\!\!\!\!\sum_{\substack{j=0 \\ j\neq \ell+k-n}}^{\infty}\frac{(-\pi\lambda^2 \bm y^2)^j}{(j+n-\ell-k)j!}
$$
Here,  
$H_{\ell,k,n,d}$ is defined as in \Cref{harmonic-decomp-sing}
and
$
H'_{\ell,\lambda,k,d}=H_{\ell}+\psi(\ell+d/2)-\psi(\ell+d/2-k)-\gamma -2\log\lambda
$
with the digamma function $\psi$ and the harmonic numbers $H_\ell$.
\end{theorem}

\begin{proof}
By \Cref{harmonic-decomp} we can write
$$
\hat s^{(\bm\alpha)}_{\nu}(\bm y)
=
\sum_{k=0}^{\lfloor|\bm\alpha|/2\rfloor}h_{k,\bm\alpha}(\diffop)\Delta^k \hat s_{\nu}(\bm y)\,.
$$
Subtracting the singularity from the $\bm p=\bm y$ summand in the second sum of \Cref{harmonic-decomposition-crandall-representation}, we have
\begin{align*}
&Z_{\Lambda,\nu,\bm\alpha}^{\mathrm{reg}}(\bm x,\bm y)
=
\sum_{k=0}^{\lfloor |\bm\alpha|/2\rfloor}
\frac{(\pi/\lambda^2)^{\nu/2-k}}{\Gamma(\nu/2-k)}
\Bigg[
\sum_{\bm z\in \Lambda-\bm x}
h_{k,\bm\alpha}(\bm z)G_{\nu-2k}(\bm{z}/\lambda)e^{-2\pi i\bm y\cdot \bm z}
\\
&
+\frac{(-1)^k\lambda^{d}}{V_{\Lambda}i^{|\bm\alpha|}}\bigg[h_{k,\bm\alpha}(\bm y)f_{\nu}(\bm y)+
\lambda^{2|\bm\alpha|-4k}\!\!\!\!
\sum_{\substack{\bm p\in\Lambda^{\ast}+\bm y\\ \bm p\neq \bm y}}
h_{k,\bm\alpha}(\bm p)G_{d-\nu+2(|\bm\alpha|-k)}
(\lambda \bm p)e^{-2\pi i\bm x\cdot(\bm p-\bm y)}
\bigg]
\Bigg]
\end{align*}
where we, for $\bm y\neq \bm 0$ and $h_{k,\bm\alpha}(\bm y)\neq  0$, formally define $f_{\nu}(\bm y)$ as the difference between the $\bm p=\bm y$ summand and the singularity with all the factors as
$$
\lambda^{2|\bm\alpha|-4k}
G_{d-\nu+2(|\bm\alpha|-k)}(\lambda \bm y)
-(-1)^k
\frac{(-2\pi)^{-|\bm\alpha|}}{h_{k,\bm\alpha}(\bm y)}
\frac{\Gamma(\nu/2-k)}{\lambda^d(\pi/\lambda^2)^{\nu/2-k}}
h_{k,\bm\alpha}(\diffop)
\Delta^k 
\hat s_{\nu}(\bm y)
$$
and smoothly extend to $\bm y=\bm 0$ and $h_{k,\bm\alpha}(\bm y)=\bm 0$.
The statement follows from showing that
$
f_{\nu}
=
G_{d-\nu,\lambda,k,|\bm\alpha|}^{\rm reg}\,.
$

Let $\nu\notin(d+2(k+m+\mathds N))$,
where $m=|\bm\alpha|-2k$ is the degree of $h_{k,\bm\alpha}$, then by \Cref{harmonic-decomp-sing}
$$
f_{\nu}(\bm y)
=
\lambda^{2|\bm\alpha|-4k}
\Big(
G_{d-\nu+2(|\bm\alpha|-k)}(\lambda \bm y)
-
\frac{\Gamma((d-\nu)/2+|\bm\alpha|-k)}{(\pi \lambda^2 \bm y^2)^{(d-\nu)/2+|\bm\alpha|-k}}
\Big)
$$
so $f_{\nu}=G^{\rm reg}_{d-\nu,\lambda,k,|\bm\alpha|}$ by the fundamental relation, see \Cref{crandall-lower-prop}.

Let $\ell\in\mathds N$, $\ell \ge m+k$, then
by \Cref{harmonic-decomp-sing}
$$
f_{d+2\ell}(\bm y)=
\lambda^{2|\bm\alpha|-4k}\Big(
G_{2(|\bm\alpha|-\ell-k)}(\lambda \bm y)
+\frac{ (- \pi \lambda^2\bm y^2 )^{\ell+k-|\bm\alpha|} }{(\ell+k-|\bm\alpha|)!}
\Big(
\log (\pi  \bm y^{2})
+
H_{\ell,k,|\bm\alpha|,d}\Big)
\Big)
\,.
$$

To obtain a stable representation around $\bm y=\bm 0$
we follow the proof of \Cref{low-ders-aniso-reg}.
To that end, let 
$p=\ell+k-|\bm\alpha|$ and $z=\pi\lambda^2\bm y^2$, so
that the equation above is given by
$$
f_{d+2\ell}(\bm y)=
\lambda^{2|\bm\alpha|-4k}
z^p\Big(\Gamma(-p,z)
+\frac{(-1)^{p}}{p!} (\log z-\log\lambda^2+H_{\ell,k,|\bm\alpha|,d})\Big)
$$
and the series representation
\cite[Eq. (5.1.12), (6.5.9)]{abramowitz} as in the proof of \Cref{low-ders-aniso-reg} yields the last equality of the statement.
Since  $f_{\nu}(\bm y)$, $\nu\in\mathds C$, $\bm y\in\mathds R^d$ is independent of $h_{k,\bm\alpha}(\bm y)$ and continuous in $\bm y\in\mathds R^d$, see the holomorphicity of the lower Crandall function in \Cref{crandall-lower-prop} and use the continuity of power series,
the representation holds for $h_{k,\bm\alpha}(\bm y)= 0$ and $\bm y=\bm 0$. 
\end{proof}

This Crandall type representation for the regularized anisotropic Epstein zeta function, obtained by the harmonic decomposition of the derivative operator, enables efficient computations for high-order anisotropy.

\section{Algorithm}
\label{sec:alg}

In this section we present an algorithm for stably computing anisotropic Epstein zeta functions for arbitrary anisotropy, based on the harmonic-decomposition Crandall representation.
\Cref{alg} is a generalization of the algorithm for the computation of the Epstein zeta function introduced in \cite[Sec. 3]{buchheitComputationPropertiesEpstein2026}; many details carry over to the anisotropic case.
In particular, a thorough discussion regarding the preprocessing of the arguments and the details of the computation of incomplete Gamma functions can be found there.
In what follows, we focus on the most relevant case of real parameters, we set the value of the splitting parameter to one, $\lambda = 1$, and solve the additional numerical challenges of stably evaluating
anisotropic Epstein zeta functions, in particular the backward stable evaluation of harmonic polynomials and the handling of cancellation errors for large exponents.
In \Cref{sec:harmonic-computation} we derive  an explicit representation of the harmonic polynomials, where the coefficients are given by sums over dyadic rationals that can be computed exactly.
In \Cref{sec:cancellation} we show how to avoid cancellation in the real space sum for large exponents
and 
 near the shifted poles of the Gamma or upper Crandall function.
Also note the discussion of suitable truncation parameters in \Cref{sec:appendix-trunc}.

\begin{algorithm}
\small
\DontPrintSemicolon
\caption{Computation of anisotropic Epstein zeta functions.}
\KwIn{$d\in\mathds{N}_+$, $A\in \mathds{R}^{d\times d}$ regular, $\bm x,\bm y\in\mathds{R}^d$, $\nu\in\mathds{R}$, $\bm\alpha\in\mathds N^d$, $r>0$}

\BlankLine
\underline{\bfseries 1. Rescale lattice matrix and vectors
$\phantom{j}$
}

$a = V_{\Lambda}^{1/d},\quad A \gets A/a,\quad \bm x \gets \bm x/a,\quad \bm y \gets a\bm y$

\BlankLine
\underline{\bfseries 2. Project $\bm x$ and $\bm y$ to elementary lattice cells}

$\tilde{\bm x} = A^{-1}\bm x,\quad \tilde{\bm y} = A^{T}\bm y$

\For{$j = 1,\ldots,d$}{
    $\tilde{x}^{(j)} \gets \bigl((\tilde{x}^{(j)}+1/2)\bmod 1\bigr)-1/2$

    $\tilde{y}^{(j)} \gets \bigl((\tilde{y}^{(j)}+1/2)\bmod 1\bigr)-1/2$
}

$\tilde{\bm x} \gets A\tilde{\bm x},\quad \tilde{\bm y} \gets A^{-T}\tilde{\bm y}$

\BlankLine
\underline{\bfseries 3. Handle pole and symmetry related zeros}

\If{$\nu = d+|\bm\alpha|\;\land\;\tilde{\bm y} = \bm 0 \land \chi(\bm\alpha)=1$}{
    \Return NaN
}

\If{$|\bm\alpha|\;\text{odd}\;\land\;\tilde{\bm x}=\bm 0\;\land\;2\tilde{\bm y}\in\Lambda^*$}{
    \Return $0$
}

\If{$\exists j:\ z^{(j)}\mapsto -z^{(j)}\;\text{is a symmetry of}\;\Lambda\;\land\;\alpha^{(j)}\;\text{odd}$
$\phantom{..}\land \big((2\tilde{x}^{(j)}\bm e_j\in\Lambda\;\land\;
 (\Lambda\cap\mathds R\bm e_j)\cdot\tilde{\bm y}\subseteq\mathds Z)\;\lor\;(2\tilde{y}^{(j)}\bm e_j\in\Lambda^*\;\land\;
 (\Lambda^*\cap\mathds R\bm e_j)\cdot\tilde{\bm x}\subseteq\mathds Z)\big)$%
 }{
    \Return $0$
}

\BlankLine
\underline{\bfseries 4. Summation in real and reciprocal space}

$\textnormal{result}=0$

$\textnormal{choose\_sing\_branch}=\nu> 10 \land |\bm\alpha|\ge 2$

\For{$k=0,1,\ldots,\lfloor |\bm\alpha|/2\rfloor$}{

     $\textnormal{sum\_real}=0,\quad \textnormal{sum\_reci}=0,\quad\textnormal{sum\_total}=0$

    \If{$\nu\in2k-2\mathds N$}{
        \If{$\tilde{\bm x}= \bm 0\land \nu=2k=|\bm\alpha|\land \chi(\bm\alpha)=1$}{
        $\textnormal{sum\_total}=-h_{k,\bm\alpha}(\bm 0)e^{-2\pi i \bm x\cdot\bm y}$
        }
    }
    \Else{
        \For{$\bm v \in \mathds{Z}^d:\, \bm z=A\bm v-\tilde{\bm x},\,|\bm z|\le r$}{

            \If{$\textnormal{choose\_sing\_branch}\land\, \|\bm v\|_{1}\le 1$}{
                $\textnormal{crandall}=-g_{\nu-2k}(\bm z)$
            }
            \Else{
                $\textnormal{crandall}=G_{\nu-2k}(\bm z)$
            }
            $\textnormal{sum\_real} \pluseq h_{k,\bm\alpha}(\bm z)\exp(-2\pi i\,A\bm v\cdot\tilde{\bm y})\,\textnormal{crandall}$
            \quad\tcp{Kahan}   
        }

        $\textnormal{sum\_real}\asteq \exp(-2\pi i(\bm x-\tilde{\bm x})\cdot\bm y)$
    
        \BlankLine
        \For{$\bm p \in (A^{-T}\mathds{Z}^d+\tilde{\bm y}):\,|\bm p|\le r,\, h_{k,\bm\alpha}(\bm p)\neq 0$}{
            
            $\textnormal{sum\_reci}\pluseq
            h_{k,\bm\alpha}(\bm p)\exp(-2\pi i\,\bm x\cdot \bm p)G_{d-\nu+2(|\bm\alpha|-k)}(\bm p)
            $\quad\tcp{Kahan}
        }
        
        $\textnormal{sum\_reci}\asteq (-1)^k\, i^{-|\bm\alpha|}$
        
        $\textnormal{sum\_total}=(\textnormal{sum\_real}+\textnormal{sum\_reci})/\Gamma(\nu/2-k)$
    }
    
    $\textnormal{result}\pluseq \pi^{(\nu-2k)/2}\,\textnormal{sum\_total} $
}
\If{$\textnormal{choose\_sing\_branch}$}{

    $\textnormal{sum\_sing\_outer}=0$
    \For{$\bm{v} \in \{-1,0,1\}^d:\bm z=A\bm v-\tilde{\bm x}\neq \bm 0,\,\|\bm v\|_1\le 1$}{

        $\textnormal{sum\_sing\_inner}=0$
        
        \For{$k=0,1,\ldots\lfloor|\bm\alpha|/2\rfloor$}{

            \If{$\nu\notin 2k-2\mathds N$}{
                $\textnormal{sum\_sing\_inner}\pluseq
                   |\bm z|^{2k}\,h_{k,\bm\alpha}(\bm z)
                $\quad\tcp{Kahan}
            }
            
        }
        $\textnormal{sum\_sing\_outer}\pluseq \exp(-2 \pi i A\bm v\cdot\tilde{\bm y})\,|\bm z|^{-\nu}\,\textnormal{sum\_sing\_inner}$\quad\tcp{Kahan}
    }

    $\textnormal{sum\_sing\_outer}\asteq \exp(-2\pi i(\bm x-\tilde{\bm x})\cdot\bm y)$

    $\textnormal{result}\pluseq \textnormal{sum\_sing\_outer}$

}

$\textnormal{result}\asteq a^{|\bm\alpha|-\nu}$

\BlankLine
\Return $\textnormal{result}$
\label{alg}
\end{algorithm}

\subsection{Harmonic polynomials}
\label{sec:harmonic-computation}

We begin by deriving a stably computable representation of the harmonic polynomials as defined by the harmonic decomposition, \Cref{harmonic-decomp};
the representation is
based on \cite{axlerHarmonicFunctionTheory2001} and a proof is provided in \Cref{sec:appendix-alg}.

\begin{theorem}[Efficient computation of harmonic polynomials]
\label{harmonic-comp}
Let $\bm y\in\mathds R^d$, let $\bm\alpha\in\mathds N^d$
and let $h_{k,\bm\alpha}$ be
the harmonic homogeneous polynomials
of degree $|\bm\alpha|-2k$, $k\in\mathds N$, $2k\le |\bm\alpha|$
implicitly defined in \Cref{harmonic-decomp} by
$$
\bm y^{\bm\alpha}
=
\sum_{k=0}^{\lfloor |\bm\alpha|/2\rfloor}
|\bm y|^{2k}h_{k,\bm\alpha}(\bm y)
\,,\qquad \bm y \in\mathds R^d\,.
$$
Then, the harmonic polynomials admit the representation
$$
h_{k,\bm\alpha}(\bm y)=
\omega_{k,|\bm\alpha|}
\!\!\!
\sum_{\substack{
2\bm\gamma\ge\bm\alpha
\\
|\bm\gamma|=|\bm\alpha|-k
}}
\!\!\!\!
c_{k,\bm\alpha,\bm\gamma}
\;
\bm y^{2\bm\gamma-\bm\alpha}\,.
$$
Here, the coefficients are the dyadic rationals
$$
c_{k,\bm\alpha,\bm\gamma}
=
\!\!\!\!\!\!\!
\sum_{\substack{
\bm\beta\ge\bm0
\\
\bm 0 
\le 
\bm\alpha+\bm\beta-\bm\gamma
\le \bm\alpha/2
}}
\!\!\!\!\!\!
(-2)^{-|\bm\beta|}
\,
(\bm\gamma-\bm\beta)_{\bm\alpha+\bm\beta-\bm\gamma}
\,
q_{k,|\bm\alpha|,|\bm\beta|}
\,
\frac{|\bm\beta|!}{\bm\beta!}
\binom{|\bm\beta|+k}{k}
\binom{\bm\alpha}{\bm\alpha+\bm\beta-\bm\gamma}
\,,
$$
and the global constant is defined by
$$
\frac{1}{\omega_{k,n}}
=
2^k q_{k,n,0}\prod_{j=1}^k (2n+d-2k-2j)
$$
for $n\in\mathds N$, $n\ge2k$,
where 
$$
q_{k,n,m}
=
\prod_{j=m+1}^{\lfloor n/2\rfloor -k}(2n+d-2-4k-2j)
$$
for $m\in\mathds N$, $m\le \lfloor n/2\rfloor -k$.
\end{theorem}

The stability of \Cref{alg} rests on the backward stable evaluation of the harmonic polynomials of the harmonic decomposition in \Cref{harmonic-decomp}.
Further experiments show that backward stability can be achieved
by exactly computing the coefficients 
$c_{k,\bm\alpha,\bm\gamma}$ 
through exact addition of dyadic rationals demonstrated in \Cref{sec:num-harmonic}.
These coefficients are independent of the vector arguments $\bm y$ that appear as summation variables in \Cref{alg}.
Therefore the computational per-evaluation cost of the harmonic polynomials can be greatly reduced by computing all possible coefficients once, see \Cref{tab:harmonic_high_order}.

\subsection{Numerical stability}
\label{sec:cancellation}

In \Cref{alg} we employ Kahan summation \cite{kahan} to significantly reduce roundoff errors; the details can be found in \cite[Sec. 3.4]{buchheitComputationPropertiesEpstein2026} where the numerical stability of the Crandall representation for the Epstein zeta function, which corresponds to the harmonic-decomposition Crandall representation for $\bm\alpha=\bm 0$, is demonstrated in detail.
Structurally, the stability of the harmonic-decomposition Crandall representation, \Cref{harmonic-decomposition-crandall-representation}, is similar to a sum of $k=0,\ldots,\lfloor |\bm\alpha|/2\rfloor$ derivative Crandall representations with vanishing anisotropy $\bm\alpha$.
We will demonstrate in the next section that this representation is stable for high-order anisotropy, where only a few special cases, that will be discussed here, need special consideration.

For large exponents $\nu>10$ and $|\bm\alpha|\ge 2$, 
a small number of summands near the origin of the real space sum demand special care.
There, adding the summands near $\bm z = \bm 0$ for different values of $k$
causes catastrophic cancellation errors in the resulting anisotropic Epstein zeta function.
In contrast, the derivative Crandall representation, \Cref{derivative-crandall-representation}, does not suffer from this particular error but introduces systematic errors in each individual upper Crandall derivative for high-order anisotropy.
For large exponents we employ the following lemma where the singularity is extracted and its summation order is changed, where a proof can be found in \Cref{sec:appendix-alg}.

\begin{lemma}
\label{high-exp}
Let $\Lambda$
be a $d$-dimensional
lattice, let $\bm\alpha\in\mathds N^d$, let $\bm x,\bm y\in\mathds R^d$ and let $\nu\in\mathds C$.
Let $M\subseteq \Lambda-\bm x$ be bounded and let $\lambda>0$. 
Then it holds that
\begin{align*}
&\sum_{\substack{k=0\\ k\neq \nu/2}}^{\lfloor |\bm\alpha|/2\rfloor}
\frac{(\pi/\lambda^2)^{(\nu-2k)/2}}{\Gamma((\nu-2k)/2)}
\sum_{\bm z\in M}
h_{k,\bm\alpha}(\bm z)G_{\nu-2k}(\bm{z}/\lambda)e^{-2\pi i\bm y\cdot (\bm z+\bm x)}
\\
=&\sums_{\substack{\bm z\in M}}
|\bm z|^{-\nu}
e^{-2\pi i\bm y\cdot (\bm z+\bm x)}
\sum_{\substack{k=0\\ k\notin \nu/2+\mathds N}}^{\lfloor |\bm\alpha|/2\rfloor}
h_{k,\bm\alpha}(\bm z)
|\bm z|^{2k}
\\
&-\sum_{\substack{k=0\\ k\notin \nu/2+\mathds N}}^{\lfloor |\bm\alpha|/2\rfloor}
\frac{(\pi/\lambda^2)^{(\nu-2k)/2}}{\Gamma((\nu-2k)/2)}
\sum_{\bm z\in M}
h_{k,\bm\alpha}(\bm z)g_{\nu-2k}(\bm{z}/\lambda)e^{-2\pi i\bm y\cdot (\bm z+\bm x)}\,.
\end{align*}
\end{lemma}

In \Cref{high-exp} the $2k=\nu$ summand is excluded due to the singularities of the lower Crandall function. The limit in $\nu$ of the left hand side of that statement is well defined and evaluated in \Cref{alg} as shown in the following lemma, where a proof is provided in \Cref{sec:appendix-alg}.

\begin{lemma}
\label{special-k-sum}
\anisoreq 
For $k\in \{0,1,\ldots,\lfloor |\bm\alpha|/2\rfloor\}$ consider an individual $k$-summand in the harmonic-decomposition Crandall representation, \Cref{harmonic-decomposition-crandall-representation}, as a function of $\nu$, that is, for $\lambda>0$, consider the function
\begin{align*}
\nu\mapsto
&\frac{(\pi/\lambda^2)^{\nu/2-k}}{\Gamma(\nu/2-k)}
\Bigg[
\sum_{\bm z\in \Lambda-\bm x}
h_{k,\bm\alpha}(\bm z)G_{\nu-2k}(\bm z/\lambda)e^{-2\pi i\bm y\cdot (\bm z+\bm x)}\\
&+(-1)^k
\frac{\lambda^{d+2|\bm\alpha|-4k}}{i^{|\bm\alpha|}V_{\Lambda}}
\sum_{\bm p\in \Lambda^{\ast}+\bm y}
h_{k,\bm\alpha}(\bm p)G_{d-\nu+2(|\bm\alpha|-k)}
(\lambda\bm p)e^{-2\pi i\bm x\cdot \bm p}
\Bigg]\,,
\end{align*}
where
each summand in which $h_{k,\bm\alpha}$ vanishes is understood to be zero, irrespective of the other factors.
Then, this function is well defined at $\nu\in 2k-2\mathds N$ in the sense of the limit.
This limit 
is equal to $-h_{k,\bm\alpha}(\bm 0)e^{-2\pi i\bm x\cdot \bm y}$ if
$\nu= |\bm\alpha|=2k$, $\chi(\bm\alpha)=1$ and $\bm x\in\Lambda$
and vanishes otherwise.
\end{lemma}

In the following corollary we recover the trivial zeros of the anisotropic Epstein zeta function, which is the generalization of
\cite[Lem. 13]{buchheitComputationPropertiesEpstein2026} 
to the anisotropic case.

\begin{corollary}
\label{cor-special-case-early-return}
Let $\Lambda$ be a $d$-dimensional lattice, let $\ell\in\mathds N$, let $\bm x,\bm y\in \mathds R^d$ and let $\bm\alpha\in\mathds N^d$. 
Then
$$
Z_{\Lambda,-2\ell,\bm\alpha}(\bm x,\bm y)
=
\begin{cases}
   - e^{-2\pi i \bm x\cdot\bm y} 
    & \ell=0,\ |\bm\alpha|=0,\text{ and } \bm x\in\Lambda, 
    \\
0 &\text{else}.
\end{cases}
$$
\end{corollary}
\begin{proof}
Let $\bm\alpha=\bm 0$.
Then the anisotropic Epstein zeta function coincides with the Epstein zeta function so the statement holds by
\cite[Lem. 13]{buchheitComputationPropertiesEpstein2026}.
Let $\bm\alpha\neq\bm 0$ and $\nu=-2\ell$.
Then by \Cref{special-k-sum} all $k$-summands
in the harmonic-decomposition Crandall representation, \Cref{harmonic-decomposition-crandall-representation}, vanish
since, for all $k\in \{0,1,\ldots,\lfloor |\bm\alpha|/2\rfloor\}$ we have $\nu\in 2k-2\mathds N$ and  $\nu\le 0\neq |\bm\alpha|$.
\end{proof}

The condition for the first early zero-return statement in the third step of  \Cref{alg} is equivalent to the conditions on the inversion related symmetry zero in \Cref{syms-zero},
since $\tilde{\bm x}=\bm 0$ is equivalent to $\bm x\in\Lambda$ and $2\tilde{\bm y}\in\Lambda^*$ to $2\bm y\in\Lambda^*$.
The condition for the second early zero-return statement matches the conditions on the mirror symmetry related zeros in \Cref{syms-zero}
due to the following statement.

\begin{lemma}
\label{equv-sym-zeros}
Let $j\in\{1,\ldots,d\}$, let $\Lambda$ be a $d$-dimensional lattice that is mirror-symmetric in the $j$'th component, and let $\bm u\in\mathds R^d$.
Then 
$(\Lambda^*\cap\mathds R\bm e_j)\cdot \bm u\subseteq\mathds Z$
if and only if
there exists some $\bm v\in \Lambda$ such that $u^{(j)}+v^{(j)}=0$.
Both conditions of this equivalence are invariant under $\bm u\mapsto \bm u+\bm w$, $\bm w\in\Lambda$, so in particular under the projection to the elementary lattice cell in the second step of \Cref{alg},
and the statement holds with $\Lambda$ and $\Lambda^*$ interchanged.
\end{lemma}

Note that \Cref{equv-sym-zeros} concerns the equality of the second non-translationally invariant conditions in \Cref{syms-zero}.
Since $2x^{(j)}\bm e_j\in\Lambda$ if and only if $2(x^{(j)}+v^{(j)})\bm e_j\in\Lambda$, $\bm v\in\Lambda$, the third step of the algorithm is a valid application of \Cref{syms-zero}.
Since the special cases of \Cref{cor-special-case-early-return}
are already covered by \Cref{special-k-sum},
there is no need to catch these zeros in an early return statement in \Cref{alg}.
We finally note that special care needs to be taken when evaluating \Cref{alg} very close to symmetry-related zeros of \Cref{syms-zero}, where the error grows with the inverse distance to the zero and with the anisotropy order
which can be mitigated by appropriately recentering the summation shells.
This is not pursued further here in favor of a symmetric addition of the terms $\bm z$ and $-\bm z$ enhancing the stability near $\bm x=\bm y=\bm 0$.

\section{Numerical Experiments}
\label{sec:num}

This section provides a detailed analysis of the error, the stability and the runtime of our algorithm for the anisotropic Epstein zeta function and the harmonic polynomials for an extensive set of real parameters.
To the best of our knowledge, no other implementation of the anisotropic Epstein zeta function for general real arguments in $d>1$ dimensions exists.
Furthermore, we are unaware of any analytic representations of anisotropic Epstein zeta functions for non-trivial arguments $|\bm\alpha|>0$ and $\nu\in\mathds R$.

To reliably benchmark the error of our algorithm,
in \Cref{sec:num-prec}
we exploit symmetries in special cases where anisotropic Epstein zeta functions reduce to shifted Epstein zeta functions and compare to values obtained by direct summation for large exponents, and we derive a new representation for a certain eight-dimensional anisotropic Epstein zeta function in terms of a Dirichlet series of the Ramanujan tau function and compare against reference values obtained by a high-precision Mathematica prototype.
In \Cref{sec:num-stab}
we showcase the stability of the algorithm in terms of various real parameters $\nu\in\mathds R$, $\bm x,\bm y\in\mathds R^d$ and $\bm\alpha\in\mathds N^d$, most notably the stability in the order of the anisotropy $|\bm\alpha|$.
Furthermore we show the linear evaluation of anisotropic Epstein zeta functions for the most relevant case of low- to mid-order anisotropy in \Cref{sec:num-time}.
The stability and runtime of our algorithm for the computation of the harmonic polynomials based on the representation in \Cref{harmonic-comp} is shown in \Cref{sec:num-harmonic}.
Finally we present an application of anisotropic Epstein zeta functions, as coefficients of the singular Euler-Maclaurin expansion in a non-trivial 3D geometry in \Cref{sec:num-sem}.

\subsection{Machine precision for real parameters against independent reference values}
\label{sec:num-prec}

Let $E$ denote the minimum of the absolute and relative error
$$
E=\min \{E_{\rm{abs}},E_{\rm rel}\}\,,
$$
of our algorithm in comparison with high-precision reference values.
One example where such values can be obtained is by the representation for the sum of anisotropic Epstein zeta functions corresponding to applying the poly-Laplacian to set zeta derivatives, see \Cref{poly-laplacian}.
These sums can be exactly computed as a shifted Epstein zeta function, that is, 
in particular for $n\in\mathds N$, a lattice $\Lambda$, $\bm x,\bm y\in\mathds R^d$ and $\nu\in\mathds R$ such that $\nu\neq d+2n$ if $\bm y\in\Lambda^*$ we have
$$
\sum_{\substack{\bm\alpha\ge\bm 0 \\ |\bm\alpha|=n}}
\frac{n!}{\bm\alpha!}
Z_{\Lambda,\nu,2\bm\alpha}(\bm x,\bm y)
=
Z_{\Lambda,\nu-2n}(\bm x,\bm y)
\,.
$$
where by \Cref{zeta-ani-set-zeta-der} we expressed the set zeta derivatives in terms of anisotropic Epstein zeta functions.
In \Cref{tab:poly_laplace} we compute the error when representing the shifted Epstein zeta function through a sum of anisotropic Epstein zeta functions
over a grid of real values $0\le \nu \le 1$ and $\bm x,\bm y\in[0,1/2]^d$ in one, two and three dimensions.
We achieve machine precision with $E<7 \cdot 10^{-14}$ for $d\in\{1,2,3\}$ and anisotropy order $2n\in\{2,4,\ldots,12\}$.
For general values of $\nu$ and in particular if $Z_{\Lambda,\nu-2n}(\bm x,\bm y)$ is small relative to $Z_{\Lambda,\nu,2\bm\alpha}(\bm x,\bm y)$ cancellation errors occur on the left-hand side of the equation above which is why this benchmark features a restricted range of exponents $\nu$.
In what follows, we consider benchmarks for single anisotropic Epstein zeta functions.

\begin{table}[h]
\centering
\renewcommand*{\arraystretch}{1.2}
\begin{tabular}{c cccccc}
\toprule
$d$  & $n=1$ & $n=2$ & $n=3$ & $n=4$ & $n=5$ & $n=6$ \\
\midrule

$1$ & $6.21\cdot 10^{-16}$ & $1.32\cdot 10^{-15}$ & $3.08\cdot 10^{-15}$ & $3.36\cdot 10^{-15}$ & $3.37\cdot 10^{-15}$ & $7.02\cdot 10^{-15}$ \\

$2$ & $9.93\cdot 10^{-16}$ & $1.85\cdot 10^{-15}$ & $4.23\cdot 10^{-15}$ & $4.66\cdot 10^{-15}$ & $7.66\cdot 10^{-15}$ & $1.49\cdot 10^{-14}$ \\

$3$ & $9.24\cdot 10^{-16}$ & $3.01\cdot 10^{-15}$ & $7.39\cdot 10^{-15}$ & $2.02\cdot 10^{-14}$ & $2.56\cdot 10^{-14}$ & $6.27\cdot 10^{-14}$ \\

\bottomrule
\end{tabular}
 \caption{
Minimum of absolute and relative error of anisotropic Epstein zeta functions
$n!\sum_{|\bm\alpha|=n}Z_{\mathds Z^d,\nu,2\bm\alpha}(\bm x,\bm y)/\bm\alpha!$, which corresponds to the sum of set zeta derivatives from the poly-Laplacian
$
\Delta^n Z_{\mathds Z^d-\bm x,\nu}(\bm y)
$,
over a range of values $\bm y,\bm x$ and $\nu$ and dimensions $d=1,\ldots,3$ and orders $n=1,\ldots,6$ of the poly-Laplace operator.
For each fixed $d$, $n$, the maximum error for a grid of $\nu$, $\bm x$ and $\bm y$ values is shown,
where
$0\le\nu\le 1$, and
$\bm x,\bm y\in [0,1/2]^d$ 
in increments of $\Delta \nu=1/10$ and $\Delta x^{(j)}=\Delta y^{(j)}=1/5$.
We observe that the error has a slight increase in the dimension $d$ and the anisotropy order $n$ where, for all $n\in\{1,\ldots,6\}$ the maximal error is 
$E<8\cdot 10^{-15}$
for $d=1$,
$E<2\cdot 10^{-14}$ for $d=2$
and 
$E<7\cdot 10^{-14}$ for $d=3$.
}
\label{tab:poly_laplace}
\end{table}

For second-order anisotropy $\bm\alpha = 2 \bm e_j$, $1\le j\le d$ and all-equal vector arguments,
where we recall
$\bm e_j$ to be the $j$'th canonical unit vector in $d$-dimensions, the poly-Laplacian identity can be used to compute a single anisotropic Epstein zeta function for all-equal vector arguments as follows.
For $n=1$ and all-equal vectors $\bm x,\bm y$ with components $x,y\in\mathds R$ respectively
and $\nu\in\mathds C$ with $\nu\neq d+2$ if $\bm y\in\Lambda^*$
the equation above becomes
$$
Z_{\mathds Z^d,\nu, 2 \bm e_j}(x\bm 1,y \bm 1)
=
\frac{1}{d}
Z_{\mathds Z^d,\nu-2}(x \bm 1, y\bm 1)
$$
where $\bm 1=(1,\ldots,1)^T\in\mathds R^d$ is the all-one vector in $d$ dimensions, where we used that through the lattice sum representation one can immediately see that $Z_{\mathds Z^d,\nu, 2 \bm e_i}(x\bm 1,y \bm 1)=Z_{\mathds Z^d,\nu, 2 \bm e_j}(x\bm 1,y \bm 1)$ for all $1\le i,j\le d$.
In \Cref{tab:zetapartial2}
we benchmark the accuracy and stability of anisotropic Epstein zeta functions such as on the left-hand side in the equation above by comparing it to the reference values obtained by the Epstein zeta function on the right-hand side through EpsteinLib.
For each dimension $d\in\{1,\ldots,6\}$ the median $E^{\rm med}$ and maximum $E^{\rm max}$ error over a grid of values $-10\le \nu\le 10$ and $\bm x=x\bm 1$, $\bm y=y\bm 1$ with $0\le x,y\le 1/2$ in increments of
$\Delta \nu=1/10$ and $\Delta x=\Delta y= 1/100$
is shown.
Through dimensions one to six these grids add up to over three million values whose maximum error is bounded by $2\cdot 10^{-13}$.

\begin{table}[h]
\centering
\renewcommand*{\arraystretch}{1.2}
\resizebox{\textwidth}{!}{%
\begin{tabular}{c cccccc}
\toprule
 & $d=1$ & $d=2$ & $d=3$ & $d=4$ & $d=5$ & $d=6$ \\
\midrule

$E^{\rm med}$ & $4.19\cdot 10^{-17}$ & $5.13\cdot 10^{-17}$ & $9.06\cdot 10^{-18}$ & $8.33\cdot 10^{-17}$ & $1.39\cdot 10^{-17}$ & $5.24\cdot 10^{-17}$ \\

$E^{\rm max}$ & $4.55\cdot 10^{-15}$ & $3.83\cdot 10^{-15}$ & $1.78\cdot 10^{-13}$ & $1.53\cdot 10^{-13}$ & $1.40\cdot 10^{-14}$ & $1.53\cdot 10^{-14}$ \\

\bottomrule
\end{tabular}
}
 \caption{%
Error of the anisotropic Epstein zeta function
for the square lattice $\Lambda=\mathds Z^d$, for all-equal vector arguments $\bm x=x\bm 1$, $\bm y= y\bm 1$, and second-order anisotropy $\bm\alpha=2 \bm e_1$ in $d=1,\ldots,6$ dimensions 
where the median and the maximum of the error
over a grid of values $-10\le \nu \le 10$
and $0\le x,y\le 1/2$ in increments of
$\Delta\nu=1/10$ and $\Delta x=\Delta y=1/100$ is shown,
and where the singularities $\nu=d+2$, $\bm y=\bm 0$ are skipped.
The value of the anisotropic Epstein zeta function $Z_{\mathds Z^d,\nu,2 \bm e_1}(x\bm 1,y\bm 1)$
is compared with the reference value given in terms of the shifted Epstein zeta function as
$Z_{\mathds Z^d,\nu-2}(x\bm 1,y\bm 1)/d$.
The maximum errors are bounded by $E^{\rm max}<2 \cdot 10^{-13}$, and the median errors are below $E^{\rm med}<10^{-16}$.
}
\label{tab:zetapartial2}
\end{table}

We now present an identity for a special case of the eight-dimensional anisotropic Epstein zeta function for non-trivial arguments in terms of a Dirichlet series over the Ramanujan tau function $\tau(k)$, $k\in\mathds N_+$ defined by the relation
\cite[Eq. 27.14.18]{NIST:DLMF}
$$
\sum_{k=1}^{\infty}\tau(k)x^k
=
x\prod_{k=1}^{\infty}(1-x^k)^{24}
,\qquad x\in (-1,1)
\,,
$$
where a proof of the following identity is provided in \Cref{sec:appendix-num}.

\begin{lemma}
\label{ramanujan}
Denote by $L(\nu)$, $\nu\in\mathds C$ the Ramanujan tau $L$-function, that is defined by the meromorphic continuation of the Dirichlet series
$$
L(\nu)=
\sum_{k=1}^{\infty}\frac{\tau(k)}{k^{\nu}}
,\qquad \Re{\nu}>8
$$
to $\nu\in\mathds C$.
Then it holds that
$$
Z_{\mathds Z^8,\nu,\bm 1}(-\bm 1/2,\bm 1/2)
=
2^{-\nu/2}
L(\nu/2)\,.
$$
Here, $\tau$ is the Ramanujan tau function
and $\bm 1=(1,\ldots,1)^T\in\mathds R^8$ is the all-one vector.
\end{lemma}

In \Cref{fig:ramanujan} we compare the anisotropic Epstein zeta function for the special values from \Cref{ramanujan} over a range of values $-10\le \nu\le 10$
in increments of $\Delta\nu=1/20$, where the reference values are obtained through the Ramanujan tau $L$-function which is available in software packages such as Mathematica.
We observe machine precision with a maximum error of $E<3\cdot 10^{-14}$.

\begin{figure} 
    \centering
        \includegraphics[width=.45\linewidth]{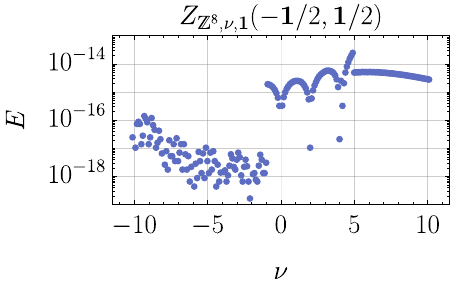}
    \caption{Error for a special case of the eight-dimensional anisotropic Epstein zeta function for the square lattice $\Lambda=\mathds Z^8$, $\bm x=-\bm 1/2$, $\bm y=\bm 1/2$ and $\bm\alpha=\bm 1$ as a function of $-10\le \nu\le 10$, where it can be exactly represented by the Ramanujan tau $L$-function.
    We achieve machine precision with a maximum error $E< 3\cdot 10^{-14}$ overall 
and an error $E<2\cdot 10^{-16}$ for $\nu<1$.
    }
    \label{fig:ramanujan}
\end{figure}

For large exponents, reference values for anisotropic Epstein zeta functions can be readily obtained through direct summation of the defining lattice sum, see \Cref{def-zeta-aniso}.
In \Cref{fig:sum_low}
we compute the error for two-dimensional anisotropic Epstein zeta function with second order anisotropy $|\bm\alpha|=2$ for large $\nu\ge 9$ 
so that the sum can be accurately computed by direct summation of squares
$\{-1000,\ldots,999,1000\}^2$.
The error is shown as a function of $9\le\nu\le 13$
in increments of $\Delta \nu=1/50$ for the square lattice $\Lambda=\mathds Z^2$ and the hexagonal lattice
$\Lambda_{\rm hex}$ with lattice matrix
$$
A_{\rm{hex}}=
\begin{bmatrix}
    1& 1/2 \\ 0 & \sqrt{3}/2
\end{bmatrix}
$$
for different values of $\bm x$.
Furthermore, in \Cref{fig:sum_high} we show the error of the two-dimensional anisotropic Epstein zeta function for high-order anisotropy $|\bm\alpha|=18$
as a function of $25\le\nu\le29$, where direct summation of $\{-1000,\ldots,999,1000\}^2$ squares offers accurate reference values.
The error is flat and stays below $E<2\cdot 10^{-15}$ for the low- and below $E<3\cdot 10^{-16}$ for the high-order anisotropy direct summation benchmark.

For generic arguments $\bm x,\bm y$,
no closed-form reference values exist, and the derivative Crandall representation in
\Cref{derivative-crandall-representation} requires a working precision that grows linearly
in $|\bm\alpha|$ due to catastrophic cancellation, which makes a dense scan over parameters costly.
Comparing the floating point algorithm against its own high-precision implementation, we observe a very mild loss of precision in the generic case with, typically, errors 
$E\lesssim 10^{-15}$ for $|\bm\alpha|\le 10$, 
$E\lesssim 10^{-14}$ for $|\bm\alpha|\le 20$ and
$E\lesssim 10^{-12}$ for $|\bm\alpha|\le 40$.

\begin{figure} 
    \centering
\includegraphics[width=1.\linewidth]{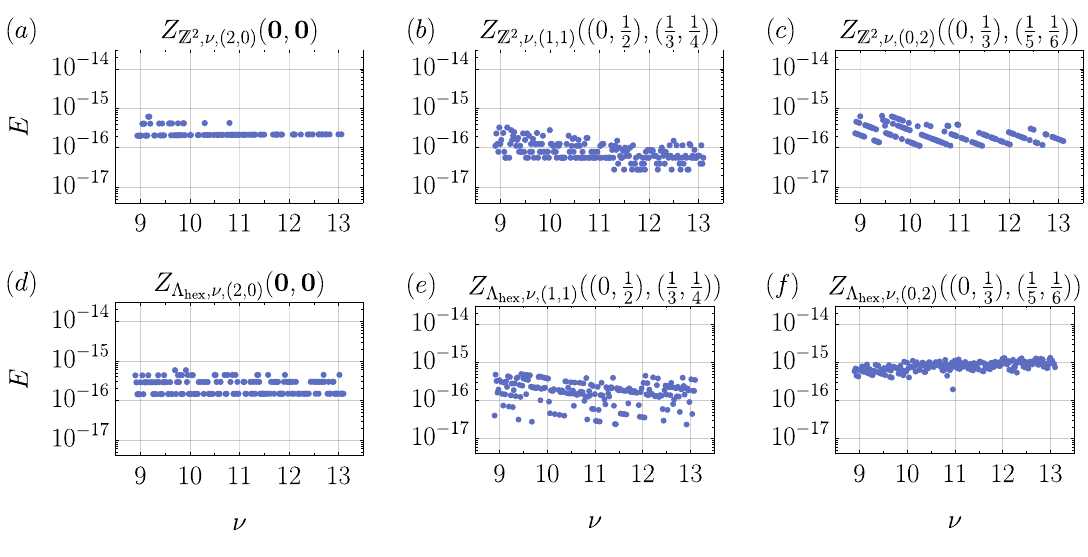}
    \caption{Minimum of absolute and relative error of $Z_{\Lambda,\nu,\bm\alpha}(\bm x,\bm y)$ for various values of $\bm x,\bm y$
    and second order derivatives $\bm\alpha$ for
    the square lattice
    $\Lambda=\mathds Z^2$ (a-c)
    and the hexagonal lattice $\Lambda=\Lambda_{\rm{hex}}$ (d-f)
    as a function of $9\le \nu\le 13$
    where the direct sum gives accurate reference values.
    The error is flat and bounded by $E<2\cdot 10^{-15}$.
    }
\label{fig:sum_low}
\end{figure}

\begin{figure} 
    \centering
\includegraphics[width=1.\linewidth]{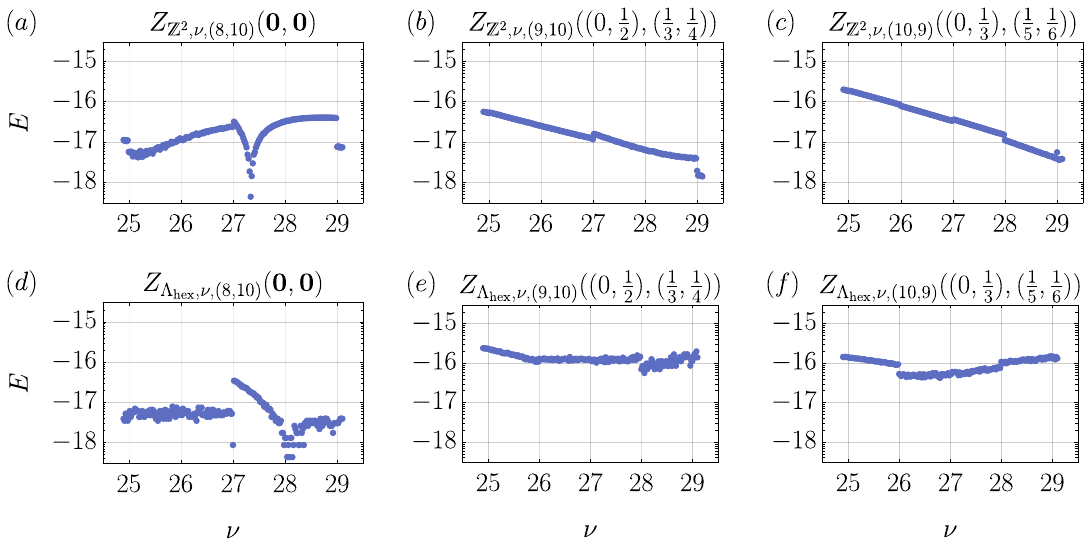}
    \caption{Minimum of absolute and relative error of $Z_{\Lambda,\nu,\bm\alpha}(\bm x,\bm y)$ for various values of $\bm x,\bm y$
    and high-order anisotropy $\bm\alpha$ for
    the square lattice
    $\Lambda=\mathds Z^2$ (a-c)
    and the hexagonal lattice $\Lambda=\Lambda_{\rm{hex}}$ (d-f)
    as a function of $25\le \nu\le 29$
    where the direct sum gives accurate reference values.
    The error is flat and bounded by $E<3\cdot 10^{-16}$.
    }
\label{fig:sum_high}
\end{figure}

\subsection{Stability for high-order anisotropy}
\label{sec:num-stab}

The preceding section showcases machine precision over a wide parameter range with a maximum error bounded by $E<2\cdot 10^{-13}$ for all benchmarks with independent reference values which indicates stability of the algorithm over the real parameters considered due to the small step sizes tested.
For low-order anisotropy, the stability of the algorithm corresponds to the stability of the Crandall formula for the Epstein zeta function, whose stability was thoroughly discussed in \cite[Ch. 4]{buchheitComputationPropertiesEpstein2026}.
We now focus on the stability in the order of the anisotropy $|\bm\alpha|$.

The main feature of our algorithm based on the harmonic-decomposition Crandall representation, \Cref{harmonic-decomposition-crandall-representation}, is its stability in the order of the anisotropy $|\bm\alpha|$.
While the derivative Crandall representation \Cref{derivative-crandall-representation} is exact and a useful analytical tool it is unsuitable for stable computations due to cancellation errors that get exponentially magnified in the order of the anisotropy in floating point arithmetic.
In \Cref{fig:fig_stab} we compare \Cref{alg} which is based on the harmonic-decomposition Crandall representation, \Cref{harmonic-decomposition-crandall-representation}, with an algorithm based on the derivative Crandall representation in \Cref{derivative-crandall-representation}.
The errors are shown
for the square lattice $\Lambda=\mathds Z^d$, for $\nu=5/2$
for various combinations of values $\bm x\in\{\bm 0,\bm e_1/3\}$ and $\bm y\in\{\bm 0,\bm e_1/4\}$
and for $\bm\alpha=n\bm e_1$ as a function of the anisotropy order $0\le n\le 60$.
The values obtained by the floating point arithmetic algorithm based on the upper Crandall derivative representation displays an exponential error scaling in the order of the anisotropy $n$.
In contrast the error of the anisotropic Epstein zeta function values obtained by \Cref{alg} in its EpsteinLib implementation remains flat over the entire parameter range, even for values corresponding to extremely high-order derivatives $n=60$.

\begin{figure} 
    \centering
\includegraphics[width=1.\linewidth]{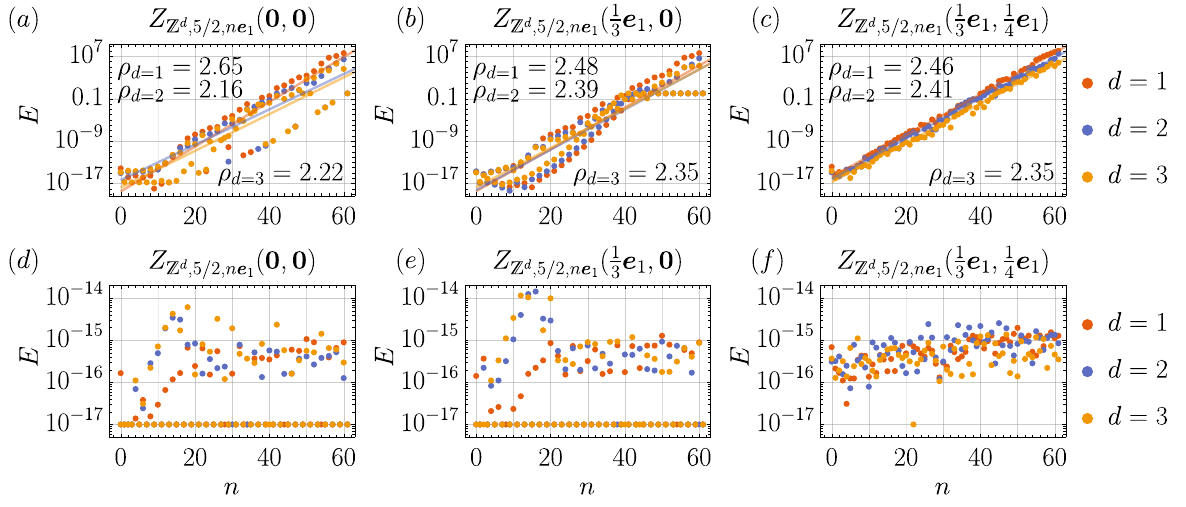}
    \caption{
Minimum of absolute and relative error of
    the anisotropic Epstein zeta function computed by 
    the derivative Crandall representation, \Cref{derivative-crandall-representation}, (a--c)
    in comparison with \Cref{alg}
    based on the harmonic-decomposition Crandall representation, \Cref{harmonic-decomposition-crandall-representation}, (d--f) for $\Lambda=\mathds Z^d$, $\bm\alpha=n\bm e_1$
    and various values of $\bm x\in \{\bm 0,\bm e_1/3\}$ and $\bm y\in \{\bm 0,\bm e_1/4\}$
    as a function of the order of the anisotropy $0\le n\le 60$
in $d\in\{1,2,3\}$ dimensions shown as red, blue, and green dots respectively.
    Here the reference values are obtained by a Mathematica prototype of the derivative Crandall representation, \Cref{derivative-crandall-representation}, with high-precision arithmetic.
For the algorithm based on the derivative Crandall representation (a--c)
we observe that the error scales exponentially in $n$ and the values of $\rho$ such that
$E\approx C\rho^n$ for some $C>0$ 
are shown in (b,c) with the curve displayed as semi-transparent lines in the color of the corresponding dimension.
Note that in (a) the error values for odd $n$, where the reference value vanishes due to the mirror symmetry in \Cref{syms-zero}, are lower than those of neighboring even $n$.
The error of \Cref{alg} (d--f) is flat with $E<2\cdot 10^{-14}$ for all $0\le n\le 60$.
In contrast the maximum error in (a--b) exceeds $E>4\cdot 10^5$ for all six anisotropic Epstein zeta functions considered
with a maximum error of $E>10^8$ in (c) for $d=1$.
}
\label{fig:fig_stab}
\end{figure}

We now benchmark the stability of anisotropic Epstein zeta functions at zero through extremely high-order Taylor series.
For vanishing wave vectors, the Taylor coefficients of the regularized anisotropic Epstein zeta function are given by the anisotropic Epstein zeta function as the following lemma shows, where a proof is provided in \Cref{sec:appendix-num}.

\begin{lemma}[Anisotropic Epstein zeta function for vanishing wave vectors]
\label{aniso-equals-reg-in-zero}
Let $\Lambda$
be a $d$-dimensional
lattice, let $\bm\alpha\in\mathds N^d$, let $\bm x\in\mathds R^d$ and let $\nu\in\mathds C$ such that $\nu\neq d+|\bm\alpha|$ if $\chi(\bm\alpha)=1$. 
Then
$$
Z^{\rm reg}_{\Lambda,\nu,\bm\alpha}(\bm x,\bm 0)
=
Z_{\Lambda,\nu,\bm\alpha}(\bm x,\bm 0)
\,.
$$
\end{lemma}

As an immediate consequence of \Cref{aniso-equals-reg-in-zero},
we can express coefficients of the Taylor series of the regularized Epstein zeta function around $\bm y=\bm 0$ in terms of the anisotropic Epstein zeta function,
where the choice of the convergence radius parameter is justified in \Cref{sec:appendix-num}.

\begin{corollary}
\label{taylor}
Let $\Lambda$ be a $d$-dimensional lattice, let $\bm x,\bm y\in\mathds R^d$ and let $\nu\in\mathds C$
such that $\nu\notin(d+2\mathds N)$.
Then
$$
Z^{\rm reg}_{\Lambda,\nu}(\bm x,\bm y)
=
\sum_{\bm\alpha \ge \bm 0} \frac{1}{\bm\alpha!} (-2 \pi i \bm 
y)^{\bm\alpha} Z_{\Lambda,\nu,\bm\alpha}(\bm x,\bm 0),
\qquad \sqrt{2}|\bm y|<\lambda_1^*\,
$$
where $\lambda_1^*=\min_{\bm z\in\Lambda^*\setminus\{\bm 0\}}|\bm z|$
is the minimal distance of the reciprocal lattice.
\end{corollary}

In \Cref{fig:error_taylor} we reconstruct the regularized Epstein zeta function in a ball around the origin of the wave vector $\bm y \in B_{1/2}(\bm 0)$ by its Taylor series in $\bm y=\bm 0$, where the coefficients are computed through anisotropic Epstein zeta functions, see \Cref{aniso-equals-reg-in-zero}.
The regularized anisotropic Epstein zeta function is analytic in the wave vector around this origin, see \Cref{hol-reg} and we numerically recover the expected exponential convergence of its Taylor series.

\begin{figure} 
    \centering
        \includegraphics[width=.6\linewidth]{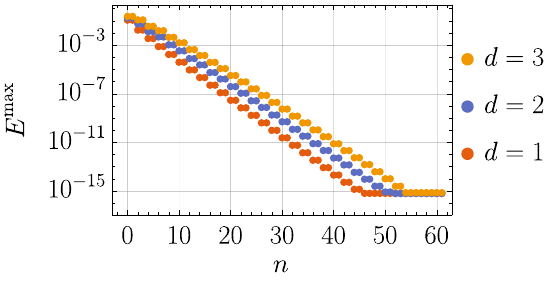}
    \caption{Error of the regularized Epstein zeta function 
    $Z^{\rm reg}_{\mathds Z^d,1/2}(\bm 0,\bm y)$
    reconstructed through
    its Taylor series in terms of the anisotropic Epstein zeta function
    $
\sum_{|\bm\alpha| \le n} \frac{1}{\bm\alpha!} (-2 \pi i \bm 
y)^{\bm\alpha} Z_{\mathds Z^d,1/2,\bm\alpha}(\bm 0,\bm 0)
    $
    in $d\in\{1,2,3\}$ dimensions
    , where the maximum error over
    a grid of values in a ball around the origin $\bm y\in [-1/2,1/2]^d$, $|\bm y| \le 1/2$ in increments of $\Delta y^{(j)}=1/20$
    is shown as a function of the order of the Taylor series $0\le n\le 60$.
    Note that error of odd orders $n$ is the same as the error at even orders $n-1$ since  all odd coefficients vanish by the mirror symmetry, see \Cref{syms-zero}.
    We observe exponential convergence of the reconstruction of the regularized Epstein zeta function through the anisotropic Epstein zeta function in the order of the Taylor series with $E^{\rm max}<10^{-15}$ for $n\ge 54$ in the dimensions considered.
    }
    \label{fig:error_taylor}
\end{figure}

\subsection{Linear runtime for low-order anisotropy}
\label{sec:num-time}

In the most relevant cases of low- to mid-order anisotropy we observe near linear runtime of our algorithm for anisotropic Epstein zeta functions in the order of the anisotropy $|\bm\alpha|$.
In \Cref{fig:time}
we show the evaluation times of the anisotropic Epstein zeta function in $d\in\{1,\ldots,6\}$ dimensions up to tenth anisotropy order $0\le|\bm\alpha|\le 10$ (a) and the relative evaluation times, where the runtime value is divided by the runtime value at $\bm\alpha=\bm 0$, in $d\in\{1,2,3\}$ up to twentieth order anisotropy $0\le|\bm\alpha|\le 20$ (b).
We observe a moderate increase in runtime in the order of the anisotropy, in particular in the most relevant two- and three-dimensional case an increase of the anisotropy order from zero to ten amounts to linear increase in runtime of one order of magnitude.
The increase in evaluation time of around ten derivatives
is similar to the increase of evaluation time of increasing the dimension.
Note that \Cref{tab:harmonic_high_order} features the evaluation times of the harmonic polynomials, where the precomputation of the coefficient $c_{k,\bm\alpha,\bm\gamma}$, shown as $t_{\rm pre}$ in the first column,
dominates the runtime.
The subsequent evaluation of the harmonic polynomials $h_{k,\bm\alpha}(\bm y)$ for different values of $k=0,1,\ldots,\lfloor|\bm\alpha|/2\rfloor$ and $\bm y\in \mathds R^d$ has fast maximum evaluation times $t_{\rm eval}^{\rm max}$, as
shown in the second column.
The values were obtained on an Intel Core i7-1260P (12th Gen) 16-core processor with 32 GB of RAM.

\begin{figure} 
    \centering
\includegraphics[width=1.\linewidth]{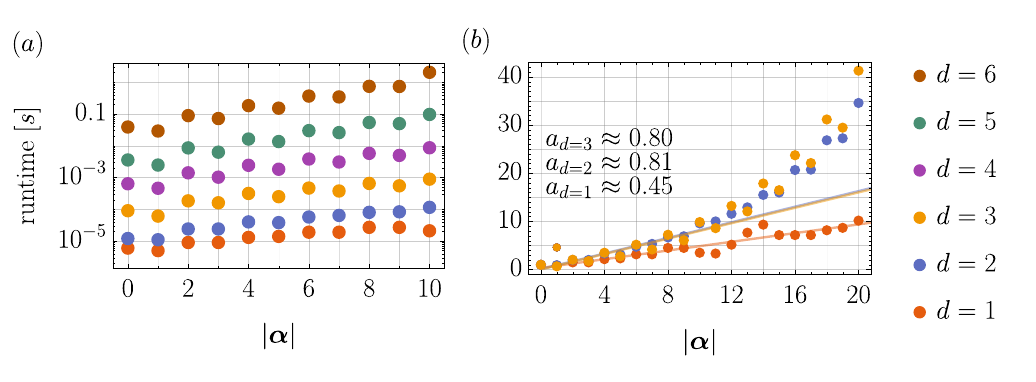}
    \caption{Evaluation times of the anisotropic Epstein zeta function $Z_{\mathds Z^d,1/2,\bm\alpha}(\bm 0,y \,\bm e_1)$ for $\bm\alpha=n \bm e_1$, $n\in\mathds N$
    for $0\le n\le 10$ 
    in $d\in \{1,\ldots,6\}$ dimensions (a).
    The values shown are the maximum evaluation times obtained for $-1/2\le y\le 1/2$ in steps of $\Delta y=1/10$ where for each $y$ the median evaluation time over $10$ runs was chosen.
    For $|\bm\alpha|\le 10$ we observe evaluation times of less than $30$ microseconds in $d=1$, less than $0.2$ milliseconds in $d=2$, less than $0.9$ milliseconds in $d=3$ and less than $2.1$ seconds in $d=6$.
    For odd-order anisotropy $|\bm\alpha|$, the symmetries reduce the computational work so that the runtime is sometimes lower than for $|\bm\alpha|-1$.
    In (b) the evaluation times in $d\in\{1,2,3\}$
    for $0\le n\le 20$ are shown,
    where each time is normalized by dividing by the respective evaluation time at $n=0$.
    In (b) we observe a relative increase of the runtime for low-dimensions and low- to mid-order derivatives of one order of magnitude in $0\le n\le 20$ for $d=1$ and $0\le n\le 10$ in $d\in\{2,3\}$.
There the semi-transparent lines $a\,n+C$, $a,C>0$ and their slope $a_d$ obtained by a linear fit on the values $0\le n\le 20$ for $d=1$ and $0\le n\le 10$ for $d\in\{2,3\}$ are shown as well.
    }
\label{fig:time}
\end{figure}

\subsection{Runtime and stability of harmonic polynomial evaluation}
\label{sec:num-harmonic}

We now showcase numerical stability and fast evaluation times of 
our algorithm for the
computation of
the harmonic polynomials $h_{k,\bm\alpha}(\bm y)$,
$\bm y\in\mathds R^d$, $\bm\alpha\in\mathds N^d$ and $0\le k\le \lfloor|\bm\alpha|/2\rfloor$,
based on \Cref{harmonic-comp}, where the inner dyadic coefficients $c_{k,\bm\alpha,\bm\gamma}$ are computed exactly.
We denote by $h^{\rm num}_{k,\bm\alpha}(\bm y)$ the values obtained by the implementation of \Cref{harmonic-comp} in EpsteinLib and compare with exact values $h^{\rm exact}_{k,\bm\alpha}(\bm y)$ obtained by an exact addition of the dyadic summands for component-wise dyadic $\bm y$.
Denote by
$h_{k,\bm\alpha}^{\rm abs}(\bm y)$
the sum of the exact absolute values of the representation in \Cref{harmonic-comp}, that is
$$
h_{k,\bm\alpha}^{\rm abs}(\bm y)
=
\sum_{\substack{
2\bm\gamma\ge\bm\alpha  
\\ 
|\bm\gamma|=|\bm\alpha|-k
}}
|c^{\rm full}_{k,\bm\alpha,\bm\gamma}
\,
\bm y^{2\bm\gamma-\bm\alpha}
|
\,,
$$
with $c^{\rm full}_{k,\bm\alpha,\bm\gamma} 
=
\omega_{k,|\bm\alpha|}c_{k,\bm\alpha,\bm\gamma}$
and $\omega_{k,|\bm\alpha|}$ from \Cref{harmonic-comp}.
Then,
for
 $h^{\rm exact}_{k,\bm\alpha}(\bm y) \neq 0$,
we denote by
$$
\mathcal{E}_{\rm bwd}
=
\frac{
|h^{\rm num}_{k,\bm\alpha}(\bm y)-h^{\rm exact}_{k,\bm\alpha}(\bm y)|}{h^{\rm abs}_{k,\bm\alpha}(\bm y)}
$$
the relative backward error
of the algorithm based on the representation in \Cref{harmonic-comp} with the associated condition number
$$
\kappa_{h}= 
\frac{h^{\rm abs}_{k,\bm\alpha}(\bm y)}{|h^{\rm exact}_{k,\bm\alpha}(\bm y)|}\,,
$$
see \cite[Ch. 1.5]{highamAccuracyStabilityNumerical2002},
where $\kappa_{h}$ is well defined since $h^{\rm exact}_{k,\bm\alpha}(\bm y) \neq 0$ implies $h^{\rm abs}_{k,\bm\alpha}(\bm y) \neq 0$.
Note that $\mathcal{E}_{\rm bwd}$
 corresponds to relatively perturbing the
coefficients $c^{\rm full}_{k,\bm\alpha,\bm\gamma}$
to $c^{\rm full}_{k,\bm\alpha,\bm\gamma}(1+\delta_{\bm\gamma})$
by at most $\delta_{\bm\gamma}\in\mathds R$,
$|\delta_{\bm\gamma}|\le\mathcal{E}_{\rm bwd}$.
Then, the relative error of the algorithm
is approximately bounded by the product of the condition number with the relative backward error
$$
E_{\rm rel}\lesssim
\kappa_{h} \cdot \mathcal{E}_{\rm bwd}
\,,
$$
where we can interpret $\mathcal{E}_{\rm bwd}$ as the error due to the algorithm 
and $\kappa_{h}$, which can be large due to cancellation of the summands of $h_{k,\bm\alpha}(\bm y)$, is a property of the problem.
We now show through a wide range of test cases that $\mathcal{E}_{\rm bwd}$ is small irrespective of $\kappa_{h}$ and that $\kappa_{h}$ is small in the median. 

In \Cref{tab:harmonic_high_order}
we compute the runtime, the condition numbers and the relative backward error over 
our implementation of the
harmonic polynomials $h_{k,\bm\alpha}(\bm y)$
based on \Cref{harmonic-comp} 
for
high-order $|\bm\alpha|=24$ 
with all-equal $\bm\alpha=(n,\ldots,n)^T$, $d\,n=24$,
over a grid of values $k\in\{0,1,\ldots,12\}$ and $\bm y\in[-1,1]^d$ in 
$d\in\{1,\ldots,4\}$ dimensions.
The increments $\Delta y^{(j)}$ are chosen
to ensure a similar total number of points per dimension
and as dyadic rationals so that the error is due to the algorithm and not rounding of the input values.
In all test cases we observe that the relative backward error is flat and strictly less than two times the machine epsilon
$
\mathcal{E}_{\rm bwd}^{\rm max}<2\varepsilon_{\rm M}
$
with 
$\varepsilon_{\rm M}\approx 2.2\cdot 10^{-16}$,
irrespective of the dimension or the maximum condition number  with values $\kappa_{h}^{\rm max}\ge 2 \cdot 10^{5}$ in $d\in\{2,3,4\}$ dimensions.
This strongly suggests that the algorithm based on the representation in \Cref{harmonic-comp} gives the exact solution to a problem with coefficients $c^{\rm full}_{k,\bm\alpha,\bm\gamma}$ perturbed by no more than an unavoidable rounding error.
As a consequence, we have that, in the median, the relative error is approximately bounded by
$$
E_{\rm rel}\lesssim
\kappa_{h}^{\rm med} \mathcal{E}_{\rm bwd}^{\rm max}< 10^{-14}
\,.
$$
For the runtime we observe that the one-time pre-computation $t_{\rm pre}$ of the coefficients $c_{k,\bm\alpha,\bm\gamma}$ for fixed $\bm\alpha$ dominates the computation time
with an increase of roughly three orders of magnitude across dimensions $d=1$ to $d=4$ with $t_{\rm pre}< 6\cdot 10^{-2}$ in $d=4$.
Given these pre-computed coefficients we observe very fast maximum evaluation times with an increase of less than three orders of magnitude across $d=1$ up to 
$t^{\rm max}_{\rm eval}<9\cdot 10^{-6}$ in $d=4$.

In \Cref{tab:harmonic_stab}
we compute the runtime, the condition numbers and the relative backward error for low- to high-order two-dimensional harmonic polynomials $h_{k,\bm\alpha}(\bm y)$,
with fixed anisotropy ratio and alternatingly even and odd order $\bm\alpha=(n,2n)$, $n\in\{0,1,\ldots,10\}$.
We consider a grid of values $k\in\{0,1,\ldots,\lfloor|\bm\alpha|/2\rfloor\}$
and
$\bm y\in [-1,1]^2$ in increments of $\Delta y^{(j)}=1/64$.
The backward error is bounded by less than two times the machine epsilon 
$\mathcal{E}_{\rm bwd}^{\rm max}<2\varepsilon_{\rm M}$
irrespective of the order of the anisotropy $0\le|\bm\alpha|\le 30$ and the maximum condition number with
$\kappa_{h}^{\rm max}> 10^5$ for $|\bm\alpha|\ge 20$, so that the numerical data suggest backward stability in the order of the anisotropy $|\bm\alpha|$.
In particular,
the median of the relative error is approximately bounded by
$$
E_{\rm rel}\lesssim
\kappa_{h}^{\rm med} \mathcal{E}_{\rm bwd}^{\rm max}< 2\cdot 10^{-14}
\,.
$$
For the runtime we observe an increase in runtime for the pre-computation of the coefficients $c_{k,\bm\alpha,\bm\gamma}$
of four orders of magnitude across anisotropy order $0\le|\bm\alpha|\le 30$
up to
$t_{\rm pre}\le 2\cdot 10^{-3}$
at $\bm\alpha=(10,20)$.
Given these pre-computed coefficients, we observe very fast maximum evaluation times with an increase of only a single order of magnitude across $0\le|\bm\alpha|\le 30$
with $t_{\rm eval}^{\rm max}< 2\cdot 10^{-7}$ throughout.
The values were obtained on an Intel Core i7-1260P (12th Gen) 16-core processor with 32 GB of RAM.
The numerical experiments gathered here strongly suggest backward stability with, given the precomputed coefficient, very fast evaluation times of our algorithm based on \Cref{harmonic-comp} for high-order anisotropy.

\begin{table}[h]
\centering
\renewcommand*{\arraystretch}{1.2}
\begin{tabular}{c|cc|cc|c}
\toprule
 $d$
 & $t_{\rm pre}$ [s] 
 & $t_{\rm eval}^{\rm max}$ [s] 
 & $\kappa_{h}^{\rm med}$
 & $\kappa_{h}^{\rm max}$ 
 & $\mathcal{E}_{\rm bwd}^{\rm max}$ 
\\
\midrule

$1$ & $4.8\cdot 10^{-5}$ & $2.2\cdot 10^{-8}$ & $1.0$ & $1.0$ & $1.11\cdot 10^{-16}$ \\

$2$ & $8.3\cdot 10^{-4}$ & $1.4\cdot 10^{-7}$ & $12.3$ & $2.0\cdot 10^{5}$ & $2.29\cdot 10^{-16}$ \\

$3$ & $7.8\cdot 10^{-3}$ & $9.8\cdot 10^{-7}$ & $28.2$ & $2.6\cdot 10^{6}$ & $3.40\cdot 10^{-16}$ \\

$4$ & $5.1\cdot 10^{-2}$ & $8.1\cdot 10^{-6}$ & $28.4$ & $5.5\cdot 10^{5}$ & $2.11\cdot 10^{-16}$ \\

\bottomrule
\end{tabular}
 \caption{%
Evaluation times, condition numbers and backward errors 
of the
high-performance implementation of the harmonic polynomials $h_{k,\bm\alpha}(\bm y)$
of degree $0,2,\ldots ,24$ in EpsteinLib in comparison with an exact implementation
over a grid of values $k$ and $\bm y$ for fixed $\bm\alpha$ in dimensions $d=1,\ldots,4$.
Here we fix $\bm\alpha=(24)$ in 1D,
 $\bm\alpha=(12,12)$ in 2D,
 $\bm\alpha=(8,8,8)$ in 3D
 and $\bm\alpha=(6,6,6,6)$ in 4D
 so that $|\bm\alpha|=24$ in each dimension and compute the evaluation times, condition numbers and backward error over
 a grid of values $k=0,1\ldots,12$ and $\bm y\in [-1,1]^d$
  in increments of 
  $\Delta y^{(j)}=1/512$ in 1D,
 $\Delta y^{(j)}=1/32$ in 2D,
 $\Delta y^{(j)}=1/8$ in 3D
 and 
 $\Delta y^{(j)}=1/4$ in 4D.
For each fixed $\bm\alpha$, 
all coefficients $c_{k,\bm\alpha,\bm\gamma}$, $k=0,\ldots,12$, $|\bm\gamma|=|\bm\alpha|-k$ are precomputed
and reused for all $k$ and $\bm y$ with a runtime of $t_{\rm pre}$ shown in the second column,
and the maximum evaluation times of the harmonic polynomials using these precomputed coefficients are shown as $t_{\rm eval}^{\rm max}$
in the third column.
Note that throughout each calculation is repeated an appropriate number of times and the median evaluation time of these repeated calculations is chosen before the overall time is reported, where, in the case of $t_{\rm eval}$, the values are restricted to a subset of the grid of $\bm y$, since the times depend only on $\bm\alpha$ and $k$, with the longest evaluation times obtained for $k=0$.
We observe an increase in runtime from 
$t_{\rm pre}<5\cdot 10^{-5}$ and $t_{\rm eval}^{\rm max}<3\cdot 10^{-8}$ at $d=1$ 
up to
$t_{\rm pre}<6\cdot 10^{-2}$ and $t_{\rm eval}^{\rm max}<9\cdot 10^{-6}$ at $d=4$.
The median and maximum condition numbers
$\kappa_{h}^{\rm med}$
and
$\kappa_{h}^{\rm max}$
and
the maximum backward error $\mathcal{E}_{\rm bwd}^{\rm max}$
are shown in the last three columns.
Here, all values with $h^{\rm num}_{k,\bm\alpha}(\bm y)=h^{\rm exact}_{k,\bm\alpha}(\bm y)=0$ are skipped as the algorithm is exact there even though the associated condition number is infinite
and the backward error stays well defined because 
$h^{\rm exact}_{k,\bm\alpha}(\bm y)=0$
implies $h^{\rm num}_{k,\bm\alpha}(\bm y)=0$
in all test cases.
We observe median condition numbers $\kappa_{h}^{\rm med}< 29$ 
and a relative backward error
bounded by $\mathcal{E}_{\rm bwd}^{\rm max}\le 3.4\cdot 10^{-16}$ throughout 
and maximum condition numbers $\kappa_{h}^{\rm max}\ge 2\cdot 10^5$ for $d\ge 2$.
}
\label{tab:harmonic_high_order}
\end{table}

\begin{table}[h]
\centering
\renewcommand*{\arraystretch}{1.2}
\begin{tabular}{c|cc|cc|c}
\toprule
 $\bm\alpha$
 & $t_{\rm pre}$ [s] 
 & $t_{\rm eval}^{\rm max}$ [s] 
 & $\kappa_{h}^{\rm med}$
 & $\kappa_{h}^{\rm max}$ 
 & $\mathcal{E}_{\rm bwd}^{\rm max}$ 
\\
\midrule

$(0,0)$ & $2.2\cdot 10^{-7}$ & $1.4\cdot 10^{-8}$ & $1.0$ & $1.0$ & $0$ \\

$(1,2)$ & $1.1\cdot 10^{-6}$ & $2.2\cdot 10^{-8}$ & $1.0$ & $1.4\cdot 10^{3}$ & $0$ \\

$(2,4)$ & $1.0\cdot 10^{-5}$ & $3.8\cdot 10^{-8}$ & $2.0$ & $3.4\cdot 10^{3}$ & $0$ \\

$(3,6)$ & $2.1\cdot 10^{-5}$ & $4.3\cdot 10^{-8}$ & $2.0$ & $1.1\cdot 10^{4}$ & $2.87\cdot 10^{-16}$ \\

$(4,8)$ & $6.8\cdot 10^{-5}$ & $6.5\cdot 10^{-8}$ & $4.0$ & $3.3\cdot 10^{4}$ & $2.43\cdot 10^{-16}$ \\

$(5,10)$ & $1.1\cdot 10^{-4}$ & $8.9\cdot 10^{-8}$ & $5.6$ & $2.9\cdot 10^{4}$ & $2.66\cdot 10^{-16}$ \\

$(6,12)$ & $2.8\cdot 10^{-4}$ & $1.1\cdot 10^{-7}$ & $7.1$ & $1.8\cdot 10^{5}$ & $3.08\cdot 10^{-16}$ \\

$(7,14)$ & $4.1\cdot 10^{-4}$ & $1.2\cdot 10^{-7}$ & $11.9$ & $1.2\cdot 10^{6}$ & $3.06\cdot 10^{-16}$ \\

$(8,16)$ & $8.3\cdot 10^{-4}$ & $1.5\cdot 10^{-7}$ & $14.1$ & $6.6\cdot 10^{5}$ & $2.65\cdot 10^{-16}$ \\

$(9,18)$ & $1.1\cdot 10^{-3}$ & $1.6\cdot 10^{-7}$ & $23.6$ & $9.1\cdot 10^{7}$ & $3.24\cdot 10^{-16}$ \\

$(10,20)$ & $2.0\cdot 10^{-3}$ & $1.8\cdot 10^{-7}$ & $29.0$ & $1.3\cdot 10^{7}$ & $4.15\cdot 10^{-16}$ \\

\bottomrule
\end{tabular}
 \caption{%
Evaluation times, condition numbers and backward errors 
of the
2D
high-performance implementation of the harmonic polynomials $h_{k,\bm\alpha}(\bm y)$ based on \Cref{harmonic-comp},
where for each $\bm\alpha=(n,2n)$ for $0\le n\le 10$
the maximum error over a grid of values
$0\le k \le \lfloor|\bm\alpha|/2\rfloor$
and 
$y\in [-1,1]^2$
in increments of $\Delta y^{(j)}=1/64$
is taken.
The runtime for the pre-computation of the coefficients $c_{k,\bm\alpha,\bm\gamma}$ for fixed $\alpha=(n,2n)$ is shown as 
$t_{\rm pre}$ in seconds in the second column
and the maximum evaluation times for specific harmonic polynomials $h_{k,\bm\alpha}(\bm y)$ are shown in the third column.
Note that throughout each calculation is repeated an appropriate number of times and the median evaluation time of these repeated calculations is chosen before the overall time is reported, where, in the case of $t_{\rm eval}$, the values are restricted to a subset of the grid of $\bm y$, since the times depend only on $\bm\alpha$ and $k$, with the longest evaluation times obtained for $k=0$.
We observe an increase in runtime from 
$t_{\rm pre}<3\cdot 10^{-7}$ and $t_{\rm eval}^{\rm max}<2\cdot 10^{-8}$ at $\bm\alpha=(0,0)$ 
up to
$t_{\rm pre}\le 2\cdot 10^{-3}$ and $t_{\rm eval}^{\rm max}<2\cdot 10^{-7}$ at $\bm\alpha=(10,20)$.
The median and maximum condition numbers
$\kappa_{h}^{\rm med}$
and
$\kappa_{h}^{\rm max}$
and 
the maximum backward error $\mathcal{E}_{\rm bwd}^{\rm max}$
are shown in the last three columns.
Here, all values with $h^{\rm num}_{k,\bm\alpha}(\bm y)=h^{\rm exact}_{k,\bm\alpha}(\bm y)=0$ are skipped as the algorithm is exact there even though the associated condition number is infinite
and the backward error stays well defined because 
$h^{\rm exact}_{k,\bm\alpha}(\bm y)=0$
implies $h^{\rm num}_{k,\bm\alpha}(\bm y)=0$
in all test cases.
The backward error is bounded by 
$\mathcal{E}_{\rm bwd}^{\rm max}<4.2\cdot 10^{-16}$
throughout.
 }
\label{tab:harmonic_stab}
\end{table}

\subsection{Singular Euler-Maclaurin expansion}
\label{sec:num-sem}

\begin{figure} 
    \centering
\includegraphics[width=1.\linewidth]{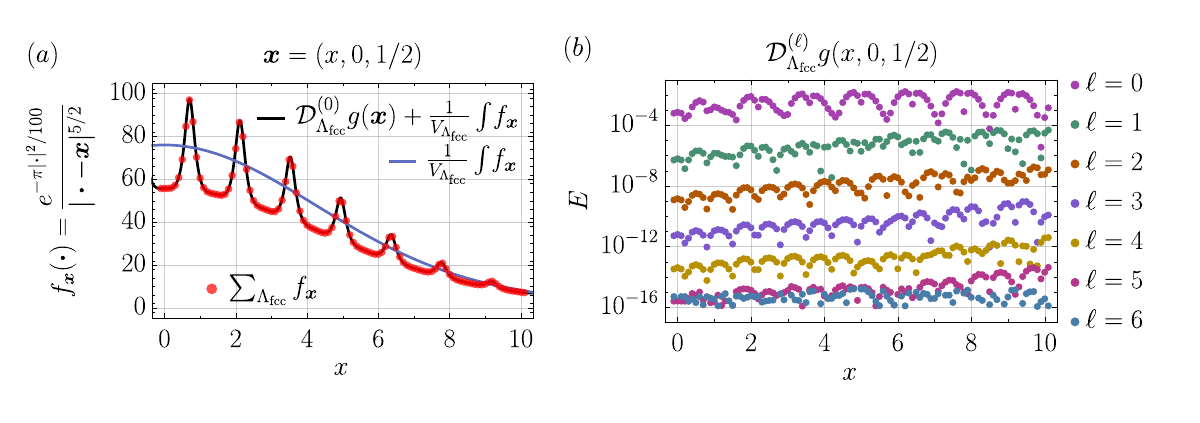}
    \caption{%
    Singular Euler-Maclaurin expansion (SEM)
 $f_{\bm x}(\bm \cdot)=s_{\nu}(\bm \cdot-\bm x)g(\bm \cdot)$ for the Gaussian $g(\bm \cdot)=e^{-\pi |\bm \cdot|^2/\sigma^2}$
with $\nu=5/2$, $\sigma=10$ and
$\bm x=(x,0,1/2)$,
as a function of $x$ for $0\le x\le 10$.
In (a),
the integral
$
\textstyle \frac 1{V_{\Lambda_{\rm fcc}}}\int_{\mathds R^3}f_{(x,0,1/2)}(\bm y)\d \bm y
$
is shown as a smooth blue line,
and the sum
$\sum_{\bm z\in\Lambda_{\rm fcc}}f_{(x,0,1/2)}(\bm z)$
whose values are obtained by summing over the truncated range $\bm z\in A_{\rm fcc}\{-50,-49,\ldots,50\}^3$,
is shown as red dots
in increments of $\Delta x=1/10$.
The zero SEM operator correction term 
$\mathcal{D}^{(0)}_{\Lambda_{\rm fcc}}g(x,0,1/2)$
added to the integral
is shown as a black line which intersects the sum at the values of $x$ where the sum is shown.
In (b)
the error of the SEM is shown for the orders $\ell\in\{0,1,\ldots,6\}$, 
where the difference of the sum with the integral is taken as the reference.
We observe a maximum error $E<2\cdot 10^{-2}$ for $\ell=0$ and $E< 2\cdot 10^{-15}$
for $\ell=6$.
}
\label{fig:sem}
\end{figure}

We showcase our algorithm for the computation of the anisotropic Epstein zeta function which appears in the coefficients of the
singular Euler-Maclaurin (SEM) expansion,
one of the authors' recent generalization of the classical Euler-Maclaurin expansion to singular kernels $s_\nu(\bm \cdot)=\vert\bm\cdot\vert^{-\nu}$
\cite{buchheit2022efficient,buchheit2022singular}
with applications to topological superconductors \cite{buchheit2023exact}.
For a $d$-dimensional lattice $\Lambda$,
$\bm x\in\mathds R^d$, $\nu\in \mathds C$,
and a smooth function $g\colon \mathds R^d\to\mathds C$ consider
$f_{\bm x}(\bm \cdot)=s_\nu(\bm \cdot-\bm x)g(\bm \cdot)$,
then, following
\cite[Ch. 5]{buchheit2022singular},
we have that the difference of the sums and the integral
$$
\sums_{\bm z\in\Lambda}f_{\bm x}(\bm z)
-\frac{1}{V_{\Lambda}}\int_{\mathds R^d}f_{\bm x}(\bm y)\d \bm y
=
\mathcal{D}^{(\ell)}_{\Lambda}g(\bm x)
+\mathcal{R}^{(\ell)}_{\Lambda}g(\bm x)
$$
is given in terms of the SEM operator
$$
\mathcal D^{(\ell)}_{\Lambda}g(\bm x)=\sum_{
\substack{
\bm\alpha \in \mathds N^d
\\
|\bm\alpha|\le 2\ell
}
}
\frac{1}{\bm\alpha!}
Z^{\rm reg}_{\Lambda,\nu,\bm\alpha}(\bm  x,\bm 0)
    g^{(\bm\alpha)}(\bm x)
$$
with an error $\mathcal{R}^{(\ell)}_{\Lambda}g$
that scales as $\nabla^{2\ell+1} g$.

\begin{lemma}[Singular Gaussian integral in three dimensions]
\label{gauss-int-3d}
Let $\bm x\in\mathds R^3$, let $\nu\in\mathds C\setminus(3+2\mathds N)$
and $f_{\bm x}(\bm \cdot)=s_\nu(\bm \cdot-\bm x)g(\bm \cdot)$ with
$g=g_{\sigma}= e^{-\pi (\bm\cdot)^2/\sigma^2}$, $\sigma>0$.
Then
$$
\int_{\mathds R^3} f_{\bm x}(\bm y)\, \d \bm y
=
2\;\!\pi^{(\nu-1)/2}\sigma^{3-\nu}
\Gamma\Big(\frac{3-\nu}{2}\Big)
\,_1\!\!\:F_1\Big(\frac \nu2;\frac 32;-\pi \frac{\bm x^2}{\sigma^2}\Big)
,\qquad 
\Re\nu < 3
$$
where the right hand side is the holomorphic continuation of the integral to $\nu\in\mathds C\setminus(3+2\mathds N)$ and
where $_1\!\!\:F_1(\cdot,\cdot,\cdot)$ is the confluent hypergeometric function of the first kind (also known as Kummer's function
\cite[Sec. 6.1, Eq. (1)]{Batemann1953a})
$$
_1\!\!\:F_1(a;b;z)
=
\sum_{n=0}^{\infty}
\frac{(a)_{\overline{n}}}{(b)_{\overline{n}}}\frac{z^n}{n!}
\,,
\qquad 
a,b,z\in\mathds C,\ b\notin-\mathds N
\,,
$$
where $(q)_{\overline{n}}=q(q+1)\ldots (q+n-1)$ is the rising Pochhammer symbol.
\end{lemma}

Under the conditions of \Cref{gauss-int-3d}, choosing $g=g_{\sigma}$ with $\sigma$ large enough such that $\nabla^{2\ell+1}g$ is negligible for small and moderately large $\ell\in\mathds N$ we have by the SEM that
$$
\sums_{\bm z\in\Lambda}f_{\bm x}(\bm z)
\approx
\frac{1}{V_{\Lambda}}\int_{\mathds R^d}f_{\bm x}(\bm y)\d \bm y
+
\mathcal{D}^{(\ell)}_{\Lambda}g(\bm x)\,.
$$
By the correspondence of the anisotropic Epstein zeta function with its regularized version for vanishing wave vectors, see \Cref{aniso-equals-reg-in-zero},
in the equation for the SEM operator we can, for $\nu\notin d+2\mathds N$, replace
the regularized anisotropic Epstein zeta function with the non-regularized variant
so that by the Faà di Bruno formula \Cref{derivative-fy2}
we have the explicit representation
$$
\mathcal D^{(\ell)}_{\Lambda}g(\bm x)=\sum_{
\substack{
\bm\alpha \in \mathds N^d
\\
|\bm\alpha|\le 2\ell
}
}
\frac{1}{\bm\alpha!}
Z_{\Lambda,\nu,\bm\alpha}(\bm  x,\bm 0)
\sum_{2\bm\beta\le\bm\alpha}
\frac{p_{\bm\alpha,\bm\beta}(\bm x)}{(\sigma^2)^{|\bm\alpha|-|\bm\beta|}}g(\bm x)
$$
with the polynomial $p_{\bm\alpha,\bm\beta}(\bm \cdot)$ as in \Cref{derivative-crandall-representation}.
We now consider the face-centered cubic lattice with unit nearest neighbor distance $\Lambda_{\rm fcc}=A_{\rm fcc}\mathds Z^3$ with lattice matrix
$$
A_{\rm fcc}=
\frac{1}{\sqrt{2}}
\begin{bmatrix}
1 & 1 & 0 \\ 
1 & 0 & 1 \\
0 & 1 & 1
\end{bmatrix}
\,.
$$
In \Cref{fig:sem}
we showcase the three-dimensional SEM for the Gaussian
$f_{\bm x}=s_{\nu}(\bm \cdot-\bm x)g(\bm\cdot)$
with $g=g_{\sigma}$ as in \Cref{gauss-int-3d} and $\sigma=10$
so that the error of the SEM is small and so that
summation over the truncated range
$\bm z\in A_{\rm fcc}\{-50,-49,\ldots,50\}^3$
gives accurate reference values for the infinite sum
$\sum_{\bm z\in\Lambda_{\rm fcc}}f_{\bm x}(\bm z)$.
In (a) we observe that the values of the sum are sensitive to the distance of $\bm x$ to the lattice $\Lambda_{\rm fcc}$
with big differences to the value of the integral, especially for small $x$.
In contrast, the integral plus the zero order SEM correction term $\mathcal D^{(0)}_{\Lambda_{\rm fcc}}g$ is visually indistinguishable from the sum and accurately captures the structure of the lattice.
In (b) we observe an error of less than $2\%$ for the zero order SEM operator $\mathcal D^{(\ell)}g$, $\ell=0$
and machine precision $E<2\cdot 10^{-15}$ for $\ell=6$.

\section{Outlook and Conclusions}
\label{sec:outlook}

We have defined, analyzed and derived stably computable representations for anisotropic Epstein zeta functions, a general class of lattice sums over anisotropic interaction kernels central to the numerical mathematics and physics of long-range interactions.
Together with our account of its properties 
and the implementation in EpsteinLib \cite{epsteinlib}, this extends methods based on Epstein zeta functions beyond isotropic interaction kernels and makes the anisotropic Epstein zeta function available as a new standard special function.

Based on the work of Crandall, we have derived two distinct analytic representations that give access to the meromorphic continuation of anisotropic Epstein zeta functions.
With the derivative Crandall representation, we have shown that the anisotropic Epstein zeta function is jointly holomorphic in the exponent $\nu$ and in complex cones of the vector arguments $\bm x$ and $\bm y$, with a simple pole in $\nu$ whose presence depends on whether $\bm y\in\Lambda^*$ and all components of $\bm\alpha$ are even.
Furthermore, we have shown that the anisotropic Epstein zeta function can be decomposed into a power-law singularity
and a holomorphic function near $\bm y=\bm 0$, allowing for its efficient precomputation through interpolation and computations without singularity related cancellation errors.

Our main result is a second, stably computable harmonic-decomposition Crandall representation that reaches the previously intractable strongly anisotropic regime.
Numerical benchmarks against independent reference values show  machine precision up to extremely high-order anisotropy $|\bm\alpha|=60$ and a linear increase in runtime for low- and mid-order anisotropy in one to three dimensions.
The method is based on a harmonic decomposition of the differential operator which rests on the backward stable evaluation of the harmonic polynomials, for which we derive an algorithm.

The anisotropic Epstein zeta function is now included
in EpsteinLib \cite{epsteinlib}
that is in active use by the quantum condensed matter community \cite{koziolQuantumAnnealingLattice2025,koziol2025melting,yadavObservationUnprecedentedFractional2025,buchheitEpsteinZetaMethod2025,buchheitZetaExpansionLongrange2026}.
Through the singular Euler-Maclaurin expansion, the anisotropic Epstein zeta function yields the exact difference between sums and integrals.
It further gives access to the Taylor series of the Epstein zeta function and the regularized Epstein zeta function,
enabling pre-computations on a grid and stabilizing computations near singularities.
We plan to include variants of Epstein zeta functions for
complex arguments and general geometries, such as lattices embedded in higher-dimensional space \cite{buchheitZetaExpansionLongrange2026}
 and continue our investigation into long-range interacting classical and quantum systems.

\section*{Acknowledgements}

We thank David Haink, Tage Malby, Gary Schmiedinghoff, Benedikt Fauseweh, and Sergej Rjasanow for insightful comments and valuable discussions that helped improve this work.

\section*{Declarations}

\subsection*{Funding}
\noindent
J.B.~acknowledges the support of the Quantum Fellowship Program of the German Aerospace Center (DLR) for funding their contribution to this work. A.B.~gratefully acknowledges support by the Klaus-Tschira Stiftung under Grant No. 00.025.2025.

\subsection*{Declaration of generative AI and AI-assisted technologies in the writing process}
\noindent
During the preparation of this work, the authors used Claude AI for grammar checking and minor language enhancements.
Furthermore, Claude Opus 5 was employed in an adversarial role and iteratively tasked with identifying errors and inconsistencies in drafts of the proofs.
Claude Opus 5 found the identity of the special case of the anisotropic Epstein zeta function as a Ramanujan tau $L$-function in \Cref{ramanujan}.

\appendix

\section{Absolute uniform convergence of set zeta derivatives}
\label{sec:zetadirectsum}

The following lemma shows that the sum in \Cref{defsetzeta} is well defined.

\begin{lemma}
\label{zetadirectsum}
Let $L\subseteq \mathds R^d$ be uniformly discrete,
let $\bm\alpha\in\mathds N^d$
and 
let $\nu\in \mathds C$, 
$\Re{\nu}>|\bm\alpha|+d$.
Then the differentiated series
$$
S_{\bm\alpha}(\bm y)=
\sums_{\bm z\in L} 
\diffop^{\bm\alpha}
e^{-2\pi i \bm z\cdot \bm y}|\bm z|^{-\nu}
,\qquad \bm y\in\mathds R^d
$$
converges absolutely and uniformly in $\bm y$ 
and $S_{\bm\alpha}(\bm y)=Z^{(\bm\alpha)}_{L,\nu}(\bm y)$
for $\bm y\in\mathds R^d$.
\end{lemma}

\begin{proof}
For any $\bm z\in L\setminus\{\bm 0\}$ and $\bm \beta\in\mathds N^d$ let
$f_{\bm z}(\bm y)=e^{-2\pi i \bm z\cdot\bm y}|\bm z|^{-\nu}$
so that
$
S_{\bm\beta}(\bm y)=
\textstyle{\sum_{\bm z\in L}^{\prime}}
f^{(\bm\beta)}_{\bm z}(\bm y)
$.
By the Weierstrass M-test we show absolute and uniform convergence of $S_{\bm\beta}(\bm y)$ in $\bm y$ for $\bm\beta\le\bm\alpha$;
iterated termwise differentiation is then justified at every step by uniform convergence of the differentiated series, so 
$Z^{(\bm\alpha)}_{L,\nu}(\bm y)=S_{\bm\alpha}(\bm y)$.

The summand function $f_{\bm z}(\bm y)$ is $|\bm\beta|$-times continuously differentiable in $\bm y$ with derivatives
$f^{(\bm\beta)}_{\bm z}(\bm y)=(-2\pi i \bm z)^{\bm\beta}e^{-2\pi i \bm z\cdot\bm y}|\bm z|^{-\nu}$
and the right-hand side of the bound
$$
|f^{(\bm\beta)}_{\bm z}(\bm y)|
\le 
(2\pi)^{|\bm\beta|}
|\bm z|^{-(s-|\bm\beta|)}
$$
is independent of $\bm y$, so from here it is enough to show absolute convergence; here
we used that $||\bm z|^{-\nu}|=|\bm z|^{-s}$
for $\Re{\nu}=s$
and $|\bm z^{\bm\beta}|\le |\bm z|^{|\bm\beta|}$.
We split the lattice sum into a truncated sum with truncation parameter $r>0$, and a series away from the origin
$$
\sideset{}{'}\sum_{\bm z\in L}|f^{(\bm\beta)}_{\bm z}(\bm y)|
\le
(2\pi)^{|\bm\beta|}
\Big(
\sideset{}{'}\sum_{\substack{\bm z\in L \\ |\bm z|\le r}}
|\bm z|^{-(s-|\bm\beta|)}
+
\sum_{\substack{\bm z\in L \\ |\bm z|>r}}
|\bm z|^{-(s-|\bm\beta|)}
\Big)\,,
$$
so it is enough to show that the second sum converges absolutely, as the index set of the first sum is bounded and uniformly discrete and thus contains only finitely many points.
Let $\varepsilon=\delta/\sqrt{d}$ with
$\delta= \inf\{|\bm x-\bm y|:\bm x,\bm y\in L,\ \bm y\neq \bm x\}$
and define $\bm n_{\bm z}\in\varepsilon\mathds Z^d$ by $\bm z\in \bm n_{\bm z}+[0,\varepsilon)^d$.
Let $r=3\delta$.
For $|\bm z|>r$, we have
$
|\bm n_{\bm z}|
\ge |\bm z|-\delta 
\ge 2\delta
$, which implies
$$
|\bm z|\ge |\bm n_{\bm z}|-\delta
\ge \frac{1}{2}|\bm n_{\bm z}|\,,
$$
and we have
$$
\sum_{\substack{\bm z\in L \\ |\bm z|>r}}
|\bm z|^{-(s-|\bm\beta|)}
\le 
2^{s-|\bm\beta|}
\sum_{\substack{\bm z\in L \\ |\bm z|>r}}
|\bm n_{\bm z}|^{-(s-|\bm\beta|)}\,.
$$
Note that the map 
$L\to\varepsilon\mathds Z^d$, $\bm z\mapsto\bm n_{\bm z}$
is injective, because the distance between points in a half-open lattice cell $\bm n_{\bm z}+[0,\varepsilon)^d$ of the scaled square lattice $\varepsilon\mathds Z$ is strictly less than the diameter $\sqrt{d}\varepsilon=\delta$, which is also the infimum of the distance in $L$.
Therefore, the right-hand side can be augmented to include the full scaled square lattice
$$
\sum_{\substack{\bm z\in L \\ |\bm z|>r}}
|\bm n_{\bm z}|^{-(s-|\bm\beta|)}
\le
\sideset{}{'}
\sum_{\bm n\in \varepsilon\mathds Z^d}|\bm n|^{-(s-|\bm\beta|)}
=
Z_{\varepsilon\mathds Z^d,s-|\bm \beta|}(\bm 0,\bm 0)\,,
$$
and the Epstein zeta function sum converges absolutely for $s-|\bm\beta|\ge s-|\bm\alpha|>d$ by
\cite[Lem. 1]{buchheitComputationPropertiesEpstein2026}.
\end{proof}

\section{Derivatives of Crandall functions}
\label{sec:crandall-derivatives}

The Faà di Bruno formula allows one to write the derivatives of the composition of a function with a squared argument as a weighted sum over the functions themselves.

\begin{lemma}[Faà di Bruno formula]
\label{derivative-fy2}
Let $\Omega\subseteq\mathds C$ be open and $U\subseteq \mathds C^d$ be open such that $\bm y^2\in \Omega$ for every $\bm y\in U$.
Let $f$ be holomorphic in $\Omega$ and let
 $\bm\alpha \in\mathds N^d$.
Then it holds that
$$
\diffop^{\bm\alpha}f(\bm y^2)
=
\sum_{2\bm\beta\le\bm\alpha}
w_{\bm\alpha,\bm\beta}(\bm y)
f^{(|\bm\alpha|-|\bm\beta|)}(\bm y^2)
$$
where $w_{\bm\alpha,\bm\beta}$ denotes the monomial
$$
w_{\bm\alpha,\bm\beta}(\bm y)=
(\bm\alpha-\bm\beta)_{\bm\beta}
\binom{\bm\alpha}{\bm\beta}
(2\bm y)^{\bm\alpha-2\bm \beta}\,.
$$
\end{lemma}

\begin{proof}
Let $g:U\to\Omega$,
$g(\bm y)=\bm y^2$,
then $f\circ g:U\to\mathds C$ is holomorphic, as it is a composition of holomorphic functions.
For $\bm\alpha=\bm 0$, the formula is the identity and for $\alpha^{(1)},\ldots,\alpha^{(d)}\le 1$ we have the componentwise chain rule.
Suppose there is some $n=\alpha^{(i)}\ge 2$. By the Faà di Bruno formula \cite{faadiBruno1855}, we have
$$
\frac{\partial^n}{\partial (y^{(i)})^n}f
=
\sum_{\substack{\bm\gamma\in\mathds N^n \\ \gamma^{(1)}+2\gamma^{(2)}+\ldots+n\gamma^{(n)}=n}}
\frac{n!}{\bm \gamma!}
f^{(|\bm\gamma|)}\circ g
\prod_{j=1}^n\Big(\frac{g^{(j\bm e_i)}}{j!}\Big)^{\gamma^{(j)}}
\,.
$$
Note that \cite{faadiBruno1855} naturally extends to this complex setting, as the proof relies only on the chain rule.
Further note that $g^{(j \bm e_i)}$ vanishes for $j>2$.
It is therefore sufficient to sum over $\bm\gamma\in\mathds N^2$ for $n=2$ and $\bm\gamma\in\mathds N^2\times \{0\}^{n-2}$ for $n>2$.
Consider such fixed $\bm\gamma$ and define $k=\gamma^{(2)}$, then $\gamma^{(1)}=n-2k$ and $k\le n/2$, as both $\gamma^{(1)}$ and $\gamma^{(2)}$ are non-negative.
As $g^{(\bm e_i)}(\bm y)/1!=2y^{(i)}$ and $g^{(2 \bm e_i)}(\bm y)/2!=1$, the restricted sum is given by
$$
\frac{\partial^n}{\partial (y^{(i)})^n}f(\bm y^2)
=
\sum_{k=0}^{\lfloor n/2\rfloor}
\frac{n!}{k!(n-2k)!}f^{(n-k)}(\bm y^2) (2y^{(i)})^{n-2k}\,.
$$
The one-dimensional result follows from
$
\frac{n!}{k!(n-2k)!}=
(n-k)_{k}
\binom{n}{k}
$.
For the $d$-dimensional result observe that all further derivatives only act on $f^{(n-k)}(\bm y^2)$, so that the result tensorizes.
\end{proof}

The following lemma is vital for the continuity statement in \Cref{derivative-crandall-representation}.

\begin{lemma}
\label{continuity-crandall-upper}
Let $m\in\mathds N$ and let $h_m\colon\mathds R^d\to\mathds R$ be a homogeneous polynomial
of degree $m$.
Let $\bm u\in\mathds R^d\setminus\{\bm 0\}$ and let $t\in\mathds C$, $\Re{t}<m$.
Then $h_m(\bm u)G_t(\bm u)$ continuously
extends to $\bm u=\bm 0$, and the limit vanishes if $m>0$ and is equal to  $-2h_0(\bm 0)/t$ if $m=0$.
\end{lemma}

\begin{proof}
If $\Re{t}<0$, 
we have that 
the upper Crandall function is continuous at zero by \Cref{crandall-upper-prop} so the product is continuous at $\bm 0$ and, since $h_m$ is homogeneous of degree $m$, the limit is zero whenever $m>0$
and, by the definition of the upper Crandall function, is equal to
$-2h_0(\bm 0)/t$ if $m=0$.
Let $0\le \Re{t}<m$; we show this implies that the product converges to zero.
Let $t\neq 0$ and use that by the fundamental relation, \Cref{crandall-lower-prop},
$$
h_m(\bm u)G_t(\bm u)
=
\Gamma(t/2)\frac{h_m(\bm u)}{(\pi\bm u^2)^{t/2}}
-
h_m(\bm u)g_t(\bm u)\,.
$$
By the smoothness of the lower Crandall function, see \Cref{crandall-lower-prop},
the second term vanishes in the $\bm u\to\bm 0$ limit.
For the first term note that
$|h_m(\bm u)|\le C|\bm u|^m$ 
with $C=\max_{|\bm w|=1}|h_m(\bm w)|$
so that
$$
|
h_m(\bm u)
(\pi\bm u^2)^{-t/2}
|
\le 
C
\pi^{-\Re t/2}
|
\bm u
|^{m-\Re t}
$$
vanishes since $\Re{t}<m$.
Finally, let $t=0$.
From \cite[Eq. (5.1.12) \& (6.5.9)]{abramowitz} we have
$$
G_{0}(\bm u)
=
-\gamma
-\log (\pi\bm u^2)-\sum_{\substack{j=1}}^{\infty}
\frac{(-\pi \bm u^2)^j}{j\;\!j!}\,,
$$
where $\gamma$ is the Euler–Mascheroni constant and
where the  series tends to zero as $\bm u\to\bm 0$ and so does its product with the polynomial.
For the logarithm, use that $|\log(x)|<x^{-\varepsilon}/\varepsilon$ for $0<x\le 1$ and $\varepsilon >0$.
Note that $t=0$ implies 
$
m>0
$
and choose $2\varepsilon <m$, so that for $\pi\bm u^2\le 1$
$$
|h_m(\bm u)
\log(\pi \bm u^2)|
\le
C\pi^{-\varepsilon}|\bm u|^{m-2\varepsilon}/\varepsilon
$$
which
converges to zero as $\bm u\to\bm 0$.
\end{proof}

By the Faà di Bruno formula, the upper Crandall derivatives are immediate from their integral representation.

\begin{proof}[Proof of \Cref{derivative-crandall-representation}]
Let $\bm u\in D$ and let $K \subseteq D$ be compact
such that $\bm u\in K$.
We employ the representation
$$
G_{\nu}(\bm u)=
\int_{-1}^{1}|t|^{-\nu}e^{-\pi\bm u^2/t^2}\frac{\mathrm{d} t}{|t|}\,.
$$
The integral over the differentiated integrand
$$
|t|^{-\nu-1}\diffop[u]^{\bm\alpha}e^{-\pi\bm u^2/t^2}
=
|t|^{-\nu-1}\sum_{2\bm\beta\le\bm\alpha}
(-\pi/t^2)^{|\bm\alpha|-|\bm\beta|}
w_{\bm\alpha,\bm\beta}(\bm u)
e^{-\pi \bm u^2/t^2}
$$
with $w$ as in \Cref{derivative-fy2}
is integrable uniformly in $K$.
We thus may differentiate under the integral and with  $(-\pi)^{|\bm\alpha|-|\bm\beta|}w_{\bm\alpha,\bm\beta}=p_{\bm\alpha,\bm\beta}$  
we obtain
$$
G_{\nu}^{(\bm\alpha)}(\bm u)=
\sum_{2\bm\beta\le\bm\alpha}
p_{\bm\alpha,\bm\beta}(\bm u)
\int_{-1}^{1}|t|^{-\nu-2(|\bm\alpha|-|\bm\beta|)}e^{-\pi\bm u^2/t^2}\frac{\mathrm{d} t}{|t|}\,,
$$
where the integral is equal to $G_{\nu+2|\bm\alpha|-2|\bm\beta|}(\bm u)$.
Let $\bm u\in\mathds R^d\setminus\{\bm 0\}$ and $\Re{\nu}+|\bm\alpha|<0$.
Note that for $2\bm\beta\le\bm\alpha$, $\bm\beta\in\mathds N^d$,
the polynomial $p_{\bm\alpha,\bm\beta}$ is homogeneous of degree $m=|\bm\alpha|-2|\bm\beta|$ 
and that for $t=\nu+2|\bm\alpha|-2|\bm\beta|$  we have
$\Re t -m =\Re \nu+|\bm\alpha|<0$.
Therefore, by \Cref{continuity-crandall-upper}, all summands
$p_{\bm\alpha,\bm\beta}(\bm u)G_{\nu+2|\bm\alpha|-2|\bm\beta|}(\bm u)$
of $G^{(\bm\alpha)}_{\nu}(\bm u)$
continuously extend to $\bm u=\bm 0$
and the limit 
$$
-2\chi(\bm\alpha)\frac{p_{\bm\alpha,\bm\alpha/2}(\bm 0)}{\nu+|\bm\alpha|}
$$
is non-zero if and only if $\chi(\bm\alpha)=1$.
Here we used that $p_{\bm\alpha,\bm\beta}(\bm 0)$ is non-zero if and only if $2\bm\beta=\bm\alpha$ and that such $\bm\beta\in\mathds N^d$, $2\bm\beta\le\bm\alpha$ exists if and only if $\chi(\bm\alpha)=1$.
Therefore, the display above
provides the unique meromorphic continuation of $G_{\nu}^{(\bm\alpha)}(\bm 0)$ to $\nu\in\mathds C$
with a simple pole at $\nu=-|\bm\alpha|$ that appears if and only if $\chi(\bm\alpha)=1$,
which concludes the proof.
\end{proof}

The lower Crandall derivative identity follows analogously.

\begin{proof}[Proof of \Cref{crandall-lower-der}]
Let $\nu\in\mathds C$ with $\Re{\nu}>0$ and let $\bm u\in\mathds C^d$.
Then, by using the integral representation
$$
g_{\nu}(\bm u)=
\int_{-1}^{1}|t|^{\nu}e^{-\pi t^2\bm u^2}\frac{\mathrm{d} t}{|t|}\,,
$$
the lower Crandall derivative identity 
$$
g_\nu^{(\bm\alpha)}(\bm u) = 
\sum_{2\bm\beta\le \bm\alpha}p_{\bm\alpha,\bm \beta}(\bm u)
g_{\nu+2|\bm\alpha|-2|\bm\beta|}(\bm u)\,,
$$
in \Cref{crandall-lower-der}
follows analogously to the upper Crandall derivative identity derived in the proof above.
The holomorphy of the lower Crandall function in \Cref{crandall-lower-prop} implies the holomorphy of its partial derivative in the same holomorphy region, see \cite[Cor. 2.2.2]{hormanderIntroductionComplexAnalysis1990}.
Therefore the left-hand side of lower Crandall derivative identity is jointly holomorphic in
$$
(\nu,\bm u)\in\big(\mathds C\setminus(-2\mathds N)\big)\times \mathds C^d
\,
$$
and
the right-hand side is jointly holomorphic in the extended region
$$
(\nu,\bm u)\in\big(\mathds C \setminus(-2|\lceil\bm\alpha/2\rceil|-2\mathds N)\big)\times \mathds C^d
$$
where we used that, here, $2\bm\beta\le\bm\alpha$, $\bm\beta\in\mathds N^d$ implies
$\nu+2|\bm\alpha|-2|\bm\beta|\notin -2\mathds N$
since $|\bm\alpha|-|\bm\beta|\ge |\lceil\bm\alpha/2\rceil|$.
By the identity theorem, the lower Crandall derivative identity holds in $\nu\in  \mathds C\setminus(-2|\lceil\bm\alpha/2\rceil|-2\mathds N)$
with the left-hand side at  $\nu=-2\ell$, $\ell\in\mathds N$, $\ell<|\lceil\bm\alpha/2\rceil|$ defined in terms of the right-hand side.
Note that by the integral representation $g_t(\bm 0)=2/t$, $t\in\mathds C$, $\Re{t}>0$, which, by the holomorphy in \Cref{crandall-lower-prop}, extends to $t\in\mathds C\setminus(-2\mathds N)$.
Let $\nu\in\mathds C\setminus(-2|\lceil\bm\alpha/2\rceil|-2\mathds N)$.
Evaluating the lower Crandall derivative at $\bm u=\bm 0$,
all summands beyond $2\bm\beta=\bm\alpha$ are zero, so that
with $g_{\nu+|\bm\alpha|}(\bm 0)=2/(\nu+|\bm\alpha|)$ we obtain
that
$$
2\chi(\bm\alpha)\frac{p_{\bm\alpha,\bm\alpha/2}(\bm 0)}{\nu+|\bm\alpha|}
$$
provides the unique meromorphic continuation of $g_{\nu}^{(\bm\alpha)}(\bm 0)$ to $\nu\in\mathds C$, with a simple pole at $\nu=-|\bm\alpha|$ that appears if and only if $\chi(\bm\alpha)=1$ by the same argument as in the proof above.
\end{proof}

\section{Harmonic decomposition of the singularity}
\label{sec:appendix-reg}

We begin by deriving an identity for the derivative of the falling Pochhammer symbol.

\begin{lemma}
\label{pochhammer-der}
Let $k\in\mathds N$ and $\nu\in\mathds C$.
Then
$$
\partial_\nu(\nu)_{k}=
\begin{cases}
(-1)^{k-\nu-1}\nu!(k-\nu-1)!
& \nu\in\{0,1,\ldots,k-1\}
\\
(\nu)_{k}(\psi(\nu+1)-\psi(\nu+1-k))
&\text{else}
\end{cases}\,,
$$
where $\psi$ is the digamma function
and $(\nu)_{k}$ is the falling Pochhammer symbol.
\end{lemma}

\begin{proof}
By the product rule for the derivative
$$
\partial_{\nu}(\nu)_{k}
=
\sum_{i=0}^{k-1}
\prod_{\substack{j=0 \\ j\neq i}}^{k-1} (\nu-j)\,.
$$
Let $\nu\notin\{0,1,\ldots,k-1\}$, then $(\nu)_{k}\neq 0$ and factoring out the Pochhammer symbol yields
$$
\partial_{\nu}(\nu)_{k}
=
(\nu)_{k}
\sum_{i=0}^{k-1}
(\nu-i)^{-1}\,.
$$
The sum on the right-hand side is the difference of digamma functions as stated, see
\cite[Eq. 6.3.6]{abramowitz}.
If, on the other hand, $\nu\in\{0,1,\ldots,k-1\}$, then only a single summand of the first display survives
$$
\partial_{\nu}(\nu)_{k}
=
\prod_{\substack{j=0\\j\neq \nu}}^{k-1}(\nu-j)
=
(-1)^{(k-1)-(\nu+1)+1}\nu!
\prod_{j=\nu+1}^{k-1}
|\nu-j|
$$
where the last product is the factorial $(k-\nu-1)!$.
\end{proof}

The derivatives of the falling Pochhammer symbol are useful in simplifying calculations involving derivatives of polynomials with logarithmic terms, as the proof of the following lemma shows.

\begin{lemma}
\label{delta-ysquared}
Let $\bm y\in\mathds R^d\setminus\{\bm 0\}$ and $k\in\mathds N$.
Let $\nu\in\mathds C$, then
$$
\Delta^k (\bm y^2)^\nu= 4^k (\nu)_{k}(\nu+d/2-1)_{k}(\bm y^2)^{\nu-k}\,.
$$
Let $\ell\in\mathds N$, $\ell<k$, then
$$
\Delta^k \Big[(\bm y^2)^\ell \log(\bm y^2)\Big]
=
-4^k\ell!(k-\ell-1)!
(\ell+d/2-1)_{k}
(-\bm y^2)^{\ell-k}\,.
$$
Let $\ell\in\mathds N$, $\ell\ge k$, then
$$
\Delta^k \Big[(\bm y^2)^\ell \log(\bm y^2)\Big]
=
4^{k}
(\ell)_{k}(\ell+d/2-1)_{k}
(\bm y^2)^{\ell-k}\Big(
\log(\bm y^2)
+
H_{\ell,k,d}
\Big).
$$
Here $H_{\ell,k,d}=H_{\ell}-H_{\ell-k}+\psi(\ell+d/2)-\psi(\ell+d/2-k)$,
where $H_\ell$ is the $\ell$'th harmonic number and $\psi$ is the digamma function.
\end{lemma}

\begin{proof}
Applying a single second derivative to $(\bm y^2)^\nu$ yields
$$
\partial_{j}^2(\bm y^2)^\nu
=
4 \nu(\nu-1)(y^{(j)})^2(\bm y^2)^{\nu-2}
+
2 \nu (\bm y^2)^{\nu-1}\,.
$$
Summing over $j$ we obtain
$$
\Delta (\bm y^2)^\nu
=4 \nu (\nu+d/2-1) (\bm y^2)^{\nu-1}\,,
$$
where we used that $2 d \nu + 4\nu (\nu-1)=4\nu(\nu+d/2-1)$,
and $(y^{(1)})^2+\ldots +(y^{(d)})^2=\bm y^2$.
The first part of the statement follows by induction.

For the second part note that $(\bm y^2)^\ell \log(\bm y^2)=\partial_{\nu}(\bm y^2)^{\nu}|_{\nu=\ell}$ so that by the first part of the proof
$$
\Delta^k (\bm y^2)^\ell \log(\bm y^2)
=
4^k
\partial_{\nu}\Big[
(\nu)_{k}(\nu+d/2-1)_{k}
(\bm y^2)^{\nu-k}\Big]_{\nu=\ell}\,,
$$
where we used that $(\bm y^2)^\nu$ is smooth in $(\nu,\bm y)\in \mathds C\times (\mathds R^d\setminus\{\bm 0\})$
so the operators $\Delta^k$ and $\partial_{\nu}$ commute.
Let $\ell<k$, then $(\ell)_{k}=0$, so that the only term that survives is
$$
\partial_{\nu}
\Big[ 
(\nu)_{k}(\nu+d/2-1)_{k}
(\bm y^2)^{\nu-k}
\Big]_{\nu=\ell}
=
(\ell+d/2-1)_{k}
(\bm y^2)^{\ell-k}
 \partial_{\nu}
(\nu)_{k}|_{\nu=\ell}
$$
and the statement follows from \Cref{pochhammer-der}.
Let $\ell\ge k$, so that in particular $\ell+d/2-1\notin \{0,1,\ldots,k-1\}$ and by \Cref{pochhammer-der} and the product rule
\begin{align*}
\partial_{\nu}
\Big[ 
(\nu)_{k}(\nu+d/2-1)_{k}
(\bm y^2)^{\nu-k}
\Big]_{\nu=\ell}
=&(\ell)_{k}(\ell+d/2-1)_{k}
H_{\ell,k,d}\,
(\bm y^2)^{\ell-k}
\\ 
&+
(\ell)_{k}(\ell+d/2-1)_{k}
\partial_{\nu}(\bm y^2)^{\nu-k}|_{\nu=\ell}
\end{align*}
and 
$
\partial_{\nu}(\bm y^2)^{\nu-k}|_{\nu=\ell}
=
(\bm y^2)^{\ell-k} \log(\bm y^2)
$ concludes the proof.
\end{proof}

We now derive how the differential operators obtained from homogeneous harmonic polynomials act on smooth functions with squared arguments.

\begin{lemma}
\label{harmonic-applied-to-f(y2)}
Let $\Omega\subseteq \mathds R$ be open and let $U\subseteq \mathds R^d$ be open such that
$\bm y^2\in \Omega$ for every $\bm y\in U$ and let $m\in\mathds N$.
Let $f:\Omega\to\mathds C$ be $m$-times differentiable and let
 $h_m$ be a homogeneous harmonic polynomial of degree $m$.
Then
$$
h_m(\diffop) f(\bm y^2)
=2^m h_m(\bm y)f^{(m)}(\bm y^2)\,.
$$
\end{lemma}

\begin{proof}
Note that $h_0(\bm y)=C_0$ is constant, so
$$
h_0(\diffop)f(\bm y^2)=C_0f(\bm y^2)=h_0(\bm y)f(\bm y^2)\,,
$$
and $h_1(\bm y)=\sum_{j=1}^d C_j y^{(j)}$, so by the product rule
$$
h_1(\diffop)f(\bm y^2)= \sum_{j=1}^d C_j 2 y^{(j)} f'(\bm y^2)=2 h_1(\bm y)f'(\bm y^2)\,.
$$
The higher-order one-dimensional harmonic polynomials vanish, so in what follows let $d\ge 2$.
We proceed by induction on the degree $n\in\mathds N$ with $1\le n\le m$, where we assume the statement holds for $n-1$.
Let $\bm a\in V$ be as in \Cref{harmonic-decomp}. Then, as a special case, the induction hypothesis holds for the harmonic polynomial $(\bm a\cdot\bm y)^{n-1}$.
By the product rule
$
(\bm a\cdot\diffop)^n f(\bm y^2)
$
is equal to
$$
2^{n-1}
(\bm a\cdot\diffop)
(\bm a\cdot \bm y)^{n-1}
f^{(n-1)}(\bm y^2)
+
2^{n}(\bm a\cdot\bm y)^{n}f^{(n)}(\bm y^2)\,.
$$
Here the first term vanishes, as
$$
(\bm a\cdot\diffop)
(\bm a\cdot \bm y)
=
\sum_{i=1}^d\sum_{j=1}^d a_i a_j \partial_{i} y^{(j)}
=
\bm a^2
$$
since $\partial_{i}y^{(j)}=\delta_{ij}$ and thus
$(\bm a\cdot\diffop)(\bm a\cdot \bm y)^{n-1}=(n-1)\bm a^2 (\bm a\cdot \bm y)^{n-2}=0$ as $\bm a^2= 0$.
It follows that
$$
(\bm a\cdot \diffop)^n f(\bm y^2)=
2^n(\bm a\cdot \bm y)^n f^{(n)}(\bm y^2)
$$
and choosing $V$ so that
$h_{n}(\bm y)=\sum_{\bm a\in V}c_{\bm a}(\bm a\cdot \bm y)^{n}$
concludes the proof by linearity.
\end{proof}

We conclude by applying \Cref{harmonic-applied-to-f(y2)} to the polynomials occurring in the singularity, including a logarithmic factor.

\begin{lemma}
\label{h-ysquared}
Let $h_m$ be a homogeneous harmonic polynomial of degree $m\in\mathds N$
and let $\bm y\in\mathds R^d\setminus\{\bm 0\}$.
Let $\nu\in\mathds C$, then
$$
h_m(\diffop)(\bm y^2)^\nu = 2^{m}(\nu)_{m} h_m(\bm y)(\bm y^2)^{\nu-m}
\,.
$$
Let $\ell\in\mathds N$ with $\ell<m$, then 
$$
h_m(\diffop)\Big[(\bm y^2)^\ell \log(\bm y^2)\Big]=
-
 2^{m}\ell!(m-\ell-1)! h_m(\bm y)(-\bm y^2)^{\ell-m}
\,.
$$
Let $\ell\in\mathds N$ with $\ell\ge m$, then
$$
h_m(\diffop)\Big[(\bm y^2)^\ell \log(\bm y^2)\Big]
=
2^{m} (\ell)_{m}\, h_m(\bm y) (\bm y^2)^{\ell-m}\Big( \log(\bm y^2) + H_{\ell,m}\Big),
$$
where $H_{\ell,m}=H_{\ell}-H_{\ell-m}$ and where $H_\ell$ is the $\ell$'th harmonic number.
\end{lemma}

\begin{proof}
The first part of the statement follows from \Cref{harmonic-applied-to-f(y2)} for $f(x)=x^\nu$.

Now let $f(x)=x^\ell\log x=\partial_{\nu}x^{\nu}|_{\nu=\ell}$ with $\ell\in\mathds N$. 
By \Cref{harmonic-applied-to-f(y2)} we have
$$
h_m(\diffop)f(\bm y^2)
=
2^{m} h_m(\bm y)\,\partial_{\nu}\Big[ (\nu)_{m} (\bm y^2)^{\nu-m}\Big]_{\nu=\ell}
$$
where by the same argument as in the proof of \Cref{delta-ysquared}, the operators $h_m(\diffop)$ and $\partial_{\nu}$ commute.
If $\ell<m$, the second product rule term $(\ell)_{m}\partial_{\nu}[(\bm y^2)^{\nu-m}]_{\nu=\ell}$ vanishes and with  \Cref{pochhammer-der} we obtain
$$
\partial_{\nu}\big[
(\nu)_{m} (\bm y^2)^{\nu-m}\big]_{\nu=\ell}
=
(-1)^{m-\ell-1}\ell!(m-\ell-1)!(\bm y^2)^{\ell-m}\,,
$$
and assembling yields the second part of the statement.
Let $\ell \ge m$ then by \Cref{pochhammer-der} and the product rule we have
that
$
\partial_{\nu}\big[
(\nu)_{m} (\bm y^2)^{\nu-m}\big]_{\nu=\ell}
$
is equal to
$$
(\ell)_{m}(\psi(\ell+1)-\psi(\ell+1-m))(\bm y^2)^{\ell-m}
+
(\ell)_{m}
(\bm y^2)^{\ell-m}\log(\bm y^2)
$$
where we used that $\partial_{\nu}(\bm y^2)^{\nu-m}\big|_{\nu=\ell}=(\bm y^2)^{\ell-m}\log(\bm y^2)$.
By the formula for the digamma function for natural arguments
\cite[Eq. 6.3.2]{abramowitz} we have
$
\psi(\ell+1)-\psi(\ell+1-m)
=H_{\ell,m}
$
and assembling concludes the proof.
\end{proof}

We can now prove the central statement regarding the harmonic decomposition of the singularity.

\begin{proof}[Proof of \Cref{harmonic-decomp-sing}]
Let $\nu\notin (d+2\mathds N)$.
Applying \Cref{delta-ysquared} to the singularity and grouping the terms we obtain
$$
\Delta^k 
\hat{s}_\nu(\bm y)=
(-4\pi^2)^k 
\hat{s}_{\nu-2k}(\bm y)
$$
where we used the factorial identities
$(-1)^k\Gamma(-a)(a)_{k}=\Gamma(k-a)$
and $(b-1)_{k}\Gamma(b-k)=\Gamma(b)$
for $a=(\nu-d)/2$ and $b=\nu/2$.
Let
$\ell\in\mathds N$, $\ell<k$.
By \Cref{delta-ysquared}, we have
$$
\Delta^k 
\hat s_{d+2\ell}(\bm y)= 
(-4)^k(k-\ell-1)!
(\ell+d/2-1)_{k}
\frac{\pi^{\ell+k+d/2}}{\Gamma(\ell+d/2)}( \pi \bm y^2 )^{\ell-k}.
$$
As $k>\ell$, we have $\nu-2k\notin (d+2\mathds N)$
for $\nu=d+2\ell$
and by using that $(k-\ell-1)!=\Gamma((d-\nu)/2+k)$
and 
$
(\ell+d/2-1)_{k}=\Gamma(\nu/2)/\Gamma(\nu/2-k)
$
we again obtain that
$$
\Delta^k 
\hat s_{d+2\ell}(\bm y)
=
(-4 \pi^2)^k
\hat
s_{\nu-2k}(\bm y)\,.
$$
Applying the differential operator obtained from the harmonic polynomials, by \Cref{h-ysquared} we have
$$
h_m(\diffop)\Delta^k 
\hat{s}_\nu(\bm y)
=
\frac{(-4\pi^2)^k
(-2\pi^2)^{m}}{(\nu/2-k-1)_{m}}
h_m(\bm y)
\hat{s}_{\nu-2k-2m}(\bm y)
$$
where
we used the factorial identity
$
(-1)^m\Gamma(-a')(a')_{m}=\Gamma(m-a')
$
where $a'=(\nu-d)/2-k$
as well as
$(b'-1)_{m} \Gamma(b'-m)=\Gamma(b')$
where $b'=\nu/2-k$.
Therefore, the first formula holds for
$\nu\notin (d+2k+2\mathds N)$.

Let $\ell\in\mathds N$, $\ell\ge k $.
We apply \Cref{delta-ysquared}
to the singularity and obtain
$$
\Delta^k 
\hat s_{d+2\ell}(\bm y)=
(-4\pi^2)^k
\hat s_{d+2(\ell-k)}(\bm y)\Big(1+\frac{H_{\ell,k,d}}{\log(\pi\bm y^2)}\Big)
$$
where
$H_{\ell,k,d}=H_{\ell}-H_{\ell-k}+\psi(\ell+d/2)-\psi(\ell+d/2-k)$
and
where we used the factorial identities
$\ell!=(\ell)_{k}(\ell-k)!$ 
and
$
\Gamma(\nu/2)
=
(\nu/2-1)_{k}
\Gamma(\nu/2-k)
$
for $\nu=d+2\ell$.

In the case of $\ell< m+k$,
the reciprocal of the logarithm cancels the logarithm in the singularity leaving a power law, so that
$$
h_{m}(\diffop)\hat s_{d+2(\ell-k)}(\bm y)\frac{H_{\ell,k,d}}{\log(\pi\bm y^2)}=0
$$
as
$
(\ell-k)_{m}=0
$
for $\ell < k+m$.
We have
$\nu-2k-2m\notin (d+2\mathds N)$
and obtain
$$
\Delta^k 
h_m(\diffop)
\hat{s}_{d+2\ell}(\bm y)
=
\frac{(-4\pi^2)^k
(-2\pi^2)^{m}}{(\nu/2-k-1)_{m}}
h_m(\bm y)
\hat{s}_{\nu-2k-2m}(\bm y)
$$
by the argument analogous to the above, as the differential operators commute.
The first formula therefore holds for $\nu\notin (d+2k+2m+2\mathds N)$.

In the case of $\ell\ge m+k$,
by \Cref{h-ysquared}
we obtain that
$
h_m(\diffop)
\Delta^k
\hat s_{d+2\ell}(\bm y)
$
is equal to
$$
\frac{(-4\pi^2)^k
(-2\pi^2)^{m}}{(\ell+d/2-k-1)_{m}}
h_m(\bm y)
\hat s_{d+2(\ell-k-m)}(\bm y)
\bigg( 
1
+
\frac{H_{\ell,k,d}+H_{\ell-k,m}}{\log (\pi  \bm y^2)}
\bigg)
$$
where
$H_{\ell,m}=H_{\ell}-H_{\ell-m}$ and
where we used the factorial identity
$(\ell-k)!=(\ell-k)_{m}(\ell-k-m)!$.
The statement follows from
$
H_{\ell,k,d}+H_{\ell-k,m}
=
H_{\ell,k,m+2k,d}
$.
\end{proof}

Finally, the following lemma allows one to skip branches in the algorithm where the harmonic polynomials vanish.

\begin{lemma}
\label{h-vanish}
Let $\bm\alpha\in\mathds N^d$ and let $k\in \{0,1,\ldots \lfloor |\bm\alpha|/2\rfloor\}$
such that either $\chi(\bm\alpha)=0$ or $2k\neq |\bm\alpha|$.
Then $h_{k,\bm\alpha}(\bm 0)=0$, where $h_{k,\bm\alpha}$ is the harmonic polynomial defined in \Cref{harmonic-decomp}.
\end{lemma}

\begin{proof}
Since $h_{k,\bm\alpha}$ is a homogeneous polynomial of degree $|\bm\alpha|-2k$ we have $h_{k,\bm\alpha}(\bm 0)=0$ for $2k\neq |\bm\alpha|$.
Let $\chi(\bm\alpha)=0$, then each summand in
the explicit representation of $h_{k,\bm\alpha}(\bm 0)$ in \Cref{harmonic-comp}
has a vanishing factor $\bm 0^{2\bm\gamma-\bm\alpha}=0$, $\bm\gamma\in\mathds N^d$, $2\bm\gamma-\bm\alpha\ge \bm 0$
since $\chi(\bm\alpha)=0$ implies that there is a factor with an odd exponent $0^{2\gamma_i-\alpha_i}=0$.
\end{proof}

\section{Derivation of the details on the algorithm}
\label{sec:appendix-alg}

\begin{proof}[Proof of \Cref{harmonic-comp}]
Let $n=|\bm\alpha|$.
We use the representation
\cite[Th. 5.21]{axlerHarmonicFunctionTheory2001}
that states that
$$
h_{k,\bm\alpha}(\bm y)=
\sum_{j=k}^{\lfloor n/2\rfloor}
|\bm y|^{2(j-k)}
\overline{c}_{k,n,j}
\,\Delta^{j}\,\bm y^{\bm\alpha}\,,
$$
with the recursively defined coefficients 
$$
\overline{c}_{0,n,0}=1,\qquad
\overline{c}_{k,n,j}
=
\begin{cases}
\frac{\overline{c}_{k,n,j-1}}{2(k-j)(2n+d-2-2k-2j)}& j=k+1,\ldots,\lfloor n/2\rfloor\\
\frac{(2n+d-2k)\overline c_{k-1,n,k-1}}{2k(2n+d+2-4k)(2n+d-4k)}&j=k>0
\end{cases}
$$
for $n,j,k\in\mathds N$, $j\ge k$. 
By inserting
$$
\diffop^{2\bm\beta}\bm y^{\bm\alpha}=
	\begin{cases}
    (\bm\alpha)_{2\bm\beta}
\,\bm y^{\bm\alpha-2\bm\beta}
&2\bm\beta\le\bm\alpha \\
0&\text{else}\\
\end{cases}
$$
on the right side of the 
multinomial expansion
$$
\Delta^{j}\,\bm y^{\bm\alpha}
=
j!
\sum_{\substack{\bm\beta\ge\bm 0\\|\bm\beta|=j}}
\frac{1}{\bm\beta!}
\diffop^{2\bm\beta}\bm y^{\bm\alpha},
$$
and using that
$
(\bm\alpha)_{2\bm\beta}
=\bm\beta!(\bm\alpha-\bm\beta)_{\bm\beta}
\binom{\bm\alpha}{\bm\beta}
$,
we obtain
$$
h_{k,\bm\alpha}(\bm y)=
\sum_{\substack{2\bm\beta\le \bm\alpha\\|\bm\beta|\ge k}}
(\bm\alpha-\bm\beta)_{\bm\beta}
\,
\overline{c}_{k,n,|\bm\beta|}
\,
|\bm\beta|!
\binom{\bm\alpha}{\bm\beta}
|\bm y|^{2(|\bm\beta|-k)}\bm y^{\bm\alpha-2\bm\beta}\,.
$$
By a multinomial expansion of $(\bm y^2)^{|\bm\beta|-k}$ 
and the index substitution $\bm\gamma\mapsto\bm\alpha+\bm\beta-\bm\gamma$
we obtain
$$
|\bm y|^{2(|\bm\beta|-k)}
\bm y^{\bm\alpha-2\bm\beta}
=
\sum_{\substack{
\bm\gamma\ge\bm\alpha-\bm\beta
\\
|\bm\gamma|=n-k
}}
\frac{(|\bm\beta|-k)!}{(\bm\gamma-\bm\alpha+\bm\beta)!}
\bm y^{2\bm\gamma-\bm\alpha}\,.
$$
Insertion into  $h_{k,\bm\alpha}(\bm y)$
and exchanging both sums
yields 
$$
h_{k,\bm\alpha}(\bm y)=
\sum_{\substack{
2\bm\gamma\ge\bm\alpha
\\
|\bm\gamma|=n-k
}}
\bm y^{2\bm\gamma-\bm\alpha}
\sum_{\substack{
\bm\beta\ge\bm\alpha-\bm\gamma
\\
\bm\beta\le\bm\alpha/2
\\
|\bm\beta|\ge k
}}
(\bm\alpha-\bm\beta)_{\bm\beta} 
\,
\overline c_{k,n,|\bm\beta|}
\,
|\bm\beta|!
\binom{|\bm\beta|-k}{\bm\gamma-\bm\alpha+\bm\beta}
\binom{\bm\alpha}{\bm\beta}
\,.
$$
The index substitution $\bm\beta\mapsto\bm\alpha+\bm\beta-\bm\gamma$ yields
that $h_{k,\bm\alpha}(\bm y)$ is equal to
$$
\!\!\!
\sum_{\substack{
2\bm\gamma\ge\bm\alpha
\\
|\bm\gamma|=n-k
}}
\!
\bm y^{2\bm\gamma-\bm\alpha}
\!\!\!\!\!\!\!\!\!
\sum_{\substack{
\bm\beta\ge\bm0
\\
\bm 0 
\le 
\bm\alpha+\bm\beta-\bm\gamma
\le \bm\alpha/2
}}
\!\!\!\!\!\!
(\bm\gamma-\bm\beta)_{\bm\alpha+\bm\beta-\bm\gamma}
\,
\overline{c}_{k,n,|\bm\beta|+k}
\,
(|\bm\beta|+k)!
\frac{|\bm\beta|!}{\bm\beta!}
\binom{\bm\alpha}{\bm\alpha+\bm\beta-\bm\gamma}
\,.
$$
The representation follows from rewriting the coefficients explicitly as
$$
\overline{c}_{k,n,m+k}=
\frac{(-2)^{-m}}{k!m!}\omega_{k,n}q_{k,n,m}\,,
$$
and that $(|\bm\beta|+k)!/(k!|\bm\beta|!)=\binom{|\bm\beta|+k}{k}$ with $m=|\bm\beta|$.
Finally notice that
the Pochhammer symbol and $q_{k,|\bm\alpha|,|\bm\beta|}$ are products of integers, the binomials are integers and that $|\bm\beta|!/\bm\beta!$ is a multinomial coefficient and thus an integer.
Therefore all the summands in $c_{k,\bm\alpha,\bm\gamma}$ are a product of integers divided by $2^{|\bm\beta|}$, hence a dyadic rational.
\end{proof}

We now show the validity of the reordering of the terms in \Cref{high-exp} for the harmonic-decomposition Crandall representation, \Cref{harmonic-decomposition-crandall-representation}, that stabilizes computations for large exponents.

\begin{proof}[Proof of \Cref{high-exp}]
We first derive the statement with the sums on both sides restricted to $k\notin \nu/2+\mathds N$.
Let $\bm z\in M$ with $\bm z\neq \bm 0$.
Then by the functional equation for the Crandall functions in \Cref{crandall-lower-prop}, split the upper Crandall function into the singularity and the lower Crandall function
$$
G_{\nu-2k}(\bm z/\lambda)
=
\Gamma(\nu/2-k)\pi^{-(\nu/2-k)}\lambda^{\nu-2k}|\bm z|^{2k-\nu}
-
g_{\nu-2k}(\bm z/\lambda)\,.
$$
Since $M$ is bounded we may freely exchange the sums and refactoring yields 
\begin{align*}
&\sum_{\substack{k=0\\ k\notin \nu/2+\mathds N}}^{\lfloor |\bm\alpha|/2\rfloor}
\frac{(\pi/\lambda^2)^{(\nu-2k)/2}}{\Gamma((\nu-2k)/2)}
\sums_{\bm z\in M}
h_{k,\bm\alpha}(\bm z)G_{\nu-2k}(\bm{z}/\lambda)e^{-2\pi i\bm y\cdot (\bm z+\bm x)}
\\
=&\sums_{\substack{\bm z\in M}}
|\bm z|^{-\nu}
e^{-2\pi i\bm y\cdot (\bm z+\bm x)}
\sum_{\substack{k=0\\ k\notin \nu/2+\mathds N}}^{\lfloor |\bm\alpha|/2\rfloor}
h_{k,\bm\alpha}(\bm z)
|\bm z|^{2k}
\\
&-\sum_{\substack{k=0\\ k\notin \nu/2+\mathds N}}^{\lfloor |\bm\alpha|/2\rfloor}
\frac{(\pi/\lambda^2)^{(\nu-2k)/2}}{\Gamma((\nu-2k)/2)}
\sums_{\bm z\in M}
h_{k,\bm\alpha}(\bm z)g_{\nu-2k}(\bm{z}/\lambda)e^{-2\pi i\bm y\cdot (\bm z+\bm x)}\,.
\end{align*}
Since $G_{\nu-2k}(\bm 0)=-g_{\nu-2k}(\bm 0)$, we can include the $\bm z=\bm 0$ summand in the sum over $M$ on the left-hand side and on the last line of the right-hand side in the equation above.
Note that on the left-hand side of the equation above, the summands are bounded in $k\in \nu/2+\mathds N_+$ and that these terms vanish due to the zeros of the reciprocal Gamma function, so they can be included in the left-hand side of the statement.
\end{proof}

Here we show the limit-identity for the special $k$-summands in the harmonic-decomposition Crandall representation.

\begin{proof}[Proof of \Cref{special-k-sum}]
If $\bm x\notin\Lambda$ and $\bm y\notin\Lambda^*$,
by the properties of the lower and upper Crandall functions, see \Cref{crandall-lower-prop} and \Cref{crandall-upper-prop}, each individual term in the function is entire in $\nu$ so that the limit of the product with $1/\Gamma(\nu/2-k)$ vanishes as $\nu\to 2k-2\ell$, $\ell\in\mathds N$ due to the zeros of the reciprocal Gamma function.
For $\bm y \in \Lambda^*$ we have an additional $\bm p=\bm 0$ summand in the reciprocal sum that, without the prefactors, reads
$$
h_{k,\bm\alpha}(\bm 0)
G_{d-\nu+2(|\bm\alpha|-k)}(\bm 0)
=
-\frac{2h_{k,\bm\alpha}(\bm 0)}{d-\nu+2(|\bm\alpha|-k)}\,,
$$
that is meromorphic in $\nu\in\mathds C$ with a single pole
where the denominator vanishes.
Since $\nu=2k-2\ell$, $\ell\in\mathds N$ and $k\le \lfloor |\bm\alpha|/2\rfloor$ imply $\nu\neq d+2(|\bm\alpha|-k)$,
this term multiplied by $1/\Gamma(\nu/2-k)$ also vanishes  there.
If $\bm x\in\Lambda$ we have an
additional $\bm z=\bm 0$ summand in the real sum that, with all prefactors, reads
$$
\frac{(\pi/\lambda^2)^{\nu/2-k}}{\Gamma(\nu/2-k)}
G_{\nu-2k}(\bm 0)
h_{k,\bm\alpha}(\bm 0)
e^{-2\pi i\bm y\cdot \bm x}
$$
where $G_{\nu-2k}(\bm 0)$ is holomorphic in $\nu\in\mathds C\setminus\{2k\}$ and all other factors are entire in $\nu$.
Therefore, the term vanishes as $\nu\to 2k-2\ell$ with $\ell>0$ and as $\nu\to 2k$ we have
$$
\frac{(\pi/\lambda^2)^{\nu/2-k}}{\Gamma(\nu/2-k)}G_{\nu-2k}(\bm 0)
=
-\frac{(\pi/\lambda^2)^{\nu/2-k}}{\Gamma(\nu/2-k)(\nu/2-k)}
=-\frac{(\pi/\lambda^2)^{\nu/2-k}}{\Gamma(\nu/2-k+1)}\to -1
\,,
$$
while the remaining factors tend to $h_{k,\bm\alpha}(\bm 0)e^{-2\pi i \bm x\cdot\bm y}$.
By \Cref{h-vanish} we have $h_{k,\bm\alpha}(\bm 0) = 0$ if either $2k\neq |\bm\alpha|$ or $\chi(\bm\alpha) = 0$, which concludes the proof.
\end{proof}

We now show the convergence of the  Taylor series in \Cref{taylor}.

\begin{proof}[Proof of \Cref{taylor}]
Let $\bm\alpha\in\mathds N^d$.
By the relation between anisotropic and differentiated regularized Epstein zeta functions, \Cref{aniso-reg-derivative-id},
and by the equivalence 
of regularized anisotropic and non-regularized anisotropic Epstein zeta functions at $\bm y=\bm 0$, \Cref{aniso-equals-reg-in-zero},
and since $\nu\notin d+2\mathds N$ implies 
$\nu\neq d+|\bm\alpha|$ or $\chi(\bm\alpha)=0$,
we have
$$
\diffop^{\bm\alpha}
[Z^{\rm reg}_{\Lambda,\nu}(\bm x,\bm y)]_{\bm y=\bm 0}
=
Z_{\Lambda,\nu,\bm\alpha}(\bm x,\bm 0)
$$
which yields the representation of the Taylor series in terms of anisotropic Epstein zeta functions.

Let $r=\lambda_1^*/\sqrt{2}$ and let $\bm y\in\mathds R^d$, $|\bm y|<r$ be fixed.
Choose $0<\varepsilon\le 1$ small enough such that
$
\varepsilon(2\|\bm y\|_1+d) < r^2-|\bm y|^2
$
and consider the polydisc
$$
P
=\{
\bm u\in\mathds C^d
:
 |u^{(j)}|<|y^{(j)}|+\varepsilon,\ 1\le j\le d
\}\,.
$$
We now show that $P$ is a subset of the region of holomorphy
$D_{\Lambda^*\setminus\{\bm 0\}}$
of the regularized anisotropic Epstein zeta function, see \Cref{hol-reg}.
Let $\bm u\in P$.
Then
$$
|\bm u|^2
<
|\bm y|^2+2\varepsilon \|\bm y\|_1+d\varepsilon^2<r^2
\,
$$
and,
for every $\bm z\in\Lambda^*\setminus\{\bm 0\}$,
we have
$
|\Re{\bm u}|+|\Im{\bm u}| 
\le \sqrt{2}|\bm u|
< |\bm z|
$
which implies
 $|\Re{\bm u}-\bm z|>|\Im{\bm u}|$
 and therefore $\bm u\in D_{\Lambda^*\setminus\{\bm 0\}}$.
By \cite[Th. 2.2.6]{hormanderIntroductionComplexAnalysis1990},
the Taylor series of the regularized Epstein zeta function converges normally in $P$ which, since $\bm y\in P$, implies point-wise convergence at $\bm y$.
\end{proof}

We finally show the equivalence of the condition of the early zero return in the third step of \Cref{alg} with the conditions in \Cref{syms-zero}.

\begin{proof}[Proof of \Cref{equv-sym-zeros}]
Define $W=\{z^{(j)}\colon \bm z\in\Lambda\}$, which is a subgroup of $\mathds R$.
Due to the mirror symmetry in the $j$'th component,
$\bm z\in\Lambda$ implies $\bm z-2z^{(j)}\bm e_j\in\Lambda$,
which implies $2W\subseteq\{t:t\bm e_j\in\Lambda\}$.
Therefore $2W$ is discrete,
and so $W$ is a discrete subgroup of $\mathds R$ and there exists some $\mu>0$ such that
$W=\mu\mathds Z$.
Let $t\in\mathds R$.
By the definition of the reciprocal lattice,
 $t\bm e_j\in\Lambda^*$ if and only if 
$tz^{(j)}\in\mathds Z$ for all $\bm z\in\Lambda$,
that is if and only if $t W\subseteq \mathds Z$.
Since $W=\mu\mathds Z$, the last inclusion is equivalent to $t\in\frac 1\mu\mathds Z$
and we have shown that
$$\Lambda^*\cap \mathds R\bm e_j=\frac 1\mu\mathds Z\bm e_j
\,.
$$
Therefore
$
(\Lambda^*\cap \mathds R\bm e_j)
\cdot \bm u
\subseteq \mathds Z
$
is equivalent to $u^{(j)}\in W=-W$
which is equivalent to the existence of some $\bm v \in\Lambda$ such that $u^{(j)}+v^{(j)}=0$.

Let $\bm w\in\Lambda$.
Due to the translational invariance of $\Lambda$,
some $\bm v\in\Lambda$ with $u^{(j)}+v^{(j)}=0$ exists if and only if some $\tilde{\bm v}\in\Lambda$ with
$(u^{(j)}+w^{(j)})+\tilde v^{(j)}=0$ exists, as is seen by setting $\tilde{\bm v}=\bm v-\bm w$.

Finally note that the reflection of the $j$'th component, that is the diagonal matrix $Q\in\mathds R^{d\times d}$ with entries $1$, where the $j$'th entry is replaced by $-1$, is orthogonal
so that $Q\bm p\cdot \bm z=\bm p\cdot Q\bm z\in \mathds Z$, $\bm p\in\Lambda^*$, $\bm z\in\Lambda$.
This implies $Q\Lambda^*\subseteq\Lambda^*$, 
and since $Q$ is also an involution we have the equality
so that $\Lambda^*$ is mirror-symmetric in the $j$'th component as well
and, using $(\Lambda^*)^*=\Lambda$, the statement holds when replacing $\Lambda$ with $\Lambda^*$.
\end{proof}

\section{Truncation}
\label{sec:appendix-trunc}

We now derive a rigorous error bound remainder of the harmonic-decomposition Crandall representation, \Cref{harmonic-decomposition-crandall-representation}, for some truncation value.
We begin with the properties of upper Crandall functions for real parameters.

\begin{lemma}
\label{crandall-real-prop}
Let $\nu \in \mathds{R}$ and let $\bm u \in \mathds{R}^d\setminus\{\bm 0\}$.
Then $G_{\nu}(\bm u)$ is strictly positive,
rotationally symmetric in $\bm u$, strictly decreasing in $|\bm u|$ and strictly increasing in $\nu$.
\end{lemma}

\begin{proof}
These are immediate from 
$$
G_{\nu}(\bm u)=\int_{-1}^1|\eta|^{-\nu}e^{-\pi \bm u^2/\eta^2}\,\frac{\rm d\eta}{|\eta|}
$$
as the integrand is strictly positive,
rotationally symmetric in $\bm u$, strictly decreasing in $|\bm u|$ and, since $|\eta|\le 1$, strictly increasing in $\nu$ and since
the integral  inherits these properties.
\end{proof}

For the investigation of the summation cutoff, it is sufficient to consider the lattice $\mathds{Z}^d$, as the following lemma shows.

\begin{lemma}
\label{lattice-to-zd-bound}
Let $\Lambda=A\mathds Z^d$ for $A\in\mathds R^{d\times d}$ regular, let $t\ge 0$, let $\nu\in\mathds R$ and let $\mu>0$.
Then, for any set $M\subseteq\mathds R^d\setminus\{\bm 0\}$,
and $\bm u\in\mathds R^d$, we have
$$
\sum_{\bm z \in \Lambda \cap M} |\bm z+\bm u|^tG_\nu(\mu(\bm z+\bm u)) \leq 
\|A\|^t
\sum_{\bm w \in \mathds{Z}^d \cap A^{-1}M} 
|\bm w+A^{-1}\bm u|^t
G_\nu\left(\frac{\mu(\bm w +A^{-1}\bm u)}{\Vert A^{-1}\Vert} \right),
$$
with $\|\cdot\|$ the spectral norm.
\end{lemma}

\begin{proof}
We have that
$G_\nu(\bm \cdot)$ is rotationally symmetric and strictly decreasing in the norm of its vector argument, see \Cref{crandall-real-prop}, so an upper bound on this factor follows from
$$
|A(\bm w + A^{-1} \bm u)|
\ge 
\| A^{-1}\|^{-1}|\bm w+A^{-1}\bm u|
$$
where we used that
$
|\bm u|=|A^{-1}A\bm u|\le \|A^{-1}\||A\bm u|
$.
Similarly, $|\bm z+\bm u|^t$ is increasing since $t\ge 0$ and the upper bound on that factor follows from
\[
|A(\bm w + A^{-1} \bm u)|
\le 
\|A\||\bm w+A^{-1}\bm u|
\,,
\qedhere 
\]
\end{proof}

The following statement bounds the tail of lattice sums over upper Crandall functions with a strictly increasing power-law prefactor.

\begin{lemma}
\label{crandall-weight-lattice-sum-bound}
Let $\bm u\in\mathds R^d$, $d\in\mathds N_+$,
let $\nu\in\mathds R$,
let $r,\mu>0$
and let $t\ge d$.
Then for any $\varepsilon>0$ with 
$
\varepsilon \le (r-\sqrt{d})/2
$
and
$
\mu(r-\varepsilon)>\sqrt{t/(2\pi)}
$
it holds that
$$
\sum_{\substack{
\bm w\in\mathds Z^d
\\
|\bm w+\bm u|>r
}}
|\bm w+\bm u|^{t-d}
G_{\nu}(\mu(\bm w+\bm u))
\le
\frac{C_t}{\varepsilon}
r^{t+1}
\frac{G_{t+1}(\mu(r-\varepsilon)) - G_\nu(\mu(r-\varepsilon))}{t+1-\nu}
\,
$$
where $C_t=(3/2)^t\pi^{d/2}/\Gamma(d/2+1)$
and where for $\nu=t+1$ the right-hand side is understood in the sense of the limit.
\end{lemma}

\begin{proof}
Denote by $S$ the left-hand side of the equation in the statement
which we rewrite as sums over spherical shells of width 
$\varepsilon>0$
$$
S
=
\sum_{n=0}^{\infty}
\sum_{\substack{\bm w\in \mathds Z^d\\ |\bm w + \bm u|> r+n\varepsilon\\|\bm w + \bm u|\le  r+(n+1)\varepsilon }} |\bm w+\bm u|^{t-d}
G_{\nu}(\mu(\bm w+\bm u))\,.
$$
Under the condition $\varepsilon\le (r-\sqrt{d})/2$
we showed in \cite[Lem. A.2]{buchheitComputationPropertiesEpstein2026} that the number of lattice points within each shell, that is at most the number of the points within the closed ball of radius $r+(n+1)\varepsilon$, is bounded by
$
N_{\rm shells}(r+n\varepsilon)^d
$
with $N_{\rm shells}=(3/2)^d \pi^{d/2}/\Gamma(d/2+1)$.
By \Cref{crandall-real-prop} we have
 that the summands of $S$ are rotationally symmetric, we have that $G_{\nu}(\bm \cdot)$ is globally decreasing in $|\bm \cdot|$
 and, since $t\ge d$ and $\varepsilon\le (r-\sqrt{d})/2\le r/2$
imply
 $(r+(n+1)\varepsilon)^{t-d}\le (3/2)^{t-d}(r+n\varepsilon)^{t-d}$,
 we have
$$
S
\le
\frac{C_t}{\mu^t}
\sum_{n=0}^{\infty}
(\mu(r+n\varepsilon))^tG_{\nu}(\mu(r+n\varepsilon))\,,
$$
where we used that the powers of $3/2$ cancel each other.
We now show that the summand on the right-hand side of the equation above
$$
u^t G_{\nu}(u)
,\qquad 
u=\mu(r+n\varepsilon)
$$
is monotonically decreasing as a function of the continuous variable $n\in [-1,\infty)$ by showing that its derivative with respect to $n$ is negative.
By the representation for the upper Crandall derivatives, see \Cref{derivative-crandall-representation}, we have
$G_\nu'(u)=-2\pi u G_{\nu+2}(u)$
so that by the chain rule and the product rule and since $u>0$
$$
\frac 1{\mu\varepsilon}\frac{\mathrm{d}}{\mathrm{d}n}u^t G_{\nu}(u)
=
u^t\Big(\frac{t}{u}G_{\nu}(u)-2\pi u  G_{\nu+2}(u)\Big)\,.
$$
The bracketed term on the right-hand side of the equation above is negative since $G_{\nu+2}(u)>G_{\nu}(u)$ by the monotonicity of the upper Crandall function in $\nu$, see \Cref{crandall-real-prop},
and since
$
t/u<2\pi u
$
is implied by
$u>\mu(r-\varepsilon)>\sqrt{t/(2\pi)}$.
Therefore, by monotonicity we have
$$
S\le 
\frac{C_t}{\mu^t}
\int_{-1}^{\infty}
(\mu(r+n\varepsilon))^tG_{\nu}(\mu(r+n\varepsilon))
\, \d n
$$
where the integral is bounded near $n=-1$ since $t\ge 1$ and therefore integrable by the superexponential decay of the upper Crandall function, see \Cref{crandall-upper-prop}.
With the substitution $u=\mu(r+n\varepsilon)$ we have
$$
S\le 
\frac{C_t}{\mu^{t+1}\varepsilon}
\int_{\mu(r-\varepsilon)}^{\infty}
u^{t+1}G_{\nu}(u)
\, \frac{\d u}{u}
$$
and by the integral identity \cite[Lem. A.4]{buchheitComputationPropertiesEpstein2026} we have that for $\nu\neq t+1$ the right-hand side of the equation above
is equal to
$$
\frac{C_t}{\varepsilon}
(r-\varepsilon)^{t+1}
\frac{G_{t+1}(\mu(r-\varepsilon)) - G_\nu(\mu(r-\varepsilon))}{t+1-\nu}
$$
and therefore smaller than the right-hand side of the equation of the statement since $\varepsilon>0$.
Here the limit $\nu\to t+1$ is well defined
due to the continuity of the integral in $\nu\in\mathds R$.
\end{proof}

The polynomial weights in the two preceding statements appear due to the following upper bound on the homogeneous harmonic polynomials in \Cref{harmonic-decomp}.

\begin{lemma}[Harmonic upper bound]
\label{harmonic-upper-bound}
Let $\bm\alpha\in\mathds N^d$,
let $k=0,1,\ldots,\lfloor|\bm\alpha|/2\rfloor$
and let $m=|\bm\alpha|-2k$.
Then
$$
|h_{k,\bm\alpha}(\bm u)|
\le 
\binom{d+m-1}{d-1}^{1/2}|\bm u|^{m}
,\qquad 
\bm u\in\mathds R^d
$$ 
where $h_{k,\bm\alpha}$ are as in \Cref{harmonic-decomp}.
\end{lemma}

\begin{proof}
Let $\bm u\in\mathds R^d$ and let $\bm w\in S^{d-1}$.
Since $h_{k,\bm\alpha}$ is homogeneous of degree $m$, we have
$|h_{k,\bm\alpha}(\bm u)|\le M_{m,\bm\alpha} |\bm u|^m$
with
$$
M_{m,\bm\alpha}=\sup_{\bm w\in S^{d-1}}|h_{k,\bm\alpha}(\bm w)|\,,
$$
where $S^{d-1}$ is the unit sphere in $\mathds R^d$.
We show that $M_{m,\bm\alpha}\le \binom{d+m-1}{d-1}^{1/2}$.
Denote by
$Z_{m}(\bm\cdot,\bm w)$ a $d$-dimensional zonal harmonic of degree $m$, see \cite[Ch. 5]{axlerHarmonicFunctionTheory2001}, that is a $d$-dimensional harmonic polynomial of degree $m$ that satisfies
$$
h_{k,\bm\alpha}(\bm w)
=
\left(h_{k,\bm\alpha}(\bm \cdot),Z_m(\bm\cdot,\bm w)\right)_{L^2(\mathrm{d}\sigma)}
,\qquad 
\bm w\in S^{d-1}
$$
with $\left(\cdot,\cdot\right)_{L^2(\mathrm{d}\sigma)}$ the inner product on the Lebesgue space of square-integrable functions on $S^{d-1}$ with the normalized surface measure $\sigma(S^{d-1})=1$.
Then by the Cauchy-Schwarz inequality
$$
|h_{k,\bm\alpha}(\bm w)|
\le 
\|h_{k,\bm\alpha}\|_{L^2(\mathrm{d}\sigma)}
\|Z_m(\bm \cdot,\bm w)\|_{L^2(\mathrm{d}\sigma)}
\,.
$$
By
\cite[Prop. 5.27(e)]{axlerHarmonicFunctionTheory2001},
and since $\sigma(S^{d-1})=1$,
we have that
$
\|Z_m(\bm \cdot,\bm w)\|_{L^2(\mathrm{d}\sigma)}^2
$
is bounded by the dimension of the space of harmonics of degree $m$, and therefore in particular by the dimension $\binom{d+m-1}{d-1}$ of the space of homogeneous polynomials of degree $m$, see 
\cite[Prop. 5.8]{axlerHarmonicFunctionTheory2001}.
What is left is to show that $\|h_{k,\bm\alpha}\|_{L^2(\mathrm{d}\sigma)}\le 1$.

By the harmonic decomposition, \Cref{harmonic-decomp}, and since $|\bm w|=1$ we have
$$
\bm w^{\bm\alpha}
=\sum_{j=0}^{\lfloor|\bm\alpha|/2\rfloor}
h_{j,\bm\alpha}(\bm w)\,,
$$
so that by the bilinearity of the scalar-product we have
$$
\|(\bm \cdot)^{\bm\alpha}\|^2_{L^2(\mathrm{d}\sigma)}
=\Big(
\sum_{j=0}^{\lfloor|\bm\alpha|/2\rfloor}
h_{j,\bm\alpha}
,
\sum_{j'=0}^{\lfloor|\bm\alpha|/2\rfloor}
h_{j',\bm\alpha}
\Big)_{L^2(\mathrm{d}\sigma)}
=
\sum_{j,j'=0}^{\lfloor|\bm\alpha|/2\rfloor}
(
h_{j,\bm\alpha}
,
h_{j',\bm\alpha}
)_{L^2(\mathrm{d}\sigma)}
\,.
$$
By the orthogonality of the harmonic polynomials \cite[Prop. 5.9]{axlerHarmonicFunctionTheory2001} the scalar products on the right-hand side of the equation above vanish whenever $j\neq j'$
so that we have
$$
\|(\bm \cdot)^{\bm\alpha}\|_{L^2(\mathrm{d}\sigma)}^2
=
\sum_{j=0}^{\lfloor|\bm\alpha|/2\rfloor}\|h_{j,\bm\alpha}\|_{L^2(\mathrm{d}\sigma)}^2
\,.
$$
Since $|\bm w^{\bm\alpha}|\le |\bm w|^{|\bm\alpha|}=1$ and $\sigma(S^{d-1})=1$ we have $\|(\bm \cdot)^{\bm\alpha}\|_{L^2(\mathrm{d}\sigma)}^2\le 1$ so that we have shown that
$$
1 \ge 
\|(\bm \cdot)^{\bm\alpha}\|_{L^2(\mathrm{d}\sigma)}^2=
\sum_{j=0}^{\lfloor|\bm\alpha|/2\rfloor}\|h_{j,\bm\alpha}\|_{L^2(\mathrm{d}\sigma)}^2
\ge
\|h_{k,\bm\alpha}\|_{L^2(\mathrm{d}\sigma)}^2
$$
which concludes the proof.
\end{proof}

We can now derive the remainder estimate.
This is a generalization of the remainder estimate of the Crandall formula of the Epstein zeta function \cite[Th. 15]{buchheitComputationPropertiesEpstein2026}.

\begin{theorem}[Remainder estimate]
  \label{theorem:trunc}
Let $\Lambda=A\mathds Z^d$ with $A\in\mathds R^{d\times d}$ be regular and $\det(A)=1$, let $n\in\mathds N$ and let $\nu\in\mathds R$.
Denote by $\kappa(A)=\|A\|\|A^{-1}\|$ the condition number, where $\|\cdot\|$ denotes the spectral norm.
For $\varepsilon,r>0$ define
$$
\mathcal R^{(\varepsilon)}_{\Lambda,\nu,n}(r)=\kappa(A)^{d+n+1} v_n(\nu)
\big(
R^{(\varepsilon)}_{\nu,n}\big(r/\kappa(A)\big)
+
R^{(\varepsilon)}_{d-\nu+2n,n}\big(r/\kappa(A)\big)
\big)
\,,
$$
with the prefactor 
$$v_n(\nu) = \frac{(\lfloor n/2\rfloor+1)(3/2)^{n+d}\pi^{(\nu+d)/2}}{%
\Gamma(d/2+1)\min_{k=0,1,\ldots,\lfloor n/2\rfloor}|\Gamma(\nu/2-k)|
}\sqrt{\binom{d+n-1}{d-1}}
\,,
$$
and with the function
$$
R^{(\varepsilon)}_{\nu,n}(r)
=\frac {r^{d+n+1}}\varepsilon
\Big(
\frac{G_{d+n+1}(r-\varepsilon) - 
G_{\nu}(r-\varepsilon)}{d+n+1-\nu}\Big),
$$
where the limit $\nu\to d+n+1$ is well-defined.
For $\Lambda=A\mathds Z^d$, $\bm\alpha\in\mathds N^d$ and $\bm x,\bm y\in\mathds R^d$ denote by $\mathfrak R_{\Lambda,\nu,\bm\alpha}(r)$ the remainder obtained by truncation at $r>0$ in the harmonic-decomposition Crandall representation, \Cref{harmonic-decomposition-crandall-representation}, at
$\lambda =1$
\begin{align*}
Z_{\Lambda,\nu,\bm\alpha}(\bm x,\bm y)
&=
\sum_{k=0}^{\lfloor |\bm\alpha|/2\rfloor}
\frac{\pi^{\nu/2-k}}{\Gamma(\nu/2-k)}
\Bigg[
\sum_{\substack{
\bm w\in \Lambda-\bm x
\\
|\bm w|\le r
}}
h_{k,\bm\alpha}(\bm w)G_{\nu-2k}(\bm w)e^{-2\pi i\bm y\cdot (\bm w+\bm x)}\\
&+
\frac{(-1)^k}{i^{|\bm\alpha|}V_{\Lambda}}
\sum_{\substack{
\bm p\in \Lambda^{\ast}+\bm y
\\ 
|\bm p|\le r
}}
h_{k,\bm\alpha}(\bm p)G_{d-\nu+2(|\bm\alpha|-k)}
(\bm p)e^{-2\pi i\bm x\cdot \bm p}
\Bigg]
+
\mathfrak R_{\Lambda,\nu,\bm\alpha}(r)
\,.
\end{align*}
Then for any $\varepsilon>0$ and
$
r>\kappa(A)\big(\max\{\sqrt d,\sqrt{(n+d)/(2\pi)}\}+2\varepsilon\big)
$
and for all $\bm\alpha\in\mathds N^d$ with $|\bm\alpha|\le n$ it holds
$$
|\mathfrak R_{\Lambda,\nu,\bm\alpha}(r)|<
\mathcal R^{(\varepsilon)}_{\Lambda,\nu,n}(r)
$$
and, in particular, we have
$$
|\mathfrak R_{\Lambda,\nu,\bm\alpha}(\kappa(A)r)|
< \kappa(A)^{d+n+1}
\mathcal R^{(\varepsilon)}_{\mathds Z^{d},\nu,n}(r)
\,.
$$
\end{theorem}

\begin{proof}[Proof of \Cref{theorem:trunc}]
By the upper bound on the homogeneous harmonic polynomials, \Cref{harmonic-upper-bound},
we find that
\begin{align*}
|\mathfrak R_{\Lambda,\nu,\bm\alpha}(r)|\le
\sum_{k=0}^{\lfloor |\bm\alpha|/2\rfloor}
&\frac{\pi^{\nu/2-k}\binom{d+|\bm\alpha|-2k-1}{d-1}^{1/2}}{|\Gamma(\nu/2-k)|}
\Bigg[
\sum_{\substack{
\bm z\in \Lambda
\\
|\bm z-\bm x|>r
}}
|\bm z-\bm x|^{|\bm\alpha|-2k}G_{\nu-2k}(\bm z-\bm x)
\\
&+
\sum_{\substack{
\bm p\in \Lambda^{\ast}
\\
|\bm p+\bm y|>r
}}
|\bm p+\bm y|^{|\bm\alpha|-2k}G_{d-\nu+2(|\bm\alpha|-k)}
(\bm p+\bm y)
\Bigg]\,.
\end{align*}
Since both the upper Crandall function and the polynomial weights are monotonically increasing in the exponent due to $r>1$ and \Cref{crandall-real-prop} we have that every factor beyond the Gamma function is monotonically increasing in $|\bm\alpha|$ and decreasing in $k\in\mathds N$, so that,
since enlarging the summation from $\lfloor|\bm\alpha|/2\rfloor$ to $\lfloor n/2\rfloor$ also increases the value, we have
$$
|\mathfrak R_{\Lambda,\nu,\bm\alpha}(r)|\le
\tilde C_{n}(\nu)\Bigg[
\sum_{\substack{
\bm z\in \Lambda
\\
|\bm z-\bm x|>r
}}
|\bm z-\bm x|^{n}G_{\nu}(\bm z-\bm x)
+
\sum_{\substack{
\bm p\in \Lambda^{\ast}
\\
|\bm p+\bm y|>r
}}
|\bm p+\bm y|^{n}G_{d-\nu+2n}
(\bm p+\bm y)
\Bigg]
$$
with
$
\tilde C_{n}(\nu)=(\lfloor n/2\rfloor+1)\pi^{\nu/2}\binom{d+n-1}{d-1}^{1/2}/(\min_{k\in\{0,1,\ldots,\lfloor n/2\rfloor\}}|\Gamma(\nu/2-k)|)
$.
By \Cref{lattice-to-zd-bound} it is enough to consider the square lattice, that is we obtain
\begin{align*}
|\mathfrak R_{\Lambda,\nu,\bm\alpha}(r)|\le
&\tilde C_{n}(\nu)\Bigg[
\|A\|^{n}
\!\!\!\!\!\!\!\!
\sum_{\substack{
\bm w\in\mathds Z^d
\\ |\bm w - A^{-1}\bm x|> r/{\|A\|}
}}
\!\!\!\!\!\!\!\!
|\bm w-A^{-1}\bm x|^{n}
G_{\nu}
\left(\frac{\bm w -A^{-1}\bm x}{\| A^{-1}\|} \right)
\\
&\qquad\qquad
+
\|A^{-1}\|^{n}
\!\!\!\!\!\!\!\!
\sum_{\substack{
\bm w \in\mathds Z^d
\\ 
|\bm w + A^T\bm y|> r/{\|A^{-1}\|}
}}
\!\!\!\!\!\!\!\!
|\bm w +A^T\bm y|^{n}
G_{d-\nu+2n}
\left(\frac{\bm w +A^T\bm y}{\| A\|} \right)
\Bigg]
\,,
\end{align*}
where we enlarged the summation range using $A^{-1}(\mathds R^d\setminus B_r)\subseteq \mathds R^d\setminus B_{r/\Vert A\Vert}$ and
$A^{T}(\mathds R^d\setminus B_r)\subseteq \mathds R^d\setminus B_{r/\Vert A^{-T}\Vert}$ with $\Vert A^{-T}\Vert=\Vert A^{-1}\Vert$
where $B_r\subseteq\mathds R^d$ is the open ball around the origin with radius $r>0$.
By \Cref{crandall-weight-lattice-sum-bound}
we obtain that the real space sum including all prefactors
$$
\tilde C_{n}(\nu)\,\|A\|^{n}
\!\!\!\!\!\!
\sum_{\substack{
\bm w\in\mathds Z^d
\\ |\bm w - A^{-1}\bm x|> r/\|A\|
}}
\!\!\!\!\!\!
|\bm w-A^{-1}\bm x|^{n}
G_{\nu}
\left(\frac{\bm w -A^{-1}\bm x}{\| A^{-1}\|} \right)
$$
is
bounded by
$$
\tilde C_{n}(\nu)\,
\|A\|^{n}
\frac {C_{n+d}}{\tilde\varepsilon}
\Big(\frac{r}{\|A\|}\Big)^{d+n+1}
\frac{%
\Big(
G_{d+n+1}\big(\frac{1}{\|A^{-1}\|}(\frac r{\|A\|}-\tilde\varepsilon)\big) - G_\nu\big(\frac 1{\|A^{-1}\|}(\frac r{\|A\|}-\tilde\varepsilon)\big)\Big)
}{%
d+n+1-\nu
}
\,,
$$
for every $\tilde\varepsilon >0$ that satisfies the conditions in \Cref{crandall-weight-lattice-sum-bound}, that is
$\tilde\varepsilon \le (r/\|A\|-\sqrt{d})/2$ 
and $(r/\|A\|-\tilde\varepsilon)/\|A^{-1}\|>\sqrt{(d+n)/(2\pi)}$,
where $C_{n+d}$ is defined as the constant $C_t$ in \Cref{crandall-weight-lattice-sum-bound}.
Letting $\tilde\varepsilon=\|A^{-1}\|\varepsilon$
for which 
$
(r/\|A\|-\tilde\varepsilon)/\|A^{-1}\|=r/\kappa(A)-\varepsilon
$,
so that both conditions are implied by the assumptions on $r$,
and using $v_n(\nu)=\tilde C_{n}(\nu)\,C_{n+d}$
and $\|A\|^n\le \kappa(A)^{d+n+1}/\|A^{-1}\|^{d+n+1}$ we obtain
that the equation above is bounded by
$$
\kappa(A)^{d+n+1}
\frac {v_{n}(\nu)}{\varepsilon}
\Big(\frac{r}{\kappa(A)}\Big)^{d+n+1}
\frac{G_{d+n+1}\big(r/\kappa(A)-\varepsilon\big) - G_\nu\big(r/\kappa(A)-\varepsilon\big)}{d+n+1-\nu}
\,.
$$
By an analogous treatment for the sum in reciprocal space, and since $\kappa(A)=\kappa(A^{-T})$ we find that
$$
\tilde C_{n}(\nu)\,\|A^{-1}\|^{n}
\!\!\!\!\!\!
\sum_{\substack{
\bm w\in\mathds Z^d
\\ |\bm w + A^{T}\bm y|> r/\|A^{-1}\|
}}
\!\!\!\!\!\!
|\bm w+ A^T\bm y|^{n}
G_{d-\nu+2n}
\left(\frac{\bm w +A^{T}\bm y}{\| A\|} \right)
$$
is bounded by
$$
\kappa(A)^{d+n+1}
\frac {v_{n}(\nu)}{\varepsilon}
\Big(\frac{r}{\kappa(A)}\Big)^{d+n+1}
\frac{G_{d+n+1}\big(r/\kappa(A)-\varepsilon\big) - G_{d-\nu+2n}\big(r/\kappa(A)-\varepsilon\big)}{d+n+1-(d-\nu+2n)}
\,,
$$
so that the representation for $\mathcal R^{(\varepsilon)}_{\Lambda,\nu,n}$ in terms of $R^{(\varepsilon)}_{\nu,n}$ is obtained from reassembling the upper bound on $|\mathfrak R_{\Lambda,\nu,\bm\alpha}(r)|$.
When substituting $r$ with $r\kappa(A)$, the bound on the remainder depends on the lattice only through a single global power of the condition number
$$
\mathcal R^{(\varepsilon)}_{\Lambda,\nu,n}(r\kappa(A))
=\kappa(A)^{d+n+1} \mathcal R^{(\varepsilon)}_{\mathds Z^d,\nu,n}(r)
\,,
$$
so that the last part of the statement follows.
\end{proof}

\begin{table}
\centering
\begin{tabular}{c | *{10}{c}}
\toprule
 $n$ & $d=1$ & $d=2$ & $d=3$ & $d=4$ & $d=5$ & $d=6$ & $d=7$ & $d=8$ & $d=9$ & $d=10$  \\
\midrule

$0$ & $3.8$ & $3.9$ & $4.0$ & $4.1$ & $4.2$ & $4.2$ & $4.3$ & $4.4$ & $4.4$ & $4.5$ \\ 

$10$ & $4.8$ & $4.9$ & $5.0$ & $5.1$ & $5.2$ & $5.3$ & $5.4$ & $5.5$ & $5.5$ & $5.6$ \\ 

$20$ & $5.9$ & $6.0$ & $6.1$ & $6.2$ & $6.3$ & $6.4$ & $6.5$ & $6.6$ & $6.6$ & $6.7$ \\ 

$50$ & $8.8$ & $8.9$ & $9.0$ & $9.1$ & $9.1$ & $9.2$ & $9.3$ & $9.4$ & $9.4$ & $9.5$ \\ 

\bottomrule
\end{tabular}
\caption{Truncation values $r_0$ to achieve an absolute truncation error 
$\sup_{|\nu|\le 10}|\mathfrak R_{\mathds Z^{d},\nu,\bm\alpha}(r_0)|<10^{-18}$ of the harmonic-decomposition Crandall representation, \Cref{harmonic-decomposition-crandall-representation}, for all $\bm\alpha\in\mathds N^d$ with $|\bm\alpha|\le n$
where the bound was
obtained by 
 \Cref{theorem:trunc} for the choice of $\varepsilon = 1/50$.}
\label{tab:square-lattice-trunc-values}
\end{table}

Choosing $\varepsilon=1/50$ eliminates the free parameter of the upper bound of the absolute value of the remainder in \Cref{theorem:trunc}.
In \Cref{tab:square-lattice-trunc-values}
we present truncation values $r=r_0$ for achieving absolute precision $<10^{-18}$ for the square lattice $\Lambda=\mathds Z^d$ and up to anisotropy order fifty, $|\bm\alpha|\le n$ with $n\in\{0,10,20,50\}$, as a function of dimension $d$.
An upper bound of the error due to the remainder for general lattices $\Lambda=A\mathds Z^d$, $\det(A)=1$ 
is then obtained by the choice
$$
r= \kappa(A) r_0,
$$
where the absolute value of the remainder lies below machine precision
if $\kappa(A)^{d+n+1}\le 10^2$
for the values $r_0$ in  \Cref{tab:square-lattice-trunc-values}.

\section{Identities for numerical experiments}
\label{sec:appendix-num}

Here we provide the proofs for the identities employed in the numerical experiments in \Cref{sec:num}.
We begin by showing that the anisotropic Epstein zeta function corresponds to its regularized counterpart for vanishing wave vectors.

\begin{proof}[Proof of \Cref{aniso-equals-reg-in-zero}]
Notice that the
derivative Crandall representation, \Cref{derivative-crandall-representation},
and its regularized counterpart, 
\Cref{low-ders-aniso-reg},
are the same, except for a single 
summand $\bm p=\bm 0$ in the reciprocal sum which reads
$G^{(\bm\alpha)}_{d-\nu}(\bm 0)$
for the non-regularized function
and $G^{{\rm reg},(\bm\alpha)}_{d-\nu,1}(\bm 0)$
for the regularized function,
where we choose to set the free splitting-parameter $\lambda$ to one.
Let
$\Re{\nu}>d+|\bm\alpha|$ and let $\bm y\in\mathds R^d\setminus\{\bm 0\}$.
We show that this implies 
that
$G^{{\rm reg},(\bm\alpha)}_{d-\nu,1}(\bm y)
\to
G^{(\bm\alpha)}_{d-\nu}(\bm 0)
$
as $\bm y\to \bm 0$
which implies
$G^{{\rm reg},(\bm\alpha)}_{d-\nu,1}(\bm 0)
=
G^{(\bm\alpha)}_{d-\nu}(\bm 0)
$
since for the restricted $\nu$ both the upper and lower Crandall function continuously extend to $\bm y=\bm 0$, see \Cref{crandall-upper-prop} and \Cref{remark:hol-low-ders-aniso-reg}.
By  \Cref{hol}, at $\bm y\in\Lambda^*$ the anisotropic Epstein zeta function is meromorphic in $\nu\in\mathds C$ 
with a simple pole at $\nu=d+|\bm\alpha|$ if and only if $\chi(\bm\alpha)=1$
 and  by \Cref{hol-reg},
its regularized counterpart is holomorphic in $\nu\in \mathds C\setminus (d+2|\lceil \bm\alpha/2\rceil| +2\mathds N)$.
The identity theorem for holomorphic functions yields the statement for the smaller region $\nu\in \mathds C\setminus (d+2|\lceil \bm\alpha/2\rceil| +2\mathds N)$,
and the special values of $\nu$ not in this region are included in the discussion for $\Re\nu>d+|\bm\alpha|$ as follows.

Write
$$
G^{{\rm reg},(\bm\alpha)}_{d-\nu,1}(\bm y)
=
G^{(\bm\alpha)}_{d-\nu}(\bm y)-\pi^{-\nu/2}\Gamma(\nu/2)\hat{s}^{(\bm\alpha)}_{\nu}(\bm y)
$$
so that by the continuity of the upper Crandall function it is enough to show that the derivative of the singularity $\hat{s}^{(\bm\alpha)}_{\nu}(\bm y)$ vanishes as $\bm y\to\bm 0$.
If $\nu\notin d+2\mathds N$ by the Faà di Bruno formula, \Cref{derivative-fy2}, we have that
$$
\hat{s}_{\nu}^{(\bm\alpha)}(\bm y)
=
\sum_{2\bm\beta\le\bm\alpha}
C_{\nu,|\bm\alpha|,|\bm\beta|}
w_{\bm\alpha,\bm\beta}(\bm y)
(\bm y^2)^{(\nu-d)/2-(|\bm\alpha|-|\bm\beta|)}
$$
with coefficients $C_{\nu,|\bm\alpha|,|\bm\beta|}\in\mathds C$ independent of $\bm y$.
Since  $w_{\bm\alpha,\bm\beta}(\bm y)$ is a homogeneous polynomial of order $|\bm\alpha|-2|\bm\beta|$ 
we have that 
$|\hat{s}^{(\bm\alpha)}_{\nu}(\bm y)|$
is bounded by a homogeneous polynomial of order 
$\Re{\nu}-d-|\bm\alpha|>0$
so we have $\hat{s}^{(\bm\alpha)}_{\nu}(\bm y)\to  0$ as $\bm y\to\bm 0$.

Let $\nu=d+2 \ell$ for $\ell\in\mathds N$.
Then, by writing $(\pi \bm y^2)^{\ell}\log(\pi \bm y^2)$
as a derivative of a continuous exponent 
$
\partial_{t}[(\pi \bm y^2)^{t}]_{t=\ell}
$
and commuting $\partial_t$ with $\diffop^{\bm\alpha}$ we have by the  Faà di Bruno formula,
up to terms with $\log(\pi\bm y^2)$ replaced by one, which vanish likewise, that
$$
\hat{s}^{(\bm\alpha)}_{d+2\ell}(\bm y)
=
\sum_{2\bm\beta\le\bm\alpha}
C_{\ell,|\bm\alpha|,|\bm\beta|}
w_{\bm\alpha,\bm\beta}(\bm y)
\partial_t\big[(\pi\bm y^2)^{t-(|\bm\alpha|-|\bm\beta|)}
\big]_{t=\ell}
$$
with coefficients $C_{\ell,|\bm\alpha|,|\bm\beta|}\in\mathds C$ independent of $\bm y$.
Using that $w_{\bm\alpha,\bm\beta}(\bm y)$ is homogeneous of degree $|\bm\alpha|-2|\bm\beta|$
we have
$$
|\hat{s}^{(\bm\alpha)}_{d+2\ell}(\bm y)|
\le
\sum_{2\bm\beta\le\bm\alpha}
|C_{\ell,|\bm\alpha|,|\bm\beta|}|
\,
\pi^{|\bm\beta|-|\bm\alpha|/2}
|w_{\bm\alpha,\bm\beta}(\bm 1)|
\,
\big|\partial_t\big[z^{t-|\bm\alpha|/2}
\big]_{t=\ell}\big|
$$
with $z=\pi\bm y^2$.
Note that
$\Re{\nu}>d+|\bm\alpha|$ implies $\ell>|\bm\alpha|/2$,
so that we can define $\varepsilon =(\ell-|\bm\alpha|/2)/4>0$,
and since $|\log z|<z^{-\varepsilon}/\varepsilon$, $0<z<1$ we have
$$
|\partial_t\big[
z^{t-|\bm\alpha|/2}
\big]_{t=\ell}|
=
z^{\ell-|\bm\alpha|/2}\,|\log z|
<
\frac 1\varepsilon
z^{\ell-|\bm\alpha|/2-\varepsilon}
\to 0
$$
as $z\to 0$.
Therefore $\hat s_{d+2\ell}^{(\bm\alpha)}(\bm y)\to 0$
as $\bm y\to \bm 0$, which concludes the proof.
\end{proof}

We now derive the identity regarding the Ramanujan tau $L$-function following a theta-Mellin approach similar to the approach in  \cite{zucker}.

\begin{proof}[Proof of \Cref{ramanujan}]
For $\Re{\nu}>16$, rewrite
$Z_{\mathds Z^8,\nu,\bm 1}(-\bm 1/2,\bm 1/2)$ in terms of its direct sum
$$
2^{\nu-8}
\sum_{\bm z\in\mathds Z^8}|2\bm z+\bm 1|^{-\nu}
\prod_{j=1}^8(-1)^{z^{(j)}}(2z^{(j)}+1)
$$
and by the Mellin transform
$
\Gamma(s)
x^{-s}
=
\int_0^{\infty} 
\eta^{s-1}e^{-\eta x}
\d \eta
$,
$\Re s,x>0$, see 
\cite[Ch. 6.3, Eq. 15]{erdelyiTablesIntegralTransforms1954},
applied at $s=\nu/2$
and $x=|2\bm z+\bm 1|^2$
we obtain
$$
\frac{2^{\nu-8}}{\Gamma(\nu/2)}
\int_0^{\infty}
\eta^{\nu/2}
\sum_{\bm z\in \mathds Z^8}
\prod_{j=1}^8(-1)^{z^{(j)}}(2z^{(j)}+1)e^{-\eta (2z^{(j)}+1)^2}
\frac{\mathrm{d}\eta}{\eta}
$$
where we swapped summation and integration, which is allowed by
\Cref{zetadirectsum} since $\Re{\nu}>d+|\bm\alpha|$.
Note that the summand at a point $\bm z\in\mathds N^8$ has the same value as the summand where a component of the index $z^{(j)}$ is swapped with $-1-z^{(j)}$ and that $z^{(j)}\to -1-z^{(j)}$ has no fixed point, so that we may replace the sum over $\mathds Z^8$ by $2^8$-times the sum over $\mathds N^8$ and by the distributive law we have
$$
\frac{2^{\nu}}{\Gamma(\nu/2)}
\int_0^{\infty}
\eta^{\nu/2}
\Big(
\sum_{k=0}^{\infty}
(-1)^{k}(2k+1)e^{-(2k+1)^2\eta}
\Big)^8
\frac{\mathrm{d}\eta}{\eta}
\,.
$$
Denote by $\theta_1(z,q)$, $z,q\in\mathds C$ the theta function \cite[Eq. 20.2.1]{NIST:DLMF}
$$
\theta_1(z,q) 
= 
2\sum_{k=0}^{\infty}(-1)^k q^{(2k+1)^2/4}\sin((2k+1)z)\,.
$$
Then the series in the integrand of the representation for the anisotropic Epstein zeta function is the first derivative $\theta_1'(0,q)$ of the theta function with respect to $z$ at $z=0$ and at $q=e^{-4\eta}$,
that is we have
$$
Z_{\mathds Z^8,\nu,\bm 1}(-\bm 1/2,\bm 1/2)
=
\frac{2^{\nu}}{\Gamma(\nu/2)}
\int_0^{\infty}
\eta^{\nu/2}
\Big(\frac12
\theta_1'(0,e^{-4\eta})\Big)^8
\frac{\mathrm{d}\eta}{\eta}\,.
$$
By the identity for the derivative of the theta function and the definition of the Ramanujan tau function for $q\in(0,1)$ we have
$$
\Big(\frac12\theta_1'(0,q)\Big)^8
=
q^{2}\prod_{k=1}^{\infty}(1-(q^2)^{k})^{24}
=
\sum_{k=1}^{\infty}\tau(k)q^{2k}
$$
where we used  \cite[Eq. 20.4.2]{NIST:DLMF} and  \cite[Eq. 27.14.18]{NIST:DLMF} respectively.
Using this identity at $q=e^{-4\eta}$ we have
$$
Z_{\mathds Z^8,\nu,\bm 1}(-\bm 1/2,\bm 1/2)
=
2^{\nu}
\sum_{k=1}^{\infty}
\frac{\tau(k)}{\Gamma(\nu/2)}
\int_0^{\infty}
\eta^{\nu/2}
e^{-8k\eta}
\frac{\mathrm{d}\eta}{\eta}\,.
$$
where we changed the order of summation and integration which is justified by Fubini since
by \cite[Ex. p. 137]{apostolModularFunctionsDirichlet1990} we have that $|\tau(k)|\le C k^6$ for some $C>0$ 
which implies that
$
\sum_{k=1}^{\infty}|\tau(k)| k^{-\Re \nu/2}<\infty
$
for
$\Re\nu>14$.
Notice that by the Mellin transform identity already used in this proof, the integral including the $1/\Gamma(\nu/2)$ prefactor is equal to $(8k)^{-\nu/2}$ and that $2^\nu 8^{-\nu/2}=2^{-\nu/2}$
so we obtain the identity of the statement.
Since the left-hand side is entire in $\nu$ by the holomorphicity of the anisotropic Epstein zeta function, \Cref{hol}, and the right hand side is entire, see \cite[Th. 6.20]{apostolModularFunctionsDirichlet1990} applied to the modular form with weight twelve, Fourier coefficients $\tau(k)$, and vanishing constant term, they agree by the identity theorem for holomorphic functions.
\end{proof}

To derive the identity for the singular Gaussian integral, we first define the Fourier transform.

\begin{definition}[Fourier transform]
Let $f :\mathds{R}^d \rightarrow \mathds{C}$ be integrable. We define its Fourier transform $\mathcal{F} f = \hat{f}$ via
\[
(\mathcal{F}f)(\bm k)
=\int_{\mathds{R}^d} f(\bm x)e^{-2\pi i\bm k\cdot\bm x}\,\mathrm{d}\bm x
,\qquad \bm k \in\mathds R^d.
\]
\end{definition}

The identity for the singular Gaussian integral in three dimensions then follows.

\begin{proof}[Proof of \Cref{gauss-int-3d}]
Let $0<\Re{\nu}<3$.
By the superexponential decay of the Gaussian $g_{\sigma}(\bm y)$ for $|\bm y|\to \infty$ and the integrability of
$s_{\nu}(\bm y-\bm x)$ near $\bm x$ for $\Re{\nu}<3$, the left-hand side integral
$I(\bm x)=\int_{\mathds R^3} f_{\bm x}(\bm y)\, \d \bm y$
is well defined.
Due to the translationally invariant integration region by substituting $\bm y$ with $\bm y+\bm x$ we can write
$I(\bm x)=\int_{\mathds R^3}s_\nu(\bm y)g_{\sigma}(\bm y+\bm x)\d \bm y$.
Since $\varphi(\bm k) = e^{-2\pi i\bm k\cdot \bm x}\hat{g}_{\sigma}(\bm k)$
is a Schwartz function with Fourier transform $\hat{\varphi}(\bm y)=g_{\sigma}(\bm y+\bm x)$,
where we used that $g_{\sigma}$ is even,
we have $I(\bm x)=\int_{\mathds R^3}s_\nu(\bm y)\hat{\varphi}(\bm y)\d \bm y =\int_{\mathds R^3}\hat{s}_{\nu}(\bm k)\varphi(\bm k)\d \bm k$
with the distributional Fourier transform $\hat{s}_{\nu}$ as in \Cref{def-zeta-aniso-reg}.
Since
the Gaussian $g_1$ is a fixed point of the Fourier transform
and therefore by the scaling property $\widehat{g_{\sigma}}
=\sigma^3g_{1/\sigma}$,
we obtain the representation in Fourier space
$$
I(\bm x)
=
\sigma^3
\int_{\mathds R^3}
g_{1/\sigma}(\bm k)
\hat{s}_{\nu}(\bm k)
e^{-2\pi i \bm x\cdot\bm k}
\d \bm k
$$
where the integral is well defined due to the superexponential decay of the Gaussian and since, for $\Re{\nu}>0$ we have that $\hat{s}_{\nu}(\bm k)$
is integrable near $\bm k=\bm 0$.
By introducing spherical coordinates
$\bm k = r\bm k'$
for 
$\bm k'=(\sin\varphi\cos\psi,\sin\varphi\sin\psi,\cos\varphi)^T$
we obtain
$$
I(\bm x)=
C_{\nu}
\int_{0}^{\infty}
g_{1/\sigma}(r)r^{\nu-1}
\int_0^{2\pi}
\int_{0}^{\pi}
e^{-2\pi ir \bm x\cdot\bm k'}
\sin\varphi
\d \varphi 
\d \psi
\d r
$$
where $C_{\nu}=\sigma^3\pi^{\nu-3/2}\Gamma((3-\nu)/2)/\Gamma(\nu/2)$.
As we integrate over the whole sphere, we may rotate $\bm x$ to $(0,0,|\bm x|)^T$ without changing the value of the integral
so that the inner integral is independent of $\psi$
and evaluates as
$$
\int_0^\pi e^{-i(2 \pi r |\bm x|)\cos\varphi}\sin\varphi\d\varphi 
=\frac{\sin(2\pi r |\bm x|)}{\pi r |\bm x|}
$$
by \cite[Eq. (10.1.14) and Eq. (10.1.11)]{abramowitz}.
We finally apply \cite[Eq. 3.953.7]{gradshteinTableIntegralsSeries2015}
to the resulting integral
$$
I(\bm x)=
C_{\nu}\frac{2}{|\bm x|}
\int_{0}^{\infty}
r^{\nu-2}g_{1/\sigma}(r)
\sin(2\pi r |\bm x|)
\d r
$$
and together with Kummers transformation
$
e^{-z}\,_1\!\!\:F_1(a;b;z)=
\,_1\!\!\:F_1(b-a;b;-z)
$
\cite[Sec. 6.3, Eq. (7)]{Batemann1953b}
obtain
$$
I(\bm x)
=
2\;\!\pi^{(2-\nu)/2}\sigma^{-\nu}C_{\nu}
\Gamma\Big(\frac{\nu}{2}\Big)
\,_1\!\!\:F_1\Big(\frac \nu2;\frac 32;-\pi \frac{\bm x^2}{\sigma^2}\Big)
$$
for $\Re{\nu}>0$.
On the one hand, we have that the left-hand side $I(\bm x)$ of the statement is holomorphic in $\nu\in\mathds C$, $\Re{\nu}<3$.
On the other hand, for
 $b\notin -\mathds N$ we have that
$_1\!\!\:F_1(a;b;z)$
is an entire function of both $a\in\mathds C$ and $z\in\mathds C$
\cite[Sec. 6.7.1]{Batemann1953a},
and $C_{\nu}\Gamma(\nu/2)$ is holomorphic in $\nu\in\mathds C\setminus(3+2\mathds N)$,
so that the right-hand side of the statement is holomorphic in $\nu\in\mathds C\setminus(3+2\mathds N)$.
Since both sides agree on the strip $0<\Re{\nu}<3$,
by the identity theorem for holomorphic functions we have that the right-hand side is the holomorphic continuation of the left hand side from $\Re{\nu}<3$ to $\nu \in\mathds C\setminus(3+2\mathds N)$.
\end{proof}

\printbibliography

\end{document}